\RequirePackage{fix-cm}
\documentclass[smallextended]{svjour3}       
\smartqed  
\usepackage{latexsym}
\usepackage{amsmath,amsfonts,amssymb}
\usepackage{amsopn}
\usepackage{graphicx}
\usepackage{epstopdf}
\usepackage{hyperref}
\usepackage{mathtools}
\usepackage{caption}
\usepackage{subcaption}
\usepackage{diagbox}

\RequirePackage{booktabs}
\RequirePackage{dsfont}
\RequirePackage{microtype}

\usepackage[algo2e,ruled,linesnumbered]{algorithm2e}
\SetKw{KwTerminate}{terminate}
\SetKw{KwInofupalg}{Input:}
\SetKw{Kwgiven}{Fixed:}

\DeclareMathOperator{\divg}{div}

\newcommand{\Frechet}{\text{Fr{\'e}chet }}
\newcommand{\Gateaux}{\text{G{\^a}teaux }}
\newcommand{\DOF}{\mathrm{DOF}}
\newcommand{\gmres}{\textsc{gmres }}

\newcommand{\fin}{\mathrm{final}}
\newcommand{\calA}{ \ensuremath{{\cal{A}}} }
\newcommand{\calE}{ \ensuremath{{\cal{E}}} }
\newcommand{\calP}{ \ensuremath{{\cal{P}}} }

\newcommand{\N}{\mathds{N}}
\newcommand{\R}{\mathds{R}}
\newcommand{\Rpos}{\ensuremath{\R_{>0}}}

\newcommand{\hatU}{\ensuremath{ \Hone }}
\newcommand{\VhO}{\ensuremath { Y_h } }
\newcommand{\VhOast}{\ensuremath { Y_h^\ast } }
\newcommand{\Vh}{\ensuremath { V_h } }
\newcommand{\VO}{\ensuremath { Y } }
\newcommand{\BV}{\ensuremath {\text{BV}(\Omega)} }
\newcommand{\M}{\ensuremath {{\cal M}(\Omega)} }
\newcommand{\LLL}{\ensuremath {L^2(\Omega)} }  
\newcommand{\Linf}{\ensuremath {L^\infty(\Omega)} }
\newcommand{\Hone}{\ensuremath {H^1(\Omega)} }
\newcommand{\Lin}{\ensuremath {{\cal L}} }
\newcommand{\Ball}{\ensuremath { \mathds{B} } }
\newcommand{\Y}{\VO}

\newcommand{\st}{\ensuremath{\text{s.t.}}}
\newcommand{\Ci}{\ensuremath{ \operatorname{Ci} } }

\newcommand{\dx}{\ensuremath {\,\mathrm{d}x}}
\newcommand{\dtau}{\ensuremath {\,\mathrm{d}\tau}}

\newcommand{\dt}{\ensuremath {\,\mathrm{d}t}}
\newcommand{\dBVr}{\ensuremath { d_{\operatorname{BV},r} }}

\newcommand{\Norm}[2][]{\ensuremath{\left\|#2\right\|_{#1}}}
\newcommand{\norm}[2][]{\ensuremath{\lVert #2\rVert_{#1}}}

\newcommand{\param}{\ensuremath {{\gamma,\delta}} }
\newcommand{\paramh}{\ensuremath {{\gamma,\delta,h}} }
\newcommand{\paramn}{\ensuremath {{\gamma,\delta,{h_n}}} }
\newcommand{\paramnhat}{\ensuremath {{\gamma,\delta,h_{\hat n}}} }

\newcommand{\parambr}{\ensuremath {{(\gamma,\delta)}} }
\newcommand{\paramk}{\ensuremath {{\gamma_k,\delta_k}} }
\newcommand{\parambrk}{\ensuremath {{(\gamma_k,\delta_k)}} }

\newtheorem{assumption}[theorem]{Assumption}

\begin{document}

\title{A path-following inexact Newton method for PDE-constrained optimal control in BV (extended version)
	\thanks{\textit{This is an extended version of the corresponding journal article \cite{HafeMan_reg}. It contains some proofs that are omitted in the journal's  version.}}
}

\titlerunning{A path-following inexact Newton method for optimal control in BV}        

\author{D. Hafemeyer \and
        F. Mannel}



\institute{Dominik Hafemeyer \at
			TU M\"unchen\\
			Lehrstuhl f\"ur Optimalsteuerung, Department of Mathematics\\
			Boltzmannstr.~3, 85748 Garching b. M\"unchen, Germany\\
			\email{dominik.hafemeyer@ma.tum.de}
	\and
	Florian Mannel \at
	University of Graz\\
 	Institute of Mathematics and Scientific Computing\\
	Heinrichstraße 36, 8010 Graz, Austria\\
	\email{florian.mannel@uni-graz.at}
}
	
\date{Received: date / Accepted: date}

\maketitle

\begin{abstract}
We study a PDE-constrained optimal control problem that involves functions of bounded variation as controls and includes the TV seminorm of the control in the objective.
We apply a path-following inexact Newton method to the problems that arise from smoothing the TV seminorm and adding an $H^1$ regularization.
We prove in an infinite-dimensional setting that, first, the solutions of these auxiliary problems converge to the solution of the original problem and, second, that an inexact Newton method enjoys fast local convergence when applied to a reformulation of the auxiliary optimality systems in which the control appears as implicit function of the adjoint state. 
We show convergence of a Finite Element approximation, provide a globalized preconditioned inexact Newton method as solver for the discretized auxiliary problems, and embed it into an inexact path-following scheme. 
We construct a two-dimensional test problem with fully explicit solution and present numerical results to illustrate the accuracy and robustness of the approach.

\keywords{optimal control \and partial differential equations \and TV seminorm \and functions of bounded variation \and path-following Newton method}
\subclass{35J70 \and 49-04 \and 49M05 \and 49M15 \and 49M25 \and 49J20 \and 49K20 \and 49N60}
\end{abstract}

\section*{Problem setting and introduction}
\addcontentsline{toc}{section}{Problem setting and introduction}

This work is concerned with the optimal control problem
\begin{equation} \tag{OC} \label{eq:ocintro}
\min_{(y,u)\in H_0^1(\Omega)\times \BV} \;  \underbrace{\frac{1}{2}\Norm[L^2(\Omega)]{y-y_\Omega}^2 + \beta |u|_{\BV}}_{=:J(y,u)}
\qquad\st\qquad
Ay = u,
\end{equation}
where throughout $\Omega\subset \R^N$ is a bounded $C^{1,1}$ domain and  $N\in\{1,2,3\}$.
The control $u$ belongs to the space of functions of bounded variation $\BV$, the state $y$ lives in $Y:=H^1_0(\Omega)$, the parameter $\beta$ is positive, and $Ay=u$ is a partial differential equation of the form
\begin{equation*}
\left\{
\begin{aligned}
	\mathcal{A} y + c_0 y & =  u && \text{ in }\Omega,\\
	y & = 0 &&\text{ on }\partial\Omega
\end{aligned}
\right.
\end{equation*}
with a non-negative function $c_0\in L^\infty(\Omega)$ and a linear and uniformly elliptic operator of second order in divergence form $\calA: H^1_0(\Omega) \rightarrow H^{-1}(\Omega)$,
$\calA y (\varphi) = \int_{\Omega} \sum_{i,j=1}^N a_{ij}\partial_{i} y\partial_{j} \varphi \dx$ whose coefficients satisfy $a_{ij}=a_{ji}\in C^{0,1}(\Omega)$ for all $i,j\in\{1,\ldots,N\}$.
The specific feature of \eqref{eq:ocintro} is the appearance of the BV seminorm $|u|_{\BV}$ in the cost functional, which favors piecewise constant controls and has recently attracted considerable interest in PDE-constrained optimal control, cf. \cite{BergouniouxBonnefondHaberkornPrivat,CasasKogutLeugering,kruse,Casas2019,Clason2011,EngelKunischBVWaveSemismooth,EngelVexlerTrautmann,Hafemeyer15,HafemeyerMaster,Hafemeyer2019}
and the earlier works \cite{CasasKunischPola,CasasKunischPola2}.
The majority of these contributions focuses on deriving optimality conditions and studying Finite Element approximations. In contrast, the main focus of this work is on a path-following method. Specifically, 
\begin{itemize}
	\item we propose to smooth the TV seminorm in $J$ and add an $H^1$ regularization, and show in an infinite-dimensional setting that the solutions of the resulting family of auxiliary problems converge to the solution of \eqref{eq:ocintro};
	\item for any given auxiliary problem we prove that an infinite-dimensional inexact Newton method converges locally;
	\item we derive a practical path-following method that yields accurate solutions for \eqref{eq:ocintro} and illustrate its capabilities in numerical examples for $\Omega\subset\R^2$. 
\end{itemize}  
To the best of our knowledge, these aspects have only been investigated partially for optimal control problems that involve the TV seminorm in the objective. 
In particular, there are few works that address the numerical solution when the measure $\nabla u$ is supported in a \emph{two-dimensional} set. In fact, we are only aware of \cite{Clason2011}, where a doubly-regularized version of the Fenchel predual of \eqref{eq:ocintro} is solved
for fixed regularization parameters, but path-following is not applied. We stress that in our numerical experience the two-dimensional case is significantly more challenging than the one-dimensional case.
A FeNiCs implementation of our path-following method is available at \url{https://arxiv.org/abs/2010.11628}. 
It includes all the features that we discuss in section~\ref{sec:numericalsolution}, e.g.,
a preconditioner for the Newton systems, a non-monotone line search globalization,
and inexact path-following. 

A further contribution of this work is that 
\begin{itemize}
	\item we provide an example of \eqref{eq:ocintro} for $N=2$ with fully explicit solution.
\end{itemize}
For the case that $\nabla u$ is defined in an interval ($N=1$) such examples are available, e.g. \cite{kruse,Hafemeyer2019}, but for $N=2$ this is new.

Let us briefly address three difficulties associated with \eqref{eq:ocintro}.

First, the fact that \eqref{eq:ocintro} is posed in the non-reflexive space $\BV$ complicates the proof of existence of optimal solutions. By now it is, however, well understood how to deal with this issue also in more complicated situations, cf. e.g. \cite{kruse,Casas2019}.

Second, we notice that $u\mapsto\lvert u \rvert_{\BV}$ is not differentiable. 
We will cope with this by replacing $\lvert u \rvert_{\BV}$ with the smooth functional $\psi_\delta(u)=\int_\Omega \sqrt{\lvert\nabla u\rvert^2+\delta^2}$, $\delta\geq 0$, that satisfies $\psi_0(\cdot)=\lvert \cdot \rvert_{\BV}$.
The functional $\psi_\delta$ is well-known, particularly in the imaging community, e.g. \cite{acar,chan}. However, in most of the existing works the smoothing parameter $\delta>0$ is fixed, whereas we are interested in driving $\delta$ to zero. 
We will also add the regularizer $\gamma\lVert u\rVert_{H^1(\Omega)}^2$, $\gamma\geq 0$, to $J$ and drive $\gamma$ to zero.
This allows us to prove that for $\param>0$ the optimal control $\bar u_{\param}$ of the smoothed and regularized auxiliary problem is $C^{1,\alpha}$, 
which is crucial to show, for instance, that the adjoint-to-control mapping is differentiable; cf. Theorem~\ref{thm_PtoUfrechet}. In contrast, for $\gamma=0$ only $\bar u_{0,\delta}\in\BV$ can be expected.

Third, our numerical experience with PDE-constrained optimal control involving the TV seminorm \cite{kruse,Clason2018,Hafemeyer15,HafemeyerMaster,Hafemeyer2019} suggests that path-following Newton methods work significantly better if the optimality systems of the auxiliary problems do not contain the control as independent variable. 
Therefore, we express the auxiliary optimality conditions in terms of state and adjoint state by regarding the control as an implicit function of the adjoint state.  

Let us set our work in perspective with the available literature. 
As one of the main contributions we show that the solutions of the auxiliary problems 
converge to the solution of \eqref{eq:ocintro}, cf. section~\ref{sec:regconvergence}. 
The asymptotic convergence for vanishing $H^1$ seminorm regularization is analyzed in \cite[Section~6]{Casas2019} for a more general problem than \eqref{eq:ocintro}, but the fact that our setting is less general allows us to prove convergence in stronger norms than the corresponding \cite[Theorem~10]{Casas2019}. 
The asymptotic convergence for a doubly-regularized version of the predual of \eqref{eq:ocintro} is established in \cite[Appendix~A]{Clason2011}, but one of the regularizations is left untouched, so convergence is towards the solution of a regularized problem, not towards the solution of \eqref{eq:ocintro}. 
Next, we demonstrate local convergence of an infinite-dimensional inexact Newton method applied to the optimality system of the auxiliary problem. 
Because the control and the adjoint state are coupled by a quasilinear PDE, this convergence analysis is non-trivial; among others, it relies on 
Hölder estimates for the gradient of the control that are rather technical to derive. 
A related result is \cite[Theorem~3.5]{Clason2011}, where local q-superlinear convergence of a semismooth Newton method is shown for the doubly-regularized Fenchel predual for fixed regularization parameters.
Yet, since we work with a different optimality system, the overlap is small. 
Nonetheless, \cite{Clason2011} is closely related to the present paper, and it would be interesting to embed the semismooth Newton method 
\cite[Algorithm~2]{Clason2011} in a path-following scheme and compare it to our algorithm.
The concept to view the control as an implicit function of the adjoint state or to eliminate it, is well-known in optimal control, cf., e.g., \cite{kruse,Clason2011,HintermKunisch,HintermStadler,Hinze05,HinzeTroe,NeitzelPruefertSlawig,PieperDiss,Schiela_IPMefficient,WeiserGaenzlerSchiela}.

Turning to the discrete level we provide a Finite Element approximation and demonstrate that the Finite Element solutions of the auxiliary problems converge to their non-discrete counterparts.
Finite Element approximations for optimal control in BV involving the TV seminorm have also been studied in \cite{BergouniouxBonnefondHaberkornPrivat,CasasKogutLeugering,kruse,Casas2019,EngelKunischBVWaveSemismooth,EngelVexlerTrautmann,Hafemeyer15,HafemeyerMaster,Hafemeyer2019}, but in our assessment the regularization of \eqref{eq:ocintro} that we propose is not covered by these studies. 
The papers \cite{ElvetunNielsen,PreconditioningTVRegularization} study the linear systems that arise in split Bregman methods when applied to a discretization of \eqref{eq:ocintro} with homogeneous Neumann boundary conditions.

The BV-term in \eqref{eq:ocintro} favors sparsity in the gradient of the control. Other sparsity promoting control terms that have recently been studied are measure norms and $L^1$--type functionals, 
e.g., \cite{AllendesFuicaOtarola,CCK2012,CCK2013,CasasKunisch2014,CasasRyllTroeltzsch,CasasVexlerZuazua,HSW,LiStadler,PieperDiss,Stadler}.

TV-regularization is also of significant importance in imaging problems and its usefulness for, e.g., noise removal has long been known \cite{Rudin1992}. 
However, we take the point of view that imaging problems are different in nature from optimal control problems, for instance because their forward operator is usually cheap to evaluate and non-compact.

This paper is organized as follows. After preliminaries in section~\ref{sec:prelim}, 
we consider existence, optimality conditions and convergence of solutions in section~\ref{sec:origandregulproblems}.
In section~\ref{sec_regularity} we study differentiability of the adjoint-to-control mapping, which paves the way 
for proving local convergence of an inexact Newton method in section~\ref{sec:newton}. 
Section~\ref{sec:FEapproximation} addresses the Finite Element approximation and its convergence, while section~\ref{sec:numericalsolution} provides the 
path-following method.
Numerical experiments are presented in section~\ref{sec:numericmainchapter},  including for the test problem with explicit solution. 
In section~\ref{sec_sum} we summarize. 
Several technical results such as H\"older continuity of solutions to quasilinear PDEs are deferred to the appendix.

\section{Preliminaries}\label{sec:prelim}

We recall facts about the space $\BV$, introduce the smoothed BV seminorm that we use in this work, and collect properties of the solution operator associated to the PDE in \eqref{eq:ocintro}.

\subsection{Functions of bounded variation} 

The following statements about $\BV$ can be found in \cite[Chapter~3]{ambrosio} unless stated otherwise. 
For $u\in L^1(\Omega)$ we let
\begin{equation*}
	|u|_{\BV} := \sup_{v\in C^1_0(\Omega)^N, \lVert \lvert v\rvert \rVert_\infty \leq 1} \int_\Omega u \operatorname{div} v\dx, 
\end{equation*}
where here and throughout, $|\cdot|$ denotes the Euclidean norm. 
The space of functions of bounded variation is defined as
\begin{equation*}
\BV := \Bigl\{u\in L^1(\Omega): \; \lvert u \rvert_{\BV} < \infty
\Bigr\},
\end{equation*}
and $\lvert u\rvert_{\BV}$ is called the BV seminorm (also TV seminorm) of $u\in\BV$. 
We endow $\BV$ with the norm $\lVert \cdot \rVert_{\BV} := \lVert \cdot \rVert_{L^1(\Omega)} + |\cdot|_{\BV}$
and recall from \cite[Thm.~10.1.1]{Attouch} that this makes $\BV$ a Banach space. 
It can be shown that $u\in\BV$ iff there exists a vector measure $(\partial_{x_1} u, \dots, \partial_{x_N} u )^T = \nabla u \in\M^N$ such that for all $i\in\{1,\ldots,n\}$ there holds
\begin{equation*}
\int_\Omega \partial_{x_i} u v\dx = - \int_\Omega u \partial_{x_i} v\dx \qquad \forall v\in C_0^\infty(\Omega),
\end{equation*}
where $\M$ denotes the linear space of regular Borel measures, e.g. \cite[Chapter~2]{Rudin1987}. 
Also, for $u\in\BV$ we have $|u|_{\BV}=\lVert \lvert\nabla u\rvert \rVert_{\M}$, i.e., $|u|_{\BV}$ is the total variation of the vector-measure $\nabla u$.
The space $\BV$ embeds continuously (compactly) into $L^r(\Omega)$ for $r\in[1,\frac{N}{N-1}]$ ($r \in [1, \frac{N}{N-1})$), see, e.g., \cite[Cor.~3.49 and Prop.~3.21]{ambrosio}. We use the convention that $\frac{N}{N-1}=\infty$ for $N=1$. 
Also important is the notion of strict convergence, e.g. \cite{ambrosio,Attouch}.
\begin{definition} \label{def:strictconv}
	For $r \in [1, \frac{N}{N-1} ]$ the metric $\dBVr$ is given by
	\begin{align*}
	\dBVr\colon  & \BV \times \BV \rightarrow \R,\\
	& (u,v) \mapsto \lVert u-v \rVert_{L^r(\Omega)} + \left| |u|_{\BV} - |v|_{\BV} \right|.
	\end{align*}
	Convergence with respect to $d_{\operatorname{BV},1}$ is called \emph{strict convergence}. 
\end{definition}

\begin{remark}
	The embedding $\BV \hookrightarrow L^r(\Omega)$, for $r \in [1, \frac{N}{N-1} ]$, implies that $d_{BV, r}$ is well-defined and continuous with respect to $\lVert \cdot \rVert_{\BV}$.
\end{remark}	

We will also use the following density property. 
\begin{lemma}  \label{lem:bvdensity}
	$C^\infty(\bar \Omega)$ is dense in $(\BV\cap L^r(\Omega),\,\dBVr)$ for $r \in [1, \frac{N}{N-1} ]$.
\end{lemma}

\begin{proof}
	By straightforward modifications the proof for the special case $r=1$ in \cite[Thm.~10.1.2]{Attouch} can be extended, 
	using that the sequence of mollifiers constructed in the proof converges in $L^r$, see \cite[Prop.~2.2.4]{Attouch}. \qed
\end{proof}

For the remainder of this work we fix a number $s=s(N) \in (1,\frac{N}{N-1})$ with
\begin{center}
	\fbox{$\BV\hookrightarrow\hookrightarrow L^s(\Omega)\hookrightarrow H^{-1}(\Omega)$,}
\end{center}
where the first embedding is compact and the second is continuous. For $N=1$ we interpret $\frac{N}{N-1}$ as $\infty$.

\begin{remark}
	Consider, for instance, $N=2$ and any $r\in (1,2)$. Then we have  $\BV\hookrightarrow\hookrightarrow L^r(\Omega)$ and $H^1(\Omega)\hookrightarrow L^\frac{r}{r-1}(\Omega)$ so that any $s \in(1,2)$ can be used.
\end{remark}

\subsection{The smoothed BV seminorm}

We will replace the BV seminorm in \eqref{eq:ocintro} by the function $\psi_\delta \colon \BV \rightarrow \R$,
\begin{equation*}
\psi_\delta(u) := \sup \, \Biggl\lbrace \int_\Omega u \operatorname{div} v + \sqrt{ \delta ( 1- |v|^2 ) } \dx : \; v \in C_0^1(\Omega)^N, \,  \lVert\lvert v \rvert \rVert_{L^\infty(\Omega)} \leq 1 \Biggr\rbrace,
\end{equation*}
where $\delta\geq 0$. We stress that $\psi_\delta$ is frequently employed in imaging problems for the same purpose, for instance in \cite{acar,chan}. It has the following properties. 
\begin{lemma} \label{lem:psidelta}
	The following statements are true for all $\delta\geq 0$.
	\begin{enumerate}
		\item For any $u\in \BV$ there holds 
		\begin{equation*}
		|u|_{\BV} = \psi_0 (u) \leq \psi_\delta(u) \leq |u|_{\BV} + \sqrt{\delta} |\Omega|.
		\end{equation*}
		\item $\psi_\delta$ is lower semi-continuous with respect to the $L^1(\Omega)$ norm. 
		\item $\psi_\delta$ is convex.
		\item For all $u\in W^{1,1}(\Omega)$ we have \label{thm:psideltaitem4}
		\begin{equation*}
		\psi_\delta (u) = \int_\Omega \sqrt{ \delta + |\nabla u|^2 } \dx.
		\end{equation*}
		\item The function $\psi_\delta |_{H^1(\Omega)}$ is Lipschitz with respect to $\lVert\cdot\rVert_{\Hone}$. 
	\end{enumerate}
\end{lemma}

\begin{proof}
The first four statements are from \cite[Section~2]{acar} and the last one follows 
from $H^1(\Omega)\hookrightarrow W^{1,1}(\Omega)$, 4. and the Lipschitz continuity of $r\mapsto \sqrt{\delta+r^2}$. \qed
\end{proof}

\subsection{The solution operator of the state equation}

\begin{lemma} \label{lem:solutionoperator}
	For every $u\in H^{-1}(\Omega)$ the operator equation $Ay=u$ in 
	\eqref{eq:ocintro} has a unique solution $y=y(u)\in\Y$. The solution operator
	\begin{equation*}
	S \colon H^{-1}(\Omega) \rightarrow \Y, \quad
	u \mapsto y(u) 
	\end{equation*}
	is linear, continuous, and bijective. In particular, $S$ is $L^s$-$L^2$ continuous.
	Moreover, for given $q\in (1,\infty)$ there is a constant $C>0$ such that
	\begin{equation*}
	\lVert Su \rVert_{W^{2,q}(\Omega)} \leq C \lVert u \rVert_{L^q(\Omega)}
	\end{equation*}
	is satisfied for all $u\in L^q(\Omega)$. 
\end{lemma}

\begin{proof}
	Except for the estimate all statements follow from the Lax-Milgram theorem. 
	The estimate is a consequence of \cite[Lemma 2.4.2.1, Theorem 2.4.2.5]{grisvard}.
	\qed 
\end{proof}

\begin{remark}\label{rem_feasiblesetocnonempty}
	From $\BV\hookrightarrow L^s(\Omega)\hookrightarrow H^{-1}(\Omega)$ and Lemma~\ref{lem:solutionoperator} we obtain that \eqref{eq:ocintro} has a nonempty feasible set. 
\end{remark}

\section{The solutions of original and regularized problems} \label{sec:origandregulproblems}

In this section we prove the existence of solutions for \eqref{eq:ocintro} and the associated regularized problems,
characterize the solutions by optimality conditions, and show their convergence in appropriate function spaces.

\subsection{The original problem: Existence of solutions and optimality conditions} \label{sec:reducedproblem}

To establish the existence of a solution for \eqref{eq:ocintro} we use the \emph{reduced problem} 
\begin{equation*} \tag{ROC} \label{eq:redProblem}
	\min_{u\in\BV} \; \underbrace{\frac{1}{2} \lVert Su-y_\Omega \rVert^2_{L^2(\Omega)} + \beta \lvert u \rvert_{\BV}}_{=:j(u)}.
\end{equation*}

\begin{lemma} \label{lem:uregjconvex}\label{lem:uregjcont}
	The function $j:\BV\rightarrow\R$ is well-defined, strictly convex, continuous with respect to $d_{BV,s}$, and coercive with respect to 
	$\norm[\BV]{\cdot}$.
\end{lemma}

\begin{proof}
	The term $\frac{1}{2} \lVert Su-y_\Omega \rVert^2_{L^2(\Omega)}$ is well-defined by Remark~\ref{rem_feasiblesetocnonempty} and strictly convex in $u$ due to the injectivity of $S$. 
	Since $\lvert \cdot \rvert_{\BV}$ is convex, the strict convexity of $j$ follows.
	The continuity holds because $S$ is $L^s$-$L^2$ continuous.
	The coercivity follows by virtue of \cite[Lemma~4.1]{acar} 
	using again that $S$ is injective and $L^s$-$L^2$ continuous.
	\qed
\end{proof}

The strict convexity implies that $j$ has at most one (local=global) minimizer. 

\begin{theorem}\label{thm_originalproblemhasuniqueglobalsolution}
	The problem \eqref{eq:redProblem} has a unique solution $\bar u\in\BV$.
\end{theorem}

\begin{proof}
	The proof is included in the proof of Theorem~\ref{thm:solutions}. \qed
\end{proof}

As usual, the \emph{optimal state} $\bar y$ and the \emph{optimal adjoint state} $\bar p$ are given by
	\begin{equation*}
		\bar y:=S\bar u\in\Y\cap W^{2,r_N}(\Omega) \qquad\text{ and }\qquad \bar p := S^\ast(\bar y - y_\Omega),
	\end{equation*}	
	where, due to $\BV\hookrightarrow L^{\frac{N}{N-1}}(\Omega)$ and Lemma~\ref{lem:solutionoperator}, 
	we have $r_N=\frac{N}{N-1}$ for $N\in\{2,3\}$, respectively, $r_N \geq 1$ arbitrarily large for $N=1$. Moreover, $S^\ast\colon H^{-1}(\Omega) \rightarrow \Y$ is the dual operator of $S$, where we have identified 
	the dual space of $H^{-1}(\Omega)$ with $\Y$ using reflexivity. Since $S^\ast=S$ and $\bar y-y_\Omega\in L^2(\Omega)$, Lemma~\ref{lem:solutionoperator} yields $\bar p\in P$ for
	\begin{equation*}
		P:=H^2(\Omega) \cap H_0^1(\Omega).
	\end{equation*}
It is standard to show that $\bar p$ is the unique weak solution of
\begin{equation*}
\left\{
\begin{aligned}
\mathcal{A} p + c_0 p &=  \bar y - y_\Omega && \text{ in }\Omega,\\
p &= 0 && \text{ on }\partial\Omega.
\end{aligned}\right.
\end{equation*}

\begin{remark}
	The optimality conditions of \eqref{eq:redProblem} are only needed for the construction of the test problem with explicit solution in appendix~\ref{sec_examplewithexplicitsolution}, but not for the following analysis. They are deferred to appendix~\ref{sec_optcondorigprob}. We stress, however, that they allow to discuss the sparsity of $\nabla \bar u$, cf. Remark~\ref{rem_sparsityg}. 
\end{remark}

\subsection{The regularized problems: Existence of solutions and optimality conditions} \label{sec:regprob}

Smoothing the BV seminorm and adding an $H^1$ regularization to $j$ yields 
\begin{equation*} \tag{\mbox{ROC$_{\gamma,\delta}$}} \label{eq:regProblem}
	\min_{u\in\BV} \; \underbrace{\frac{1}{2} \Norm[\LLL]{Su-y_\Omega}^2 + \beta \psi_\delta(u) + \frac{\gamma}{2}\Norm[H^1(\Omega)]{u}^2}_{=:j_\param(u)},
\end{equation*}
where we set $j_\param(u):=+\infty$ for $u\in\BV\setminus\Hone$ if $\gamma>0$.

\begin{lemma} \label{thm:jconvex}
	For any $\param\geq 0$ the function $j_\param:\BV\rightarrow\R\cup\{+\infty\}$ is well-defined, strictly convex, and coercive with respect to $\norm[\BV]{\cdot}$.
	Moreover, the function $j_\param\vert_{\Hone}$ is $\Hone$ continuous for any $\param\geq 0$.
\end{lemma}

\begin{proof}
	The well-definition and strict convexity of $j_\param$ follow similarly as for $j$ in Lemma~\ref{lem:uregjconvex}. 
	Using Lemma~\ref{lem:psidelta}~1. we find $j\leq j_\param$, so $j_\param$ inherits the coercivity from $j$.	
	The continuity follows term by term. 
	For the first term it is enough to recall from Lemma~\ref{lem:solutionoperator} the $L^2$-$L^2$ continuity of $S$. 
	The second term is Lipschitz in $H^1$ by Lemma~\ref{lem:psidelta}. The continuity of the third term is obvious. \qed 
\end{proof}

We obtain existence of unique and global solutions. 

\begin{theorem} \label{thm:solutions}
	For any $\param \geq 0$, \eqref{eq:regProblem} has a unique solution $\bar u_\param \in \BV$. 
	For $\gamma>0$ we have $\bar u_\param\in H^1(\Omega)$.
\end{theorem}

\begin{proof}
	Since $j_\param$ is strictly convex by Lemma~\ref{thm:jconvex}, there is at most one minimizer.
	For $\gamma>0$ the existence of $\bar u_{\param}\in \Hone$ follows from standard arguments since $j_\param\vert_{H^1(\Omega)}$ is strongly convex 
	and $H^1$ continuous by Lemma~\ref{thm:jconvex}. 
	For $\gamma=0$ and any $\delta\geq 0$, the existence of a minimizer follows from \cite[Theorem~4.1]{acar} by use of the injectivity and boundedness of $S:H^{-1}(\Omega)\rightarrow Y\hookrightarrow L^2(\Omega)$. While $\Omega$ is convex in \cite{acar}, \cite[Theorem~4.1]{acar} remains true without this assumption.
	\qed 
\end{proof}

\emph{Optimal state} $\bar y_\param$ and \emph{optimal adjoint state} $\bar p_\param$ 
for \eqref{eq:regProblem} are given by 
\begin{equation*}
	\bar y_\param := S \bar u_\param\in \Y\cap W^{2,r_N}(\Omega) \qquad\text{ and }\qquad
	\bar p_\param := S^\ast \bigl( \bar y_\param - y_\Omega \bigr)\in P,
\end{equation*}
where $r_N=\frac{N}{N-1}$ for $N\in\{2,3\}$, respectively, $r_N\geq 1$ arbitrarily large for $N=1$.
In particular, $\bar p_\param$ is the unique weak solution of
\begin{equation*}
\left\{
\begin{aligned}
\mathcal{A} p + c_0 p &=  \bar y_\param - y_\Omega & &\text{ in }\Omega,\\
p &= 0 & &\text{ on }\partial\Omega.
\end{aligned}
\right.
\end{equation*}

The optimality conditions of \eqref{eq:regProblem} are based on differentiability of $j_\param$.

\begin{lemma} \label{lem:jderivatives}
	For $\param>0$ the functional $j_{\param}:\hatU\rightarrow\R$ is 
	\Frechet differentiable. 
	Its first derivative is given by 
	\begin{equation*}
	j^\prime_\param(u)v = \left( S^\ast(Su-y_\Omega), v \right)_{L^2(\Omega)} + \beta \psi_\delta^\prime(u)v + \gamma (u, v)_{H^1(\Omega)} \qquad\forall v\in \hatU,
	\end{equation*}
	where 
	\begin{equation*}
	\psi_\delta^\prime(u) v = \int_\Omega \frac{(\nabla u, \nabla v)}{\sqrt{ \delta + |\nabla u|^2} } \dx \qquad\forall v\in \hatU.
	\end{equation*}
\end{lemma}

\begin{proof}
	It suffices to argue for $\psi_\delta$, which we do in Lemma~\ref{thm:psiderivativeNEU} in appendix~\ref{sec_differentiabilityofpsidelta}. The other terms are standard. \qed 
\end{proof}

For differentiable convex functions a vanishing derivative is necessary and sufficient for a global minimizer. 
This yields the following optimality conditions.
\begin{theorem} \label{thm:optcond}
	For $\param > 0$ the control $\bar u_\param\in \hatU$ is the solution of \eqref{eq:regProblem} iff
	\begin{equation*}
	j_\param^\prime(\bar u_\param)v = 0 \qquad \forall v \in \hatU,
	\end{equation*}
	which is the nonlinear Neumann problem
	\begin{equation} \label{eq:optpde}
	\gamma (\bar u_\param, v)_{H^1(\Omega)} + \beta \int_\Omega \frac{\left( \nabla \bar u_\param, \nabla v \right) }{ \sqrt{ \delta + | \nabla \bar u_\param |^2 } }\dx = -(\bar p_\param, v)_{L^2(\Omega)} \qquad \forall v\in \hatU,
	\end{equation}
	where $\bar p_\param = S^\ast ( S \bar u_\param - y_\Omega)$.
\end{theorem}

\subsection{Convergence of the path of solutions} \label{sec:regconvergence}

We prove that $(\bar u_\param$, $\bar y_\param$,$\bar p_\param)$ converges to 
$(\bar u,\bar y,\bar p)$ for $\param\to 0$. As a first step we show convergence of the objective values.

\begin{lemma} \label{lem:vcont00}
	We have
	\begin{equation*}
	j_\param(\bar u_\param) \xrightarrow{ \R_{\geq 0}^2 \ni \parambr\rightarrow (0,0)} j(\bar u).
	\end{equation*}
\end{lemma}

\begin{proof}
	Let $\epsilon>0$ and let $(\parambrk)_{k\in\N} \subset \R_{\geq 0}^2$ converge to $(0,0)$. There holds
	\begin{equation*}
	0 \leq j_\paramk(\bar u_\paramk) - j(\bar u) = \bigl[ j_\paramk(\bar u_\paramk) - j_{\gamma_k,0}(\bar u_{\gamma_k,0}) \bigr] + \bigl[ j_{\gamma_k,0}(\bar u_{\gamma_k,0}) - j(\bar u) \bigr],
	\end{equation*}
	where we used $j(\bar u) \leq j(\bar u_\paramk) \leq j_\paramk(\bar u_\paramk)$. The first term in brackets satisfies
	\begin{align*}
	j_\paramk(\bar u_\paramk) - j_{\gamma_k, 0}(\bar u_{\gamma_k,0}) & \leq j_{\gamma_k, \delta_k}(\bar u_{\gamma_k, 0}) - j_{\gamma_k, 0}(\bar u_{\gamma_k,0})\\
	& = \beta \psi_{\delta_k}(\bar u_{\gamma_k, 0}) - \beta |\bar u_{\gamma_k, 0}|_{\BV} \leq \beta \sqrt{\delta_k} |\Omega|,
	\end{align*}
	where the last inequality follows from Lemma~\ref{lem:psidelta}.
	For the second term in brackets we deduce from Lemma~\ref{lem:bvdensity} and the $d_{BV,s}$ continuity of $j$ established in Lemma~\ref{lem:uregjcont} that there is $u_\epsilon \in C^\infty(\bar\Omega)$ such that $|j(\bar u)-j(u_\epsilon)| < \epsilon$. This yields 
	\begin{equation*}
		\begin{split}
	j_{\gamma_k,0}(\bar u_{\gamma_k, 0}) - j(\bar u) 
	& \leq j_{\gamma_k,0}(u_\epsilon) - j(\bar u)\\
	& = j(u_\epsilon) + \frac{\gamma_k}{2} ||u_\epsilon||^2_{H^1(\Omega)} - j(\bar u) \leq \epsilon + \frac{\gamma_k}{2} ||u_\epsilon||^2_{H^1(\Omega)}.
		\end{split}
	\end{equation*}
	Putting the estimates for the two terms together shows
	\begin{equation*}
	|j_\paramk(\bar u_\paramk) - j(\bar u)| \leq \beta \sqrt{\delta_k} |\Omega| + \epsilon + \frac{\gamma_k}{2} \lVert u_\epsilon \rVert^2_{H^1(\Omega)}.
	\end{equation*}
	For $k\rightarrow\infty$ this implies the claim since $\epsilon>0$ was arbitrary and
	\begin{equation*}
	0 \leq \liminf_{k\rightarrow\infty} |j_\paramk(\bar u_\paramk) - j(\bar u)| \leq \limsup_{k\rightarrow\infty} |j_\paramk(\bar u_\paramk) - j(\bar u)|  \leq \epsilon.
	\end{equation*}
	\qed
\end{proof}

We infer that the optimal controls $\bar u_{\param}$ converge to $\bar u$ in $L^r$ for suitable $r$.

\begin{lemma}\label{lem_Lsconvofregulcontrols}
	For any	$r\in[1,\frac{N}{N-1})$ we have $|| \bar u_\param - \bar u||_{L^r(\Omega)} \xrightarrow{\parambr \rightarrow (0,0)} 0$.
\end{lemma}

\begin{proof}
	Let $(\parambrk)_{k\in\N} \subset \R_{\geq 0}^2$ converge to $(0,0)$. Let $C > 0$ be so large that $\gamma_k,\delta_k\leq C$ for all $k$. The optimality of $\bar u_\paramk$ and Lemma~\ref{lem:psidelta} yield for each $k\in\N$ 
	\begin{equation*}
	j(\bar u_\paramk) \leq j_\paramk(\bar u_\paramk) \leq j_\paramk(0) \leq j_{C,C}(0).
	\end{equation*}
	As $(j(\bar u_\paramk))_{k\in\N}$ is bounded, we obtain from Lemma~\ref{lem:uregjcont} that $(\norm[\BV]{\bar u_\paramk})_{k\in\N}$ is bounded, too. 
	The compact embedding of $\BV$ into $L^{r}(\Omega)$, $r\in[1,\frac{N}{N-1})$ thus implies that there exists $\tilde u\in L^r(\Omega)$ such that a subsequence of $\left( \bar u_\paramk \right)_{k\in\N}$, denoted in the same way, converges to $\tilde u$ in $L^r(\Omega)$. It is therefore enough to show $\tilde u=\bar u$.
	Since $j$ is lower semi-continuous in the $L^s$ topology, cf. Lemma~\ref{lem:psidelta} and continuity and convexity of the other terms,
	we have
	\begin{equation*}
	j(\tilde u) \leq \liminf_{k\rightarrow\infty} j(\bar u_\paramk) \leq \liminf_{k\rightarrow\infty} j_\paramk(\bar u_\paramk)
	= j(\bar u),
	\end{equation*}
	where we used Lemma~\ref{lem:vcont00} to derive the equality. 
	This shows $\tilde u\in\BV$, hence Theorem~\ref{thm_originalproblemhasuniqueglobalsolution} implies $\tilde u = \bar u$. \qed 
\end{proof}

In fact, the convergence of $\bar u_{\param}$ to $\bar u$ is stronger.

\begin{theorem}\label{thm_convbaru}
	For any	$r\in[1,\frac{N}{N-1})$ we have $d_{BV,r}(\bar u_{\gamma,\delta}, \bar u) \xrightarrow{\parambr\rightarrow (0,0)} 0.$
\end{theorem}

\begin{proof}
	For any $\param\geq 0$ we have $j(\bar u) \leq j(\bar u_\param) \leq j_\param(\bar u_\param)$, so Lemma~\ref{lem:vcont00} yields $\lim_{\parambr\rightarrow (0,0)} j(\bar u_\param) = j(\bar u)$. Furthermore, there holds
	\begin{equation*}
	\begin{split}
	\beta \Bigl| |\bar u|_{\BV} - |\bar u_\param|_{\BV} \Bigr|  & \leq \bigl| j(\bar u) - j(\bar u_\param) \bigr|\\ 
	& \quad + \frac12 \left| \lVert S\bar u-y_\Omega \rVert^2_{L^2(\Omega)} - \lVert S\bar u_\param - y_\Omega \rVert^2_{L^2(\Omega)} \right|.
	\end{split}
	\end{equation*}
	By Lemma~\ref{lem_Lsconvofregulcontrols} and the continuity of $S$ from $L^s(\Omega)$ to $L^2(\Omega)$ we thus find
	\begin{equation*}
	|\bar u_\param|_{\BV} \xrightarrow{\parambr\rightarrow (0,0)} |\bar u|_{\BV}.
	\end{equation*}
	Together with Lemma~\ref{lem_Lsconvofregulcontrols} this proves the claim. \qed 
\end{proof}

We conclude this section with the convergence of $(\bar y_{\param}, \bar p_\param)$ to $(\bar y,\bar p)$.

\begin{theorem} \label{thm:barybarpconv}
	For any	$r\in [1,\frac{N}{N-1})$ and any $r^\prime\in[1,\infty)$ we have
	\begin{equation*}
		\lim_{\parambr\rightarrow (0,0)} \lVert \bar y_\param - \bar y \rVert_{W^{2,r}(\Omega)} = \lim_{\parambr\rightarrow (0,0)} \lVert \bar p_\param - \bar p \rVert_{W^{2,r^\prime}(\Omega)} = 0.
	\end{equation*}
\end{theorem}

\begin{proof}
	The continuity of $S$ from $L^q$ to $W^{2,q}$ for any $q\in(1,\infty)$, 
	see Lemma~\ref{lem:solutionoperator}, 
	implies with Lemma~\ref{lem_Lsconvofregulcontrols} that $\lim_{\parambr\rightarrow (0,0)} \lVert \bar y_\param - \bar y \rVert_{W^{2,r}(\Omega)} = 0$ for any $r\in[1,\frac{N}{N-1})$.
	Since for any $r^\prime\in(1,\infty)$ there is $r\in [1,\frac{N}{N-1})$ such that $W^{2,r}(\Omega)\hookrightarrow L^{r^\prime}(\Omega)$ is satisfied, we can use the $L^{r^\prime}$-$W^{2,r^\prime}$ continuity of $S^\ast = S$ to find
	$\lim_{\parambr\rightarrow (0,0)} \lVert \bar p_\param - \bar p \rVert_{W^{2,r^\prime}(\Omega)}=\lim_{\parambr\rightarrow (0,0)} \lVert S^\ast(\bar y_\param-\bar y) \rVert_{W^{2,r^\prime}(\Omega)}=0$.	
	\qed 
\end{proof}

\begin{remark}
	The results of section~\ref{sec:origandregulproblems} can also be established for nonsmooth domains $\Omega$, but 
	$\bar y,\bar p,\bar y_\param,\bar p_\param$ may be less regular since $S$ may not provide the regularity stated in Lemma~\ref{lem:solutionoperator}. 
	A careful inspection reveals that only Theorem~\ref{thm:barybarpconv} has to be modified. 
	If, for instance, $\Omega\subset\R^N$, $N\in\{2,3\}$, is 
	a bounded Lipschitz domain, then \cite[Theorem~3]{Savare98} implies
	that Theorem~\ref{thm:barybarpconv} holds if $W^{2,r}$ and $W^{2,r^\prime}$ are both replaced by $H^r$, where $r\in[1,\frac32)$ is arbitrary.
	If $\Omega$ is convex, then \cite[Theorem~3.2.1.2]{grisvard} further yields that $W^{2,r^\prime}$ can be replaced by $H^2$.
\end{remark}

\section{Differentiability of the adjoint-to-control mapping}\label{sec_regularity}

The main goal of this section is to show that the PDE
\begin{equation} \label{eq:quasilinearpde}
\left\{
\begin{aligned}
-\divg\left(\left[\gamma+\frac{\beta}{\sqrt{\delta+|\nabla u|^2}}\right]\nabla u\right)+\gamma u & = p &&\text{in }\Omega,\\
\left( \left[ \gamma + \frac{\beta}{\sqrt{\delta + |\nabla u|^2}} \right] \nabla u, \nu \right) & = 0 &&\text{on } \partial\Omega
\end{aligned}
\right.
\end{equation}
has a unique weak solution $u=u(p)\in C^{1,\alpha}(\Omega)$ 
for every right-hand side $p\in L^\infty(\Omega)$, and that $p\mapsto u(p)$ is Lipschitz continuously \Frechet differentiable in any open ball, where the Lipschitz constant is independent of $\gamma$ and $\delta$ provided $\gamma>0$ and $\delta>0$ are bounded away from zero. This is accomplished in Theorem~\ref{thm:ullfd}. 
Note that we suppress the dependency on $\param$ in $u=u(p;\param)$.

\begin{assumption}\label{ass6}
	We are given constants $0 < \gamma_0 \leq \gamma^0$, $0 < \delta_0 \leq \delta^0$ and $b^0 > 0$. 
	We denote $I:=[\gamma_0,\gamma^0]\times [\delta_0, \delta^0]$ and write
	$\Ball\subset\Linf$ for the open ball of radius $b^0>0$ centered at the origin in $\Linf$.
\end{assumption}

We first show that $p\mapsto u(p)$ is well-defined and satisfies a Lipschitz estimate.

\begin{lemma}\label{lem_Hoeldercontinuityofregmeancurvatureequation}
	Let Assumption~\ref{ass6} hold.
	Then there exist $L>0$ and $\alpha\in (0,1)$ such that for each $(\gamma,\delta)\in I$ and all $p_1, p_2 \in \Ball$ the PDE \eqref{eq:quasilinearpde} has 	
	unique weak solutions $u_1=u_1(p_1)\in C^{1,\alpha}(\Omega)$
	and $u_2 = u_2(p_2)\in C^{1,\alpha}(\Omega)$ that satisfy
	\begin{equation*}
	\Norm[C^{1,\alpha}(\Omega)]{u_1-u_2} \leq L \Norm[\Linf]{p_1-p_2}.
	\end{equation*}
	In particular, we have the stability estimate
	\begin{equation*}
	\Norm[C^{1,\alpha}(\Omega)]{u_1} \leq L\Norm[\Linf]{p_1}.
	\end{equation*}
\end{lemma}

\begin{proof}
	Unique existence and the first estimate are established in Theorem~\ref{thm:quasilinearls} in appendix~\ref{sec_HoeldercontinuityforquasilinPDEs}.
	The second estimate follows from the first for $p_2=0$. \qed 
\end{proof}

We introduce the function
\begin{equation*}
f:\R^N \rightarrow\R^N, \qquad f(v) := \beta \frac{v}{\sqrt{\delta+|v|^2}},
\end{equation*}
so that \eqref{eq:quasilinearpde} becomes  
\begin{equation} \label{eq:quasilinearpdewithf}
- \divg \Bigl( \gamma \nabla u + f(\nabla u) \Bigr) + \gamma u = p \quad\text{ in } H^1(\Omega)^\ast.
\end{equation}	

The following two results prove that the adjoint-to-control mapping is differentiable and has a locally Lipschitz continuous derivative whose Lipschitz constant is bounded on bounded sets uniformly in $I$.

\begin{theorem} \label{thm_PtoUfrechet}
	Let Assumption~\ref{ass6} hold and let $\alpha \in (0,1)$ be the constant from Lemma~\ref{lem_Hoeldercontinuityofregmeancurvatureequation}. 	
	For each $\parambr\in I$ the mapping $\Ball \ni p\mapsto u(p) \in C^{1,\alpha}(\Omega)$ is \Frechet differentiable. Its derivative $z=u'(p)d\in C^{1,\alpha}(\Omega)$ in direction $d \in L^\infty(\Omega)$ is the unique weak solution of the linear PDE
	\begin{equation} \label{eq:uprimedef}
	\left\{
	\begin{aligned}
	- \divg \Biggl( \Bigl[\gamma I+ f^{\prime}\bigl(\nabla u(p)\bigr)\Bigr] \nabla z \Biggr) + \gamma z &= d && \text{ in }\Omega, \\
	\Biggl( \Bigl[ \gamma I + f^\prime\bigl(\nabla u(p)\bigr) \Bigr] \nabla z, \nu \Biggr) & = 0 && \text{ on } \partial\Omega,
	\end{aligned}\right.
	\end{equation}
	and there exists $C>0$ such that for all $\parambr\in I$, all $p\in\Ball$, and all $d\in L^\infty(\Omega)$ we have 
	\begin{equation*}
		\Norm[C^{1,\alpha}(\Omega)]{z} \leq C\Norm[L^\infty(\Omega)]{d}.
	\end{equation*}
\end{theorem}

\begin{proof}
	Let $p\in\Ball$ and $d\in \Linf$ be such that $p+d \in\Ball$. 
	From Lemma~\ref{lem_Hoeldercontinuityofregmeancurvatureequation} we obtain
	$u(p) \in C^{1,\alpha}(\Omega)$ and $\norm[C^{1,\alpha}(\Omega)]{u(p)}\leq C\norm[L^\infty(\Omega)]{p}$, where $C$ is independent of $\gamma,\delta,p$. Combining this with Lemma~\ref{lem_hoeldercompositions} implies 
	\begin{equation} \label{eq:proof:hoeldercontinuitybeforeAestimate}
	f^\prime(\nabla u(p)) = \frac{I}{\sqrt{\delta + |\nabla u(p)|^2}} - \frac{\nabla u(p) \nabla u(p)^T}{\bigl(\delta+|\nabla u(p)|^2\bigr)^\frac{3}{2}} \in C^{0,\alpha}(\Omega, \R^{N\times N})
	\end{equation}
	and the estimate 
	$\norm[C^{0,\alpha}(\Omega)]{A}\leq a^0$ for $A:=\gamma I+f^\prime(\nabla u(p))$ with a constant $a^0$ that does not depend on $\gamma,\delta,p$.
	Since $f'(v)$, $v\in\R^N$, is the Hessian of the convex function $v\mapsto\sqrt{\delta+\lvert v\rvert^2}$, it is positive semi-definite.  
	It follows that Theorem~\ref{thm:linearregularity} is applicable. Thus, the PDE \eqref{eq:uprimedef} has a unique weak solution $z \in C^{1,\alpha}(\Omega)$ that satisfies the claimed estimate. 
	Concerning the \Frechet differentiability we obtain for $r := u(p+d)-u(p)-z \in C^{1,\alpha}(\Omega)$ 
	\begin{equation*}
		\begin{split}
			- & \divg \Biggl( \Bigl[\gamma I + f^\prime\bigl( u(p) \bigr)\Bigr] \nabla r \Biggr) + \gamma r \\
			& = - \divg \Bigl( \gamma\nabla u(p+d) \Bigr) + \gamma u(p+d) + \divg \Bigl( \gamma \nabla u(p) \Bigr) - \gamma u(p) \\
			& \enspace\;\; + \divg \Biggl( \Bigl[\gamma I + f^\prime\bigl(\nabla u(p)\bigr) \Bigr]\nabla z \Biggr) - \gamma z - \divg \Bigl( f^\prime\bigl(\nabla u(p)\bigr) w \Bigr) \\
			& = \divg \Bigl( f\bigl(\nabla u(p+d) \bigr) - f\bigl(\nabla u(p)\bigr) - f^\prime\bigl(\nabla u(p)\bigr) w \Bigr),
		\end{split}
	\end{equation*}
	where we set $w:=w(p,d):=\nabla u(p+d)-\nabla u(p)$. Theorem~\ref{thm:linearregularity} implies that there is $C>0$, independent of $d$,
	such that 
	\begin{equation*}
	\Norm[C^{1,\alpha}(\Omega)]{r} \leq C \Norm[C^{0,\alpha}(\Omega)]{f\bigl(\nabla u(p+d) \bigr) - f\bigl(\nabla u(p)\bigr) - f^\prime\bigl(u(p)\bigr) w}.
	\end{equation*}
	The expression in the norm on the right-hand side satisfies the following identity pointwise in $\Omega$
	\begin{equation*}
		\begin{split}
			f & \bigl( \nabla u(p+d) \bigr) - f\bigl(\nabla u(p)\bigr) - f^\prime\bigl(\nabla u(p)\bigr) w \\
			& = \left( \int_0^1 f^\prime\bigl( \nabla u(p) + t w \bigr) - f^\prime\bigl(\nabla u(p)\bigr) \dt \right) w \\
			& = \left( \int_0^1 \int_0^1 f^{\prime\prime}\bigl( \nabla u(p) + \tau  t w\bigr) \dtau \, t \dt\right) [w,w].
		\end{split}
	\end{equation*}
	Lemma~\ref{lem_hoeldercompositions} yields 
	\begin{equation*}
	\Norm[C^{1,\alpha}(\Omega)]{r}\leq C \int_0^1\int_0^1 \Norm[C^{0,\alpha}(\Omega)]{f^{\prime\prime}\bigl( \nabla u(p) + \tau t w \bigr)} \dtau \dt \, 
	\Norm[C^{1,\alpha}(\Omega)]{u(p+d)-u(p)}^2.
	\end{equation*}
	As $f \in C^3(\R^N,\R^N)$ with bounded derivatives we have that $f^{\prime\prime}$ is Lipschitz continuous and bounded. 
	We infer from Lemma~\ref{lem_hoeldercompositions} and Lemma~\ref{lem_Hoeldercontinuityofregmeancurvatureequation} that
	\begin{equation*}
	\Norm[C^{1,\alpha}(\Omega)]{r}\leq C \Norm[L^\infty(\Omega)]{d}^2,
	\end{equation*}
	which shows $\norm[C^{1,\alpha}(\Omega)]{r} = o(\norm[L^\infty(\Omega)]{d})$ since $C$ is independent of $d$. 
	\qed 
\end{proof}

\begin{theorem} \label{thm:ullfd}
	Let Assumption~\ref{ass6} hold and let $\alpha \in (0,1)$ be the constant from Lemma~\ref{lem_Hoeldercontinuityofregmeancurvatureequation}. 	
	Then the mapping $u^\prime: \Ball \rightarrow \Lin(L^\infty(\Omega), C^{1,\alpha}(\Omega))$ is Lipschitz continuous and the Lipschitz constant does not depend on $\parambr$, but only on $\Omega$, $N$, $\gamma_0$, $\gamma^0$, $\delta_0$, $\delta^0$ and $b^0$.
\end{theorem}

\begin{proof}
	Let $p, q \in \Ball$ and $d \in L^\infty(\Omega)$. Set
	$z_p:=\nabla \bigl(u'(p)d\bigr)$ and $z_q:=\nabla \bigl(u'(q)d\bigr)$. Then 
	\begin{equation*}
	- \divg \Bigl( \gamma \bigl[z_p - z_q \bigr] + f^\prime\bigl( \nabla u(p) \bigr) z_p - f^\prime\bigl(\nabla u(q)\bigr) z_q \Bigr) + \gamma \bigl[u^\prime(p)d - u^\prime(q) d\bigr] = 0
	\end{equation*}
	holds in $H^1(\Omega)^\ast$.
	Thus, the difference $r := u^\prime(p)d - u^\prime(q) d$ satisfies
	\begin{align*}
	- \divg \bigl( \gamma \nabla r \bigr) + \gamma r & = \divg \Bigl( f^\prime\bigl( \nabla u(p) \bigr) z_p - f^\prime\bigl(\nabla u(q)\bigr) z_q \Bigr) \\
	& = \divg \Bigl( f^\prime\bigl( \nabla u(p) \bigr) \nabla r \Bigl) + \divg \Bigl( \bigl[ f^\prime\bigl(\nabla u(p)\bigr) - f^\prime\bigl(\nabla u(q)\bigr) \bigr] z_q \Bigr),
	\end{align*}
	from which we infer that
	\begin{equation*}
	- \divg \Bigl( \bigl[\gamma I + f^\prime\bigl( \nabla u(p) \bigr) \bigr] \nabla r \Bigr) + \gamma r = \divg \Bigl( \bigl[ f^\prime\bigl(\nabla u(p)\bigr) - f^\prime\bigl(\nabla u(q)\bigr) \bigr] z_q \Bigr)
	\end{equation*}
	in $H^1(\Omega)^\ast$. By the same arguments as below \eqref{eq:proof:hoeldercontinuitybeforeAestimate}, 
	$A := \gamma I + f^\prime\bigl( \nabla u(p) \bigr)$ satisfies $\norm[C^{0,\alpha}(\Omega)]{A} \leq a^0$ with a constant $a^0$ that does not depend on $\gamma,\delta,p,q$. Moreover, $A$ is elliptic with constant $\gamma_0$.
	By Theorem~\ref{thm:linearregularity} this yields
	\begin{equation*}
	\Norm[C^{1,\alpha}(\Omega)]{r}  
	\leq C \Norm[C^{0,\alpha}(\Omega)]{\bigl[ f^\prime\bigl(\nabla u(p)\bigr) - f^\prime\bigl(\nabla u(q)\bigr) \bigr] z_q}.
	\end{equation*}
	Here, $C>0$ does not depend on $p,q$, but only on the desired quantities.
	From Lemma~\ref{lem_hoeldercompositions} and Theorem~\ref{thm_PtoUfrechet} we infer that
	\begin{equation*}
	\Norm[C^{1,\alpha}(\Omega)]{r} 
	\leq C \Norm[C^{0,\alpha}(\Omega)]{f^\prime\bigl(\nabla u(p)\bigr) - f^\prime\bigl(\nabla u(q)\bigr)} \Norm[L^\infty(\Omega)]{d}.
	\end{equation*}
	Lemma~\ref{lem_hoeldercompositions} and Lemma~\ref{lem_Hoeldercontinuityofregmeancurvatureequation} therefore imply
	\begin{equation*}
		\begin{split}
		 & \Norm[\Lin( L^\infty(\Omega), C^{1,\alpha}(\Omega) )]{u^\prime(p) - u^\prime(q)}\\ 
		& \leq C \Norm[C^{0,\alpha}(\Omega)]{\int_0^1 f^{\prime\prime}\Bigl( \nabla u(q) + t \bigl[ \nabla u(p) - \nabla u(q) \bigr]\Bigr) \dt}\Norm[C^{0,\alpha}(\Omega)]{\nabla u(p) - \nabla u(q)}\\
		& \leq C \int_0^1 \Norm[C^{0,\alpha}(\Omega)]{f^{\prime\prime}\Bigl( \nabla u(q) + t \bigl[ \nabla u(p) - \nabla u(q) \bigr] \Bigr)}\dt \,
		\Norm[L^\infty(\Omega)]{p-q}.
		\end{split}	
	\end{equation*}
	The first factor is bounded since $f''$ is bounded and Lipschitz. This demonstrates the asserted Lipschitz continuity. \qed 
\end{proof}

\begin{remark}
	Theorem~\ref{thm:ullfd} stays valid, for some different $\alpha$, if $\Omega$ is of class $C^{1,\alpha^\prime}$ for some $\alpha^\prime>0$. 
\end{remark}

\section{An inexact Newton method for the regularized problems} \label{sec:newton}

In this section we introduce the formulation of the optimality system of \eqref{eq:regProblem} on which our numerical method is based, and we show that the application of an inexact Newton method to this formulation is globally well-defined and locally convergent.
We use the following assumption.
\begin{assumption}\label{ass8}
	We are given constants $0 < \gamma_0 \leq \gamma^0$, $0 < \delta_0 \leq \delta^0$ and $b^0 \geq 0$.
	We denote $I:=[\gamma_0,\gamma^0]\times [\delta_0, \delta^0]$ and fix $\parambr\in I$.
\end{assumption}

The optimality conditions from Theorem~\ref{thm:optcond} can be cast as  $F(\bar y_\param,\bar p_\param)=0$ for
\begin{equation} \label{eq:Fdefinition}
F: \Y\times P \rightarrow \Y^\ast \times L^2(\Omega), \qquad 
F(y,p):=\begin{pmatrix}
Ay - u(-p) \\
y - y_\Omega - A^\ast p
\end{pmatrix},
\end{equation}
and the pair $(\bar y_\param,\bar p_\param)$ is the unique root of $F$. 
Note that we suppress the dependency of 
$u=u(p;\param)$ and $F=F(y,p; \param)$ on $\param$.  
By standard Sobolev embeddings we have $P\subset H^2(\Omega)\hookrightarrow L^\infty(\Omega)$,
hence $u(-p) \in C^{1,\alpha}(\Omega)$ for some $\alpha>0$ by Lemma~\ref{lem_Hoeldercontinuityofregmeancurvatureequation}, so
$F$ is well-defined. A Newton system with a somewhat similar structure is considered in \cite{Schiela_IPMefficient}. 

To find the root of $F$ we apply the inexact Newton method Algorithm~\ref{alg:newton}. 

\begin{algorithm2e}
	\DontPrintSemicolon
	\caption{An inexact Newton method for \eqref{eq:regProblem}}\label{alg:newton} 
	\KwIn{ $(y_0, p_0) \in \Y\times P$, $(\gamma,\delta) \in\Rpos^2$, $\eta\in[0,\infty)$ } 
	\For(){$k=0,1,2,\ldots$} 
	{	
		\lIf{$F(y_k,p_k)=0$}{set $(y^\ast,p^\ast):=(y_k,p_k)$; \textbf{stop}}
		Compute $(\delta y_k, \delta p_k)$ such that $\norm{F(y_k, p_k)+F^\prime( y_k, p_k ) (\delta y_k, \delta p_k)}\leq \eta_k\norm{F(y_k, p_k)}$, where $\eta_k\in[0,\eta]$\\
		Set $(y_{k+1}, p_{k+1}) = (y_k,p_k) + (\delta y_k, \delta p_k)$
	}
	\KwOut{ $(y^\ast,p^\ast)$}
\end{algorithm2e}

The remainder of this section is devoted to the convergence analysis for Algorithm~\ref{alg:newton}. A similar analysis can be carried out if the optimality system of \eqref{eq:regProblem} and the inexact Newton method are based on $u$ instead of $(y,p)$. However, in our numerical experiments the
path-following method based on \eqref{eq:Fdefinition}, cf. section~\ref{sec:numericalsolution},
was clearly superior to its counterpart based on $u$: The former could reduce $\param$ to much smaller values than the latter
and was also significantly more robust. Both observations are well in line with our previous experience \cite{kruse,Clason2018,Hafemeyer15,HafemeyerMaster,Hafemeyer2019} on PDE-constrained optimal control problems involving the TV seminorm.

Since the homotopy path $\parambr\mapsto(\bar u_{\param},\bar y_{\param},\bar p_{\param})$ is not affected by the formulation of the optimality system, 
we conjecture that the superior performance of path-following based on \eqref{eq:Fdefinition} is related to the fact that 
$(\bar y_\param,\bar p_\param)$ converges to $(\bar y,\bar p)$ in stronger norms than $\bar u_\param$ to $\bar u$, cf. Theorem~\ref{thm:barybarpconv}, respectively, Theorem~\ref{thm_convbaru}.

There are many works in which the control is considered as an implicit variable of some sort or avoided altogether, e.g., \cite{kruse,Clason2011,HintermKunisch,HintermStadler,Hinze05,HinzeTroe,NeitzelPruefertSlawig,PieperDiss,Schiela_IPMefficient,WeiserGaenzlerSchiela}. 
Concerning the optimal triple $(\bar y,\bar p,\bar u)$ for \eqref{eq:ocintro} 
we share with those works the idea to base the optimality system on the smoother variables $(y,p)$. In contrast to those works, however, \eqref{eq:Fdefinition} does neither improve the regularity of the controls that appear as iterates (in comparison to a formulation based on $u$) nor does it circumvent their use. 

The next two lemmas yield convergence of Algorithm~\ref{alg:newton}.

\begin{lemma} \label{lem:LipschitzcontOfF}
	Let Assumption~\ref{ass8} hold. 
	Then $F$ defined in \eqref{eq:Fdefinition} is locally Lipschitz continuously \Frechet differentiable. Its derivative at $(y,p)\in Y\times P$ is given by
	\begin{equation*}
	F^\prime(y,p) \colon  \Y\times P \rightarrow \Y^\ast\times L^2(\Omega), \qquad
	(\delta y,\delta p) \mapsto \begin{pmatrix}
	A & u^\prime(-p) \\
	I & -A^\ast 
	\end{pmatrix} \begin{pmatrix}
	\delta y\\ \delta p
	\end{pmatrix}.
	\end{equation*}
\end{lemma}

\begin{proof}
	Only $p\mapsto u(-p)$ is nonlinear, so the claims follow from Theorem~\ref{thm:ullfd}. \qed
\end{proof}

\begin{lemma} \label{lem:Fprimeinvertible}
	Let Assumption~\ref{ass8} hold. Then $F^\prime(y,p)$ is invertible for all $(y,p) \in \Y\times P$.
\end{lemma}

\begin{proof}
	The proof consists of two parts. First we show that $F^\prime(y, p)$ is injective and second that it is a Fredholm operator of index $0$, see \cite[Chapter~IV, Section~5]{Kato1966}. These two facts imply the bijectivity of $F^\prime(y, p)$.
	For the injectivity let $(\delta y, \delta p) \in \Y\times P$ with $F^\prime(y,p)(\delta y, \delta p) = 0 \in \Y^\ast \times L^2(\Omega)$, i.e.
	\begin{align} \label{eq:proof:Fprimeinverse1}
	0 = A\delta y+ u^\prime(-p)\delta p \in \Y^\ast \qquad\text{ and }\qquad 0 = \delta y - A^\ast \delta p \in L^2(\Omega),
	\end{align}
	and therefore
	\begin{equation*}
	\Norm[L^2(\Omega)]{\delta y}^2 = (A^\ast \delta p, \delta y)_{L^2(\Omega)} = - (u^\prime(-p) \delta p, \delta p)_{L^2(\Omega)}.
	\end{equation*}
	The representation of $z:=u^\prime(-p) \delta p$ from Theorem~\ref{thm_PtoUfrechet} yields
	\begin{equation} \label{eq:proof:ellipofuprime}
	\begin{split}
	-\Norm[L^2(\Omega)]{\delta y}^2 
	& = \Biggl(\Bigl[ \gamma I + f^\prime\bigl( \nabla u(-p) \bigr) \Bigr] \nabla z, \nabla z \Biggr)_{L^2(\Omega)} 
	+ \gamma \bigl( z, z \bigr)_{L^2(\Omega)}\\
	& \geq \Bigl( f^\prime\bigl( \nabla u(-p) \bigr) \nabla z, \nabla z \Bigr)_{L^2(\Omega)}.
	\end{split}
	\end{equation}
	Since $f^{\prime}$ is positive semi-definite, we find $\norm[L^2(\Omega)]{\delta y}^2 \leq 0$. This shows $\delta y = 0$.
	By \eqref{eq:proof:Fprimeinverse1} this yields $A^\ast \delta p = 0$ in $L^2(\Omega)$, hence $\delta p = 0$, which proves the injectivity.\\
	To apply Fredholm theory we decompose $F^\prime(y,p)$ into the two operators
	\begin{equation*}
	F^\prime(y,p) = \begin{pmatrix}
	A & 0 \\ 0 & - A^\ast
	\end{pmatrix}
	+
	\begin{pmatrix}
	0 & u^\prime(-p) \\ I & 0 
	\end{pmatrix}.
	\end{equation*}
	We want to use \cite[Chapter~IV, Theorem~5.26]{Kato1966}, which states: If the first operator is a Fredholm operator of index $0$ and the second operator is compact with respect to the first operator (see \cite[Chapter~IV, Introduction to Section~3]{Kato1966}), then their sum $F^\prime(y,p)$ is also a Fredholm operator of index $0$. By the injectivity of $F^\prime(y,p)$ this implies its bijectivity.
	
	The operators $A:\Y\rightarrow \Y^\ast$ and $A^\ast: P \rightarrow L^2(\Omega)$ are invertible
	by Lemma~\ref{lem:solutionoperator}, and thus
	\begin{equation*}
	\Y\times P\rightarrow \Y^\ast \times L^2(\Omega), \qquad
	(\delta y, \delta p) \mapsto \begin{pmatrix} A & 0 \\ 0 & - A^\ast
	\end{pmatrix} \begin{pmatrix}
	\delta y\\ \delta p
	\end{pmatrix}
	\end{equation*}
	is invertible and in particular a Fredholm operator of index $0$.
	It remains to show that
	\begin{equation*}
	\Y\times P\rightarrow \Y^\ast \times L^2(\Omega), \qquad
	(\delta y, \delta p) \mapsto  \begin{pmatrix}
	0 & u^\prime(-p) \\ I & 0 
	\end{pmatrix} \begin{pmatrix}
	\delta y\\ \delta p
	\end{pmatrix}
	\end{equation*}
	is compact with respect to the first operator. 
	Thus, we have to establish that for any sequence $( (\delta y_n, \delta p_n) )_{n\in\N} \subset\Y\times P$ such that there exists a $C>0$ with 
	\begin{align} \label{eq:proof:Fprimeinverse2}
	\Bigl( \Norm[\Y]{\delta y_n} + \Norm[P]{\delta p_n} \Bigr) + \left( \Norm[\Y^\ast]{A \delta y_n} + \Norm[L^2(\Omega)]{A^\ast \delta p_n} \right) \leq C \qquad \forall n\in\N,
	\end{align}
	the sequence $\bigl( ( u^\prime(-p) \delta p_n, \delta y_n ) \bigr)_{n\in\N} \subset \Y^\ast \times L^2(\Omega)$ contains a convergent subsequence. By \eqref{eq:proof:Fprimeinverse2} we have that $( \norm[\Y]{\delta y_n} )_{n\in\N}$ is bounded. 
	The compact embedding $\Y \hookrightarrow\hookrightarrow L^2(\Omega)$ therefore implies the existence of a point $\hat y \in L^2(\Omega)$ and a subsequence, denoted in the same way, such that $\norm[L^2(\Omega)]{\delta y_n - \hat y}\to 0$.  
	We also have that $( \norm[P]{\delta p_n} )_{n\in\N}$ is bounded. In particular $\norm[L^\infty(\Omega)]{\delta p_n} \leq b^0$ for all $n\in\N$ for some $b^0>0$. By Theorem~\ref{thm_PtoUfrechet} this implies that $( u^\prime(-p) \delta p_n )_{n\in\N}$ is bounded in $C^{1,\alpha}(\Omega)$. 
	Since $C^{1,\alpha}(\Omega)\hookrightarrow\hookrightarrow \Y^\ast$, 
	the proof is complete. \qed  
\end{proof}

\begin{remark}
	Lemma~\ref{lem:Fprimeinvertible} implies that Algorithm~\ref{alg:newton} is globally well-defined.
\end{remark}	

It is well-known that the properties established in Lemma~\ref{lem:LipschitzcontOfF} and Lemma~\ref{lem:Fprimeinvertible} are sufficient for local linear/q-superlinear/q-quadratic convergence of the inexact Newton method if the residual in iteration $k$ is of appropriate order, e.g. \cite[Theorem~6.1.4]{Kelley_Itmethodseq}. Thus, we obtain the following result.

\begin{theorem}\label{thm_convinexNewtoninfdim}
	Let Assumption~\ref{ass8} hold. 
	If $(y_0,p_0)\in\Y\times P$ is sufficiently close to $(\bar y_{\param},\bar p_{\param})$, then Algorithm~\ref{alg:newton} either terminates after finitely many iterations with output $(y^\ast,p^\ast)=(\bar y_{\param},\bar p_{\param})$ or it generates a sequence $((y_k,p_k))_{k\in\N}$ that converges 
	to $(\bar y_{\param},\bar p_{\param})$. The convergence rate is
	r-linear if $\eta<1$, q-linear if $\eta$ is sufficiently small,
	q-superlinear if $\eta_k\to 0$, and of q-order $1+\omega$ if $\eta_k=O(\norm{F(y_k,p_k)}^\omega)$. Here, $\omega\in(0,1]$ is arbitrary; for $\omega=1$ this means q-quadratic convergence.
\end{theorem}

\begin{remark}
	Lemma~\ref{lem_Hoeldercontinuityofregmeancurvatureequation} shows that 
	convergence of $(p_k)_{k\in\N}$ (with a certain rate) implies convergence of $(u(p_k))_{k\in\N}$ in $C^{1,\alpha}(\Omega)$ (with a related rate).
\end{remark}

\section{Finite Element approximation} \label{sec:FEapproximation}

In this section we provide a discretization scheme for \eqref{eq:regProblem} 
and prove its convergence. Throughout, we work with a fixed pair $\parambr\in\Rpos^2$. 

\subsection{Discretization} \label{sec:femdiscr}

We use Finite Elements for the discretization of \eqref{eq:regProblem}. Control, state and adjoint state are discretized by piecewise linear and globally continuous elements on a triangular grid. We point out that discretizing the control by piecewise constant Finite Elements 
will not ensure convergence to the optimal control $\bar u_{\param}$, in general; cf. \cite[Section~4]{Bartels2012}.

For all $h\in(0,h_0]$ and a suitable $h_0>0$ let $\mathcal{T}_h$ denote a collection of open triangular cells $T \subset\Omega$ with $h = \max_{T\in\mathcal{T}_h} \operatorname{diam}(T)$. We write $\Omega_h := \operatorname{int}( \cup_{T\in\mathcal{T}_h} \bar T )$.
We assume that there are constants $C>0$ and $c > \frac{1}{2}$ such that
\begin{equation} \label{eq:boundarydist}
\max_{x\in\partial\Omega_h}\operatorname{dist}( x, \partial\Omega ) \leq Ch^c, \qquad |\Omega\setminus\Omega_h| \xrightarrow{h\rightarrow 0} 0, \qquad |\partial\Omega_h| \leq C
\end{equation}
for all $h\in(0,h_0]$. 
We further assume $(\mathcal{T}_h)_{h\in(0,h_0]}$ to be quasi-uniform and $\Omega_h \subset \Omega_{h^\prime}$ for $h^\prime \leq h$. The assumptions in \eqref{eq:boundarydist} are rather mild and in part implied if, for example, $\Omega$ and $(\Omega_h)_{h>0}$ are a family of uniform Lipschitz domains, cf. \cite[Sections 4.1.2 \& 4.1.3]{Hafemeyer2020}.
We utilize the function spaces
\begin{equation*}
V_h := \left\lbrace v_h\in C(\bar\Omega_h): v_h|_T \text{ is affine linear } \forall \,T\in\mathcal{T}_h \right\rbrace, \quad \VhO := V_h \cap H^1_0(\Omega_h).
\end{equation*}
Because $V_h\hookrightarrow H^1(\Omega_h)$ it follows that $\VhO$ contains precisely those functions of $V_h$ that vanish on $\partial\Omega_h$. 
We use the standard nodal basis $\varphi_1, \varphi_2, \dots, \varphi_{\dim(\Vh)}$ in $\Vh$ and assume that it is ordered in such a way that $\varphi_1, \varphi_2, \dots, \varphi_{\dim( \VhO )}$ is a basis of $\VhO$.
For every $u\in L^2(\Omega_h)$ there is a unique $y_h\in\VhO$ that satisfies
\begin{equation*}
\int_{\Omega_h} \left(\sum_{i,j=1}^N a_{ij}\partial_i y_h \partial_{j}\varphi_h\right) + c_0 y_h \varphi_h \dx 
=  \int_{\Omega_h} u \varphi_h \dx \qquad\forall\varphi_h\in\VhO
\end{equation*}
and by defining $S_h u:=y_h$ we obtain the discrete solution operator
$S_h: L^2(\Omega_h) \rightarrow \VhO$ to the PDE in \eqref{eq:ocintro}. 
The discretized version of \eqref{eq:regProblem} is given by 
\begin{equation*} \tag{\mbox{ROC$_{\gamma,\delta,h}$}} \label{eq:regProblemdisc}
\min_{u\in \Vh} \; \underbrace{\frac{1}{2} \Norm[L^2(\Omega_h)]{S_h u-y_{\Omega}}^2 + \beta \psi_{\delta,h}(u) + \frac{\gamma}{2}\Norm[H^1(\Omega_h)]{u}^2}_{=:j_\paramh(u)},
\end{equation*}
where $\psi_{\delta,h}:H^1(\Omega_h)\rightarrow\R$ is defined in the same way as $\psi_\delta$, but with $\Omega$ replaced by $\Omega_h$.
By standard arguments this problem has a unique optimal solution $\bar u_\paramh$. 
Based on $\bar u_\paramh$ we define 
$\bar y_\paramh:=S_h\bar u_\paramh$ and $\bar p_\paramh:=S_h^\ast(S_h\bar u_\paramh-{y_{\Omega}})$.
For $h\to 0$ the triple $(\bar u_\paramh,\bar y_\paramh, \bar p_\paramh)$ converges to the continuous optimal triple $(\bar u_\param,\bar y_\param,\bar p_\param)$ in an appropriate sense, as we show next.

\subsection{Convergence} \label{sec:disconv}

In this section we prove convergence of the Finite Element approximation. 
We will tacitly use that extension-by-zero yields for each $v\in\VhO \subset H^1_0(\Omega_h)$ a function in $H^1_0(\Omega)$. 
Also, we need the following density result. 
\begin{lemma} \label{lem:pseudointerpol}
	Let \eqref{eq:boundarydist} hold. 
	For each $\varphi\in C_0^\infty(\Omega)$ there exists a sequence $(\varphi_h)_{h > 0}$ with $\varphi_h\in Y_h$ for all $h$ such that $\lim_{h\to 0^+} \norm[H^1(\Omega_h)]{\varphi_h-\varphi} = 0$.
\end{lemma}

\begin{proof}
	Let $\varphi\in C_0^\infty(\Omega)$. 
	Due to \eqref{eq:boundarydist} we have $\operatorname{supp}(\varphi)\subset \Omega_h$ for all sufficiently small $h$.
	The claim then follows by choosing for $\varphi_h$ the nodal interpolant of $\varphi$ since $\lim_{h\to 0^+}\norm[H^1(\Omega_h)]{\varphi_h-\varphi}=0$ for this choice, see \cite[Theorem 1.103]{Ern2004}. 
	\qed 
\end{proof}

\begin{theorem}\label{thm_FEMconvergenceinLtwo}
	Let \eqref{eq:boundarydist} hold. 
	We have 
	\begin{equation*}
		\lim_{h\to 0^+}\Norm[L^2(\Omega)^3]{\left(\bar u_\paramh,\bar y_\paramh,\bar p_\paramh\right) - \left(\bar u_\param,\bar y_\param,\bar p_\param\right)} = 0,
	\end{equation*}
	where $\bar u_\paramh$, $\bar y_\paramh$ and $\bar p_\paramh$ are extended by zero to $\Omega$.
\end{theorem}

\begin{proof}
	For ease of notation we do not change indices in this proof when passing to subsequences.
	Let $(h_n)_{n\in\N}$ be a zero sequence, without loss of generality monotonically decreasing.
	From $j_\paramn(\bar u_\paramn)\leq j_\paramn(0)\leq j_\param(0)$ it follows that there is a constant $C>0$, independent of $n$, such that $\lVert \bar u_\paramn \rVert_{H^1(\Omega_{h_n})} \leq C$. This implies
	$\lVert \bar y_\paramn \rVert_{H^1(\Omega_{h_n})} \leq C$.
	Using extension by zero we find that $\lVert \bar y_\paramn \rVert_{H^1(\Omega)} \leq C$  for some $C$ that is still independent of $n$.
	From the compact embedding of $H^1_0(\Omega)$ into $L^2(\Omega)$ and the reflexivity of $H^1_0(\Omega)$ we obtain a subsequence and a $\hat y \in H^1_0(\Omega)$ such that $\bar y_\paramn \xrightarrow{n\rightarrow\infty} \hat y$ strongly in $L^2(\Omega)$ and weakly in $H_0^1(\Omega)$.
	Extending $\bar u_\paramn$ by $0$ to $\Omega$ and using the reflexivity of $L^2(\Omega)$ we obtain on a subsequence that $\bar u_\paramn 
	\xrightarrow{n\rightarrow\infty} \hat u$ weakly in $L^2(\Omega)$ for some $\hat u\in L^2(\Omega)$.
	Let $\varphi\in C_0^\infty(\Omega)$ and $(\varphi_{h_n})_{n\in\mathbb{N}}$ be defined as in Lemma~\ref{lem:pseudointerpol}. Extending $\varphi_{h_n}$ by zero we have
	\begin{equation*}
		0 = A( \bar y_\paramn) \varphi_{h_n} - (\bar u_\paramn, \varphi_{h_n} )_{L^2(\Omega_{h_n})} \xrightarrow{n\rightarrow\infty} A( \hat y ) \varphi  - (\hat u, \varphi )_{L^2(\Omega)}.
	\end{equation*}
	Thus $\hat y = S\hat u$ by the density of $C_0^\infty(\Omega)$ in $H^1_0(\Omega)$. Analogous arguments show that the adjoints $\bar p_\paramn$ converge in the same way to some $\hat p\in H_0^1(\Omega)$ with $\hat p = S^\ast(\hat y-y_\Omega)$.
	It therefore suffices to prove that $\hat u=\bar u_\param$, 
	i.e., that $\hat u$ minimizes $j_\param$, and that $\bar u_\paramn \xrightarrow{n\rightarrow\infty} \hat u$ strongly in $L^2(\Omega)$. 
	Let $u\in H^1(\Omega) \cap C^\infty(\Omega)$ and denote by $I_h u \in H^1(\Omega_h)$ the usual nodal interpolant. Then it is well-known, e.g. \cite[Theorem 1.103]{Ern2004}, that $\lVert u-I_{h_n}u \rVert_{H^1(\Omega_{h_n})} \xrightarrow{n\rightarrow\infty} 0$.
	Let $\hat n\in\N$ and $n\geq \hat n$. Using $\Omega_{h_{\hat n}} \subset \Omega_{h_n}$ and the optimality of $\bar u_\paramn$ we find 
	\begin{equation*}
		 J_{\paramnhat}(\bar y_\paramn,\bar u_\paramn)\leq j_\paramn(\bar u_\paramn) \leq j_\paramn ( I_{h_n} u ),
	\end{equation*}
	where $J_{\paramnhat}:L^2(\Omega_{h_{\hat n}})\times H^1(\Omega_{h_{\hat n}})\rightarrow\R$ is given by
	\begin{equation*}
	 	J_{\paramnhat}(v,w) := \frac{1}{2} \norm[L^2(\Omega_{h_{\hat n}})]{v-y_\Omega}^2 + \beta \psi_{\delta,h_{\hat n}}(w) + \frac{\gamma}{2}\Norm[H^1(\Omega_{h_{\hat n}})]{w}^2.
	\end{equation*}	
	By $\lVert u-I_{h_n}u \rVert_{H^1(\Omega_{h_n})} \xrightarrow{n\rightarrow\infty} 0$
	we obtain
	\begin{equation*}
		\limsup_{n\rightarrow\infty} \, J_\paramnhat(\bar y_\paramn,\bar u_\paramn) \leq j_\param(u).
	\end{equation*}
	From $\lVert \bar u_\paramn \rVert_{H^1(\Omega_{h_{\hat n}})}
	\leq \lVert \bar u_\paramn \rVert_{H^1(\Omega_{h_n})} \leq C$ we infer that there exists $\tilde u\in H^1(\Omega_{h_{\hat n}})$ such that $\bar u_\paramn \xrightarrow{n\rightarrow\infty} \tilde u$ weakly in $H^1(\Omega_{h_{\hat n}})$ and strongly in $L^2(\Omega_{h_{\hat n}})$.
	On the other hand, we have $\bar u_\paramn \xrightarrow{n\rightarrow\infty} \hat u$ weakly in $L^2(\Omega_{h_{\hat n}})$, so $\hat u\in H^1(\Omega_{h_{\hat n}})$ and the convergence $\bar u_\paramn \xrightarrow{n\rightarrow\infty} \hat u$ holds strongly in $L^2(\Omega_{h_{\hat n}})$ and weakly in 
	$H^1(\Omega_{h_{\hat n}})$. 
	The semi-continuity properties of $\psi_{\delta,h_{\hat n}}$ and $\norm[H^1(\Omega_{h_{\hat n}})]{\cdot}$ together with 
	the fact that $\bar y_\paramn \xrightarrow{n\rightarrow\infty} \hat y$ strongly in $L^2(\Omega_{h_{\hat n}})$
	imply that
	\begin{equation*}
		J_{\paramnhat}(\hat y,\hat u)
		\leq \liminf_{n\rightarrow\infty} \, J_{\paramnhat}(\bar y_\paramn,\bar u_\paramn)
		\leq j_\param(u).
	\end{equation*}
	Because of $|\Omega\setminus\Omega_{h_{\hat n}}| \xrightarrow{\hat n\rightarrow\infty} 0$ we can infer by dominated convergence for $\hat n\rightarrow\infty$ that 
	$\bar u_\paramn \xrightarrow{n\rightarrow\infty} \hat u$ strongly in $L^2(\Omega)$
	and that $j_\param(\hat u)\leq j_\param(u)$ for all $u\in H^1(\Omega) \cap C^\infty(\Omega)$. 
	By density, this implies that $\hat u$ is the minimizer of $j_\param$, thereby concluding the proof. \qed 
\end{proof}

\begin{corollary}
	Let \eqref{eq:boundarydist} hold. We have
	\begin{equation*}
		\lim_{h\to 0^+}\Norm[H^1(\Omega)^2]{\left(\bar y_\paramh,\bar p_\paramh\right) - \left(\bar y_\param,\bar p_\param\right)} = 0,
	\end{equation*}
	where $\bar y_\paramh$ and $\bar p_\paramh$ are extended by zero to $\Omega$.
\end{corollary}

\begin{proof}
	Let $R_h \bar y_\param \in Y_h$ denote the Ritz projection with respect to $A$. Extending $\bar y_\paramh\in Y_h$ and $R_h \bar y_\param$ by zero to $\Omega$ we clearly have
	\begin{equation*}
		\lVert \bar y_\paramh - \bar y_\param \rVert_{H^1(\Omega)} \leq \lVert \bar y_\paramh - R_h \bar y_\param \rVert_{H^1(\Omega_h)} + \lVert R_h \bar y_\param - \bar y_\param \rVert_{H^1(\Omega)}.
	\end{equation*}
	By definition, $\bar y_\paramh-R_h \bar y_\param$ satisfies 
	\begin{equation*}
		A(\bar y_\paramh-R_h \bar y_\param) (\varphi_h) = (\bar u_\paramh-\bar u_\param, \varphi_h)_{L^2(\Omega_h)} \qquad \forall\varphi_h\in \VhO.
	\end{equation*}
	Thus, choosing $\varphi_h = \bar y_\paramh-R_h \bar y_\param$ and using the ellipticity of $\calA$ and $c_0\geq 0$ in $\Omega$ together with the Poincar\'e inequality in $\Omega$ yields a constant $C>0$, independent of $h$, such that
	$\lVert \bar y_\paramh-R_h \bar y_\param\rVert_{H^1(\Omega)} \leq C \lVert \bar u_\paramh-\bar u_\param \rVert_{L^2(\Omega)} \xrightarrow{h\rightarrow 0^+} 0$, where we also used extension by zero and Theorem~\ref{thm_FEMconvergenceinLtwo}. Since $R_h \bar y_\param \xrightarrow{h\rightarrow 0^+} \bar y_\param$ in $\VO$, the $H^1(\Omega)$ convergence $\bar y_\paramh\xrightarrow{h\rightarrow 0^+}\bar y_\param$ follows. The proof for $\bar p_\paramh-\bar p_\param$ is analogous. \qed 
\end{proof}

\section{Numerical solution} \label{sec:numericalsolution}

Based on the Finite Element approximation from section~\ref{sec:FEapproximation} we now study an inexact Newton method to compute the discrete solution $(\bar y_{\paramh},\bar p_{\paramh},\bar u_{\paramh})$ and we embed it into a path-following method. 

\subsection{A preconditioned inexact Newton method for the discrete problems}

We prove local convergence of an inexact Newton method when applied to a discretized version of \eqref{eq:Fdefinition} for fixed $\parambr\in\Rpos^2$. To this end, let us introduce the discrete adjoint-to-control mapping $u_h$. 
We recall that the constant $h_0>0$ is introduced at the beginning of section~\ref{sec:femdiscr}.
The proof of the following result is similar to the continuous case in Theorems~\ref{thm_PtoUfrechet}, \ref{thm:ullfd} and \ref{thm:quasilinearls}, so we omit it.

\begin{lemma}\label{lem_Lipdiffofuh}
	Let $h\in(0,h_0]$. 
	For every $p\in L^2(\Omega_h)$ there exists a unique $u_h=u_h(p)\in\Vh$ that satisfies the following discrete version of \eqref{eq:quasilinearpde}
	\begin{equation} \label{eq:discretemeancurve}
	\begin{split}
	\Bigl( \gamma \nabla u_h + f\bigl(\nabla u_h\bigr), \nabla \varphi_h \Bigr)_{L^2(\Omega_h)} + \gamma\bigl( u_h, \varphi_h \bigr)_{L^2(\Omega_h)} = (p, \varphi_h)_{L^2(\Omega_h)}
	 \quad \forall \varphi_h\in\Vh.
	 \end{split}
	\end{equation}
	The associated solution operator $u_h: L^2(\Omega_h) \rightarrow \Vh$ 
	is Lipschitz continuously \Frechet differentiable. Its derivative $u_h^\prime(p)\in\Lin(L^2(\Omega_h),\Vh)$ at $p\in L^2(\Omega_h)$ in direction $d\in L^2(\Omega_h)$ is given by $z_h = u_h^\prime(p) d \in \Vh$, where $z_h$ is the unique solution to
	\begin{equation}\label{eq:discretemeancurveder}
	\begin{split}
	\Biggl( \Bigl[\gamma I + f^\prime\bigl(\nabla u_h(p)\bigr)\Bigr] \nabla z_h, \nabla \varphi_h \Biggr)_{L^2(\Omega_h)} + \gamma \bigl(z_h, \varphi_h\bigr)_{L^2(\Omega_h)} & = (d, \varphi_h)_{L^2(\Omega_h)} \\ &\qquad \forall \varphi_h\in\Vh.
	 \end{split}
	\end{equation}
\end{lemma}

With $u_h$ at hand we can discretize \eqref{eq:Fdefinition} by 
\begin{equation*}
F_h : \VhO\times \VhO \rightarrow \VhOast \times \VhOast, \qquad F_h(y, p) := \begin{pmatrix}
A y - u_h(-p) \\
y-y_\Omega - A^\ast p
\end{pmatrix}.
\end{equation*}
The same $F_h$ is obtained if we consider the optimality conditions of \eqref{eq:regProblemdisc} and express them in terms of $(y,p)$. Moreover, $(\bar y_\paramh,\bar p_\paramh)$ is the unique root of $F_h$ and 
the properties of $F$ from Lemma~\ref{lem:LipschitzcontOfF} and Lemma~\ref{lem:Fprimeinvertible} carry over to $F_h$.
\begin{lemma} \label{lem:discpropofFprime}
	Let $h\in(0,h_0]$. 
	The map $F_h: \VhO \times \VhO \rightarrow \VhOast \times \VhOast$ is Lipschitz continuously \Frechet differentiable. Its derivative at $(y,p) \in \VhO \times \VhO$ is given by
	\begin{equation*}
	F_h^\prime(y,p) \colon \VhO \times \VhO \rightarrow \VhOast \times \VhOast, \qquad
	(\delta y,\delta p) \mapsto \begin{pmatrix}
	A & u_h^\prime(-p) \\
	I & -A^\ast 
	\end{pmatrix} \begin{pmatrix}
	\delta y\\ \delta p
	\end{pmatrix}.
	\end{equation*}
	Moreover, $F_h^\prime(y, p)$ is invertible for every $(y,p)\in\VhO\times\VhO$.
\end{lemma}

\begin{proof}
	The regularity follows from Lemma~\ref{lem_Lipdiffofuh}.
	Since $\dim( \VhO \times \VhO ) = \dim( \VhOast \times \VhOast )$, it is sufficient to show that $F_h^\prime(y,p)$ is injective. This can be done exactly as in Lemma~\ref{lem:Fprimeinvertible}. \qed
\end{proof}

Similar to Theorem~\ref{thm_convinexNewtoninfdim} we have the following result. 

\begin{theorem} \label{thm:discnewton}
	Let $h\in(0,h_0]$ and $\eta\in[0,\infty)$. 
	Then there is a neighborhood $N\subset \VhO\times\VhO$ of $(\bar y_{\paramh},\bar p_{\paramh})$ 
	such that for any $(y_0,p_0)\in N$ any sequence $((y_k,p_k))_{k\in\N}$ that is generated according to
	$(y_{k+1},p_{k+1})=(y_k,p_k) + (\delta y_k,\delta p_k)$, where
	$(\delta y_k,\delta p_k)\in\VhO\times \VhO$ satisfies for all $k\geq 0$
	\begin{equation*}
		\Norm{F_h (y_k,p_k) + F_h^\prime(y_k,p_k)(\delta y_k,\delta p_k) }\leq \eta_k \Norm{F_h (y_k,p_k)} 
	\end{equation*}
	with $(\eta_k)\subset[0,\eta]$, converges to $(\bar y_{\paramh},\bar p_{\paramh})$. The convergence is 
	r-linear if $\eta<1$, q-linear if $\eta$ is sufficiently small, 
	q-superlinear if $\eta_k\to 0$, and of q-order $1+\omega$ if 
	$\eta_k=O(\norm{F_h(y_k,p_k)}^\omega)$.
	Here, $\omega\in(0,1]$ is arbitrary.
\end{theorem}

As a preconditioner for the fully discrete Newton system based on
\begin{equation}\label{eq_defPrecond}
	F_h^\prime(y,p)
	= \begin{pmatrix}
		\mathbf{A} & \mathbf{u_h^\prime(-p)} \\ \mathbf{M} & -\mathbf{A}^T
	\end{pmatrix} \quad\text{ we use }\quad
	\calP^{-1} :=
	\begin{pmatrix}
		\mathbf{B}&\mathbf{0}\\ \mathbf{B}^T \mathbf{M} \mathbf{B} & -\mathbf{B}^T
	\end{pmatrix}, 
\end{equation}
where $\mathbf{B}=\mathbf{A}^{-1}$ is the inverse of the stiffness matrix $\mathbf{A}$, and $\mathbf{M}$ is the mass matrix.
The preconditioner would agree with $F_h^\prime(y,p)^{-1}$ if $\mathbf{u_h^\prime(-p)}$ were zero.

\subsection{A practical path-following method} \label{sec:fullydiscalg}

The following Algorithm~\ref{alg:pfdisc} is a practical path-following inexact Newton method to solve \eqref{eq:regProblemdisc}. 
We expect that its global convergence can be shown if $\rho(\gamma_i,\delta_i)$ and $\eta_k$ are chosen sufficiently small, 
but this topic is left for future research. For the choices detailed below, global convergence holds in practice.

\begin{algorithm2e} 
	\DontPrintSemicolon
	\caption{Inexact path-following inexact Newton method \label{alg:pfdisc}} 
	\KwIn{ 
		$(\hat y_0, \hat p_0) \in \VhO\times\VhO$, $(\gamma_0,\delta_0) \in \Rpos^2$, $\kappa>0$ }
	\For(){$i=0,1,2,\ldots$}
	{	
		set $(y_0, p_0) := (\hat y_i, \hat p_i)$ \\
		\For(\label{line_forloopNewton}){$k=0,1,2,\ldots$}
		{	
			\If(\label{line_Newtontermination}){$\norm{F_h(y_k,p_k)}\leq \rho(\gamma_i,\delta_i)$}
			{	
				set $(\hat y_{i+1}, \hat p_{i+1}):=(y_k,p_k)$\\
				\textbf{go to }line~\ref{line_endofl} }\label{line_endoft} 
			choose $\eta_k>0$ and use preconditioned \textsc{gmres} to determine $(\delta y_k, \delta p_k)$ such that 
			$\norm{F_h(y_k, p_k) + F_h^\prime( y_k, p_k ) (\delta y_k, \delta p_k)}\leq \eta_k\norm{F_h(y_k,p_k)}$
			\label{line_inexacttermination} 
			\\
			\textbf{call}\, Algorithm~\ref{alg:ls}, input $w_k:=(y_k, p_k)$, $\delta w_k:=(\delta y_k,\delta p_k)$; 
			output: $\lambda_k$\\				
			set $(y_{k+1}, p_{k+1}) := (y_k,p_k) + \lambda_k(\delta y_k, \delta p_k)$	
		}\label{line_endofk}
		select $\sigma_i\in(0,1)$\label{line_endofl}\\
		\lIf{$\norm[H^1]{(\hat y_{j+1},\beta^{-1}\hat p_{j+1})-(\hat y_j,\beta^{-1}\hat p_j)}\leq (1-\sigma_j)\kappa\norm[H^1]{(\hat y_{i+1},\beta^{-1}\hat p_{i+1})}$ \rm{for} $j=i$ and $j=i-1$}{terminate with output $(y^\ast,p^\ast):=(\hat y_{i+1},\hat p_{i+1})$ \label{line_pftermination}}
		\lElse{set $(\gamma_{i+1},\delta_{i+1}):= (\sigma_i\gamma_i,\sigma_i \delta_i)$\label{line_paramupdate}}
	}
	\KwOut{$(y^\ast,p^\ast)$}
\end{algorithm2e}

The inner loop in lines~\ref{line_forloopNewton}--\ref{line_endofk} of Algorithm~\ref{alg:pfdisc}
uses an inexact Newton method to compute an approximation $(\hat y_{i+1}, \hat p_{i+1})$ of the root of $F_h$ for fixed $(\gamma_i,\delta_i)$
satisfying $\norm{F_h(\hat y_{i+1},\hat p_{i+1})}\leq \rho(\gamma_i,\delta_i)$, where $\rho:\Rpos^2\rightarrow\Rpos$. 
In the implementation we use $\rho(\gamma,\delta) = \max\{10^{-6},\gamma\}$, which may be viewed as inexact path-following. For the forcing term $\eta_k$ we use the two choices $\eta_k=\bar\eta_k:=10^{-6}$ and $\eta_k=\hat\eta_k:=\max\{10^{-6},\min\{10^{-k-1},\sqrt{\delta_i}\}\}$, where $k=k(i)$. For $\bar\eta_k$ we have $\bar\eta_k\leq \norm{F_h(y_k,p_k)}$ since we terminate the inner loop if $\norm{F_h(y_k,p_k)} < 10^{-6}$.
The choice $\eta_k=\bar\eta_k$ is related to quadratic convergence, while $\eta_k=\hat \eta_k$ corresponds to superlinear convergence, cf. Theorem~\ref{thm:discnewton}.
We also terminate \gmres if the Euclidean norm of $F_h(y_k, p_k) + F_h^\prime( y_k, p_k ) (\delta y_k, \delta p_k)$ drops below $\eta_k$ 
since this seemed beneficial in the numerical experiments. 

To compute the control $u_h(-p_k)$ that satisfies \eqref{eq:discretemeancurve} we use a globalized Newton method
that can be shown to converge q-quadratically for arbitrary starting points $u^0\in\Vh$. 
In fact, since \eqref{eq:discretemeancurve} is equivalent to $u_h(p)$ being the minimizer of the smooth and strongly convex problem
\begin{equation*}
	\min_{v_h\in V_h} \, \beta\psi_{\delta,h}(v_h)+\frac{\gamma}{2}\lVert v_h\rVert_{H^1(\Omega_h)}^2 - \left(p,v_h\right)_{L^2(\Omega_h)},
\end{equation*}
standard globalization techniques for Newton's method will ensure these convergence properties (e.g., Newton's method combined with an Armijo line search \cite[Section~9.5.3]{BoydVandenberghe}).
The method terminates when the Newton residual falls below a threshold that decreases with $(\gamma_i,\delta_i)$. The linear systems are solved using SciPy's sparse direct solver \texttt{spsolve}. As an alternative we tested a preconditioned conjugate gradients method (PCG). The results were mixed: The use of PCG diminished the total runtime of Algorithm~\ref{alg:pfdisc} if all went well, but broke down on several occasions for smaller values of $(\gamma_i,\delta_i)$.

In lines~\ref{line_endofl}--\ref{line_paramupdate} it is decided whether to accept $(\hat y_{i+1},\hat p_{i+1})$ as a solution and terminate the algorithm; if not, then we continue the path-following by updating $(\gamma_i,\delta_i)$ with the factor $\sigma_i$. 
We select $\sigma_i$ based on the number of Newton steps that are needed to compute the implicit controls $\{u_h(-p_k)\}_k$ in outer iteration $i$. If this number surpasses a predefined $m\in\N$, then we choose $\sigma_i>\sigma_{i-1}$. If it belongs to $[0,0.75 m]$, then we choose $\sigma_i<\sigma_{i-1}$. Otherwise, we let $\sigma_i=\sigma_{i-1}$. In addition, we enforce $\sigma_i\geq 0.25$ for all $i$ since we found in the numerical experiments
that choosing $\sigma_i$ too small can slow down or prevent convergence in some cases once $(\gamma_i,\delta_i)$ is very small, cf. Table~\ref{tab_ex1_sigma} below.
The weighing $1/\beta$ in the termination criterion is made since the amplitude of the adjoint state is roughly of order $\beta$ in comparison to the state.
In all experiments we use $\kappa=10^{-3}$.

Algorithm~\ref{alg:ls} augments the inexact Newton method in lines~\ref{line_forloopNewton}--\ref{line_endofk} of Algorithm~\ref{alg:pfdisc} by a non-monotone line search globalization introduced in \cite{LiFukushima2000} for Broyden's method. The non-monotonicity allows to always accept the inexact Newton step and yields potentially larger step sizes than descent-based strategies. The intention is to keep the number of trial step sizes low since every trial step size requires the evaluation of $F_h$ and hence a recomputation of $u_h(-p_k)$. In the numerical experiments we use $\tau=10^{-4}$ and we observe that in the vast majority of iterations full steps are taken, i.e., $\lambda_k=1$. 
To briefly discuss convergence properties of the globalized inexact Newton method, let us assume for simplicity that $u_h(-p_k)$ is determined exactly for each $k$. 
By following the arguments of \cite{LiFukushima2000} we can show that for sufficiently small $\eta_k$ the sequence $((y_k,p_k))_{k\in\N}$ obtained by ignoring lines~\ref{line_Newtontermination}--\ref{line_endoft} must either be unbounded or converge to the unique root of $F_h$; a key ingredient in the corresponding proof is that $F_h'(y,p)$ is invertible for all $(y,p)$, which we have demonstrated in Lemma~\ref{lem:discpropofFprime}. In particular, if $((y_k,p_k))_{k\in\N}$ is bounded, then the globalized inexact Newton method in lines~\ref{line_forloopNewton}--\ref{line_endofk} terminates finitely. 
While we always observed this termination in practice, the question whether $((y_k,p_k))_{k\in\N}$ can be unbounded remains open. 
Furthermore, as in \cite{LiFukushima2000} it follows that if $((y_k,p_k))_{k\in\N}$ converges, then eventually step size 1 is always accepted, in turn ensuring that the convergence rates of Theorem~\ref{thm:discnewton} apply. 

All norms without index in Algorithm~\ref{alg:pfdisc} and \ref{alg:ls} are $L^2(\Omega_h)$ norms.

\begin{algorithm2e}
	\DontPrintSemicolon
	\caption{Computation of step size}\label{alg:ls}
	\KwInofupalg{ $(w_k,\delta w_k)$, 
		\, $\tau>0$ }\\
	\For{ $l=0,1,2,\ldots$ }
	{   
		\lIf{$\norm{F_h(w_k+2^{-l} \delta w_k)}\leq \Bigl(1+\frac{1}{(l+1)^2}\Bigr)\norm{F_h(w_k)} - \tau\norm{2^{-l} \delta w_k}^2$}{set $\lambda_k:=2^{-l}$; \textbf{stop}}\label{alg2_line_residualcheck2}
	}
	\KwOut{ $\lambda_k$ }
\end{algorithm2e}

\section{Numerical results} \label{sec:numericmainchapter}

We provide numerical results for two examples. Our main goal is to illustrate that Algorithm~\ref{alg:pfdisc}
can robustly compute accurate solutions of \eqref{eq:ocintro}. 
The results are obtained from a Python implementation of Algorithm~\ref{alg:pfdisc} using DOLFIN \cite{LoggWells2010a,LoggWellsEtAl2012a}, 
which is part of FEniCS \cite{AlnaesBlechta2015a,LoggMardalEtAl2012a}. 
The code for the second example is available at \url{https://arxiv.org/abs/2010.11628}.

\subsection{Example~1: An example with explicit solution}

The first example has an explicit solution and satisfies the assumptions used in this work.
We consider \eqref{eq:ocintro} for an arbitrary $\beta>0$ with non-convex $C^\infty$ domain $\Omega = B_{4\pi}(0) \setminus \overline{B_{2\pi}(0)}$ in $\R^2$,
$\calA=-\Delta$ and $c_0\equiv 0$. 
The desired state is
\begin{equation*}
	y_\Omega(r) = \frac{\beta}{2r^3}\Bigl((1+r)\sin(r) - 1 - (2r^2-1)\cos(r)\Bigr) + \bar y
\end{equation*}
where $r(x,y)=\sqrt{x^2+y^2}$, and the optimal state $\bar y$ is
\begin{equation*}
\bar y(r) = \begin{cases}
- \frac{r^2}{4} + A\ln(r/(4\pi)) + B & \text{ if } r\in (2\pi,3\pi),\\
C \ln (r/(4\pi)) & \text{ if } r\in (3\pi,4\pi)
\end{cases}
\end{equation*}
with constants $A,B,C$ whose values are contained in appendix~\ref{sec_examplewithexplicitsolution}.
The optimal control is 
\begin{equation*}
\bar u(r)= 1_{(2\pi,3\pi)}(r),
\end{equation*} 
i.e., $\bar u$ has value $1$ on the disc $B_{3\pi}(0) \setminus \overline{B_{2\pi}(0)}$ and value $0$ on the disc $B_{4\pi}(0) \setminus {B_{3\pi}(0)}$.
The optimal value is $j(\bar u)\approx 24.85 \beta^2 + 59.22 \beta$.
In appendix~\ref{sec_examplewithexplicitsolution} we provide details on the construction of this example and verify that 
$(\bar y,\bar u)$ is indeed the optimal solution of \eqref{eq:ocintro}. 
If not stated otherwise, we use $\beta=10^{-3}$.

We use unstructured triangulations that approximate $\partial\Omega$ increasingly better as the meshes become finer, cf. \eqref{eq:boundarydist}. 
Figure~\ref{fig_ex1} depicts the optimal control $\bar u_h$, optimal state $\bar y_h$ and negative optimal adjoint state $-\bar p_h$, which were computed by Algorithm~\ref{alg:pfdisc} on a grid with 1553207 degrees of freedom (DOF).

\begin{figure}
	\centering
	\begin{subfigure}[c]{.495\textwidth}
		\centering
		\includegraphics[width=\linewidth]{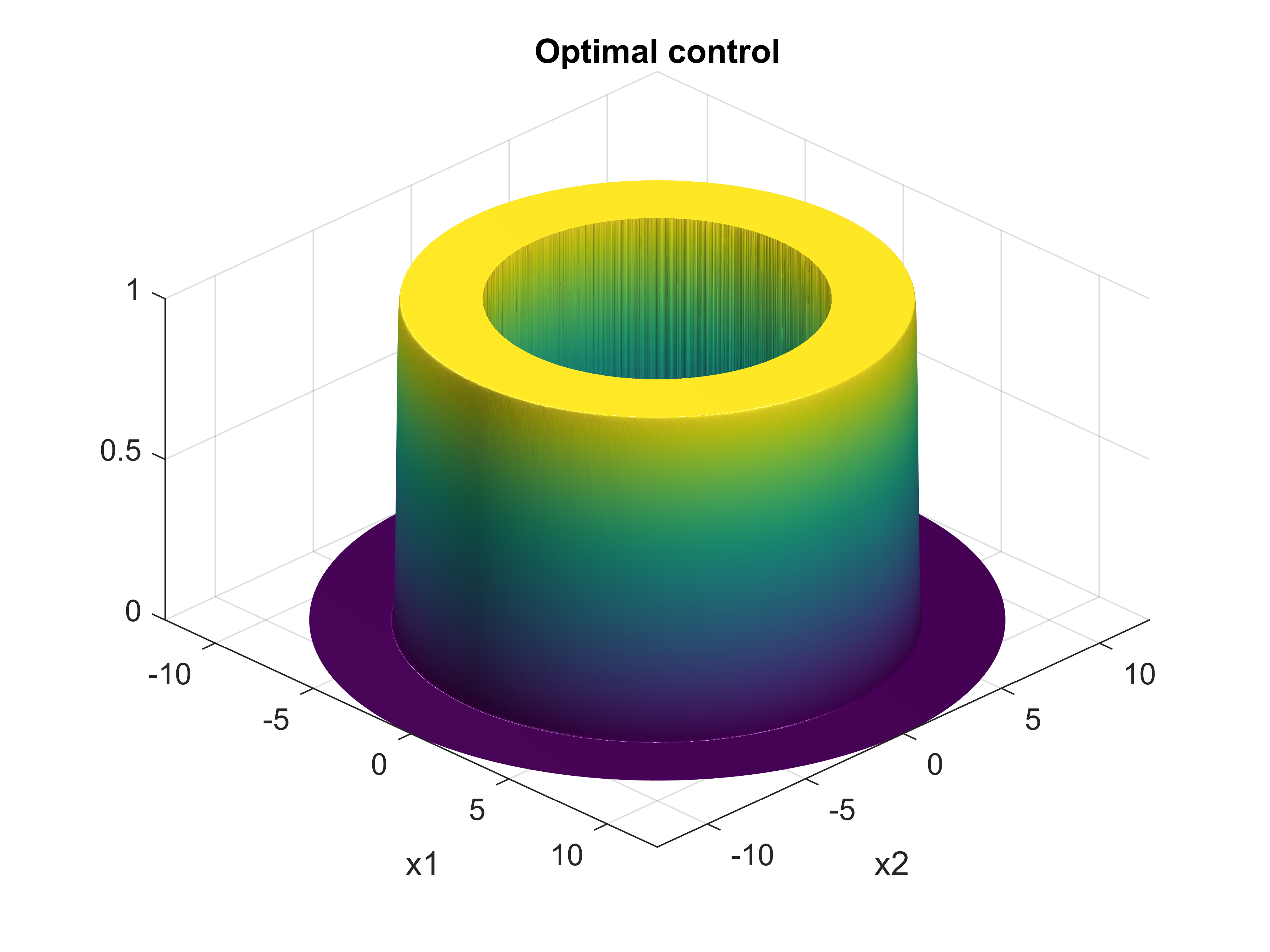}
		\caption{$\bar u_h$}
	\end{subfigure}
	\hfill
	\begin{subfigure}[c]{.495\textwidth}
	\centering
	\includegraphics[width=\linewidth]{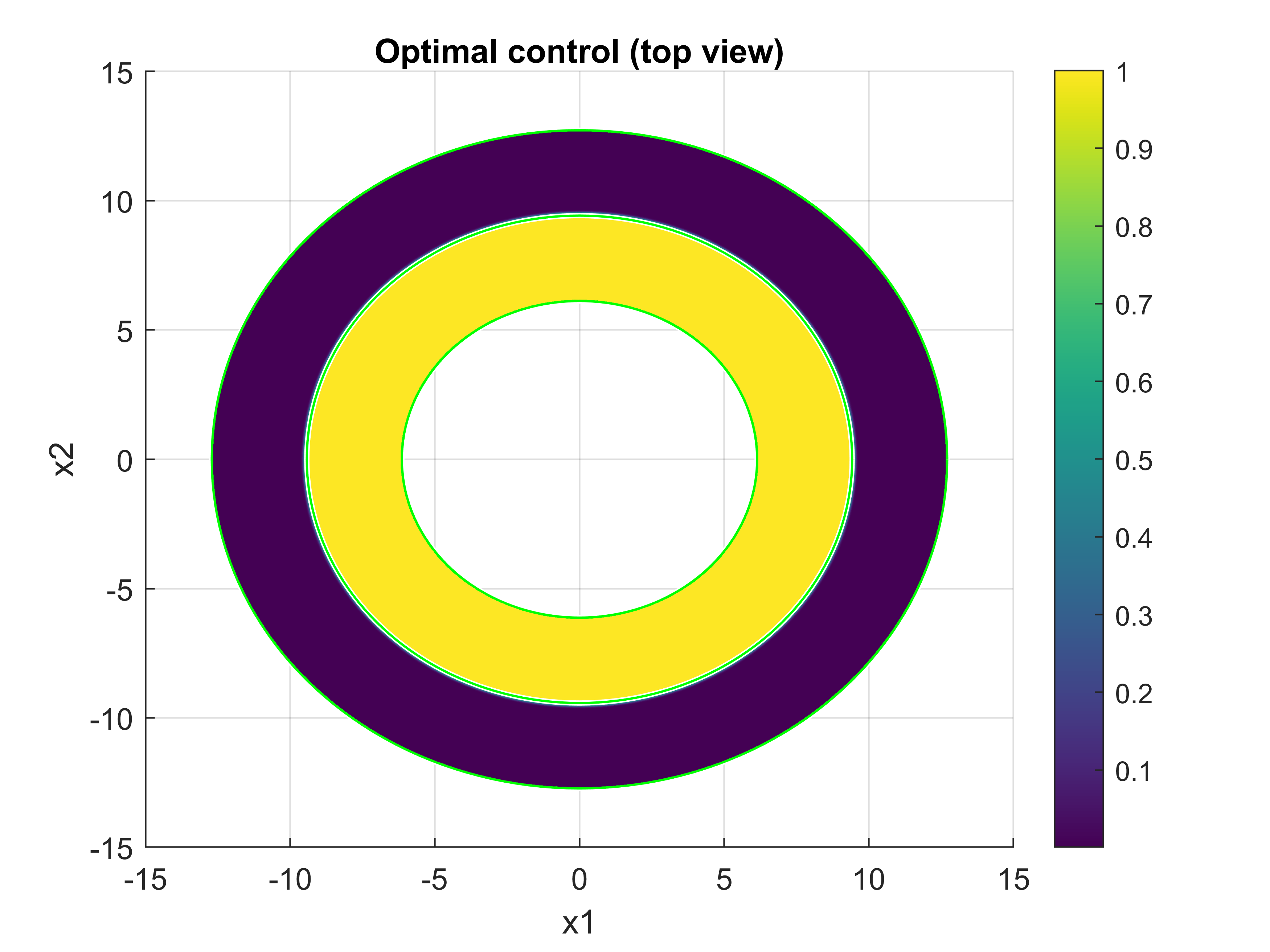}
	\caption{$\bar u_h$ with circles of radii $j\pi$, $j\in\{2,3,4\}$} 
	\end{subfigure}

	\bigskip

	\centering
	\begin{subfigure}[c]{.495\textwidth}
	\centering
	\includegraphics[width=\linewidth]{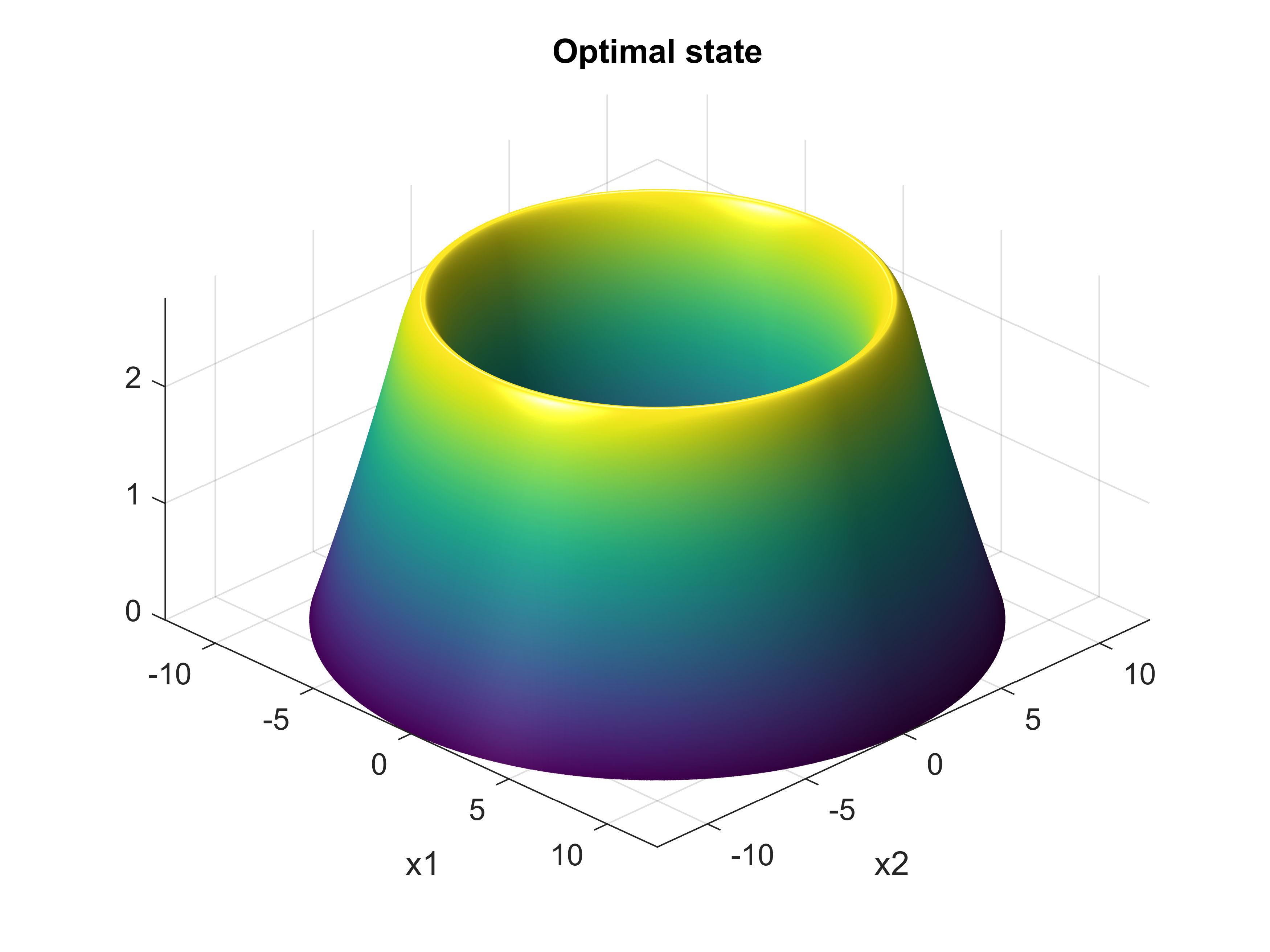}
	\caption{$\bar y_h$}
	\end{subfigure}
	\hfill
	\begin{subfigure}[c]{.495\textwidth}
	\centering
	\includegraphics[width=\linewidth]{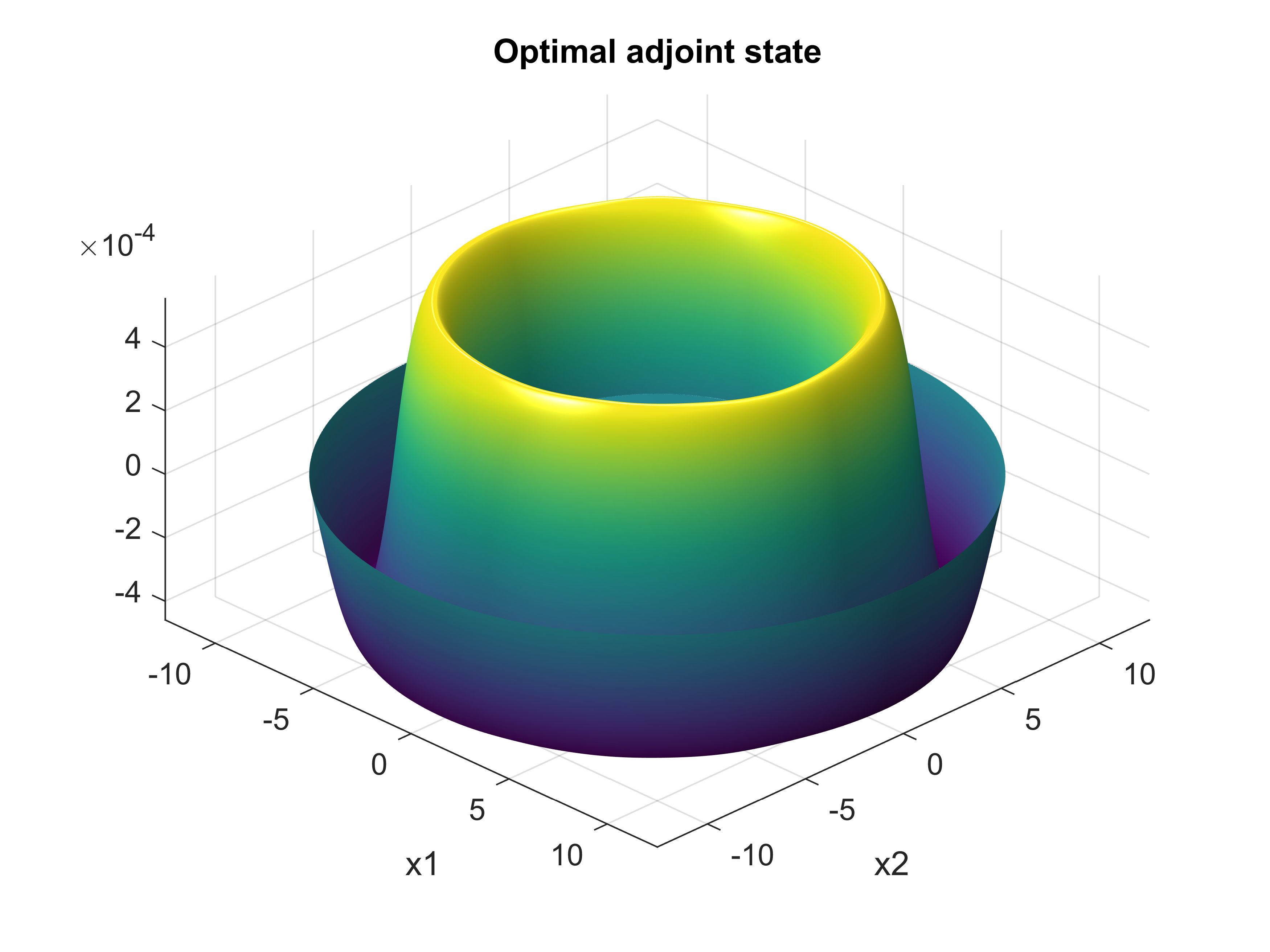}
	\caption{$-\bar p_h$} 
\end{subfigure}
\caption{Numerically computed optimal solutions for Example~1}\label{fig_ex1}
\end{figure}

We begin by studying convergence on several grids. We use the fixed ratio $(\gamma_i/\delta_i)\equiv 10^2$ and apply Algorithm~\ref{alg:pfdisc} with $(\gamma_0,\delta_0)=(1,0.01)$ and $(\hat y_0,\hat p_0)=(0,0)$. 
Table~\ref{tab_ex1_meshind} shows \#it, which represents the total number of inexact Newton steps for $(y,p)$, and \#it$_u$, which is the total number of Newton steps used to compute the implicit function $u$. 
Table~\ref{tab_ex1_meshind} also contains the errors
\begin{equation*}
\calE_j := \bigl\lvert j_{\gamma_\fin,\delta_\fin,h}-\bar j \bigr\rvert, \qquad
\calE_u := \Norm[L^1(\Omega_\ast)]{\hat u_\fin-\bar u}, \qquad
\end{equation*}
as well as 
\begin{equation*}
\calE_y := \Norm[H^1(\Omega_\ast)]{\hat y_\fin-\bar y}, \qquad
\calE_p := \Norm[H^1(\Omega_\ast)]{\hat p_\fin-\bar p}.
\end{equation*}
where $\Omega_\ast$ represents a reference grid with $\DOF=1553207$. To evaluate the errors, $\hat u_\fin$, $\hat y_\fin$ and $\hat p_\fin$ are extended to $\Omega_\ast$ using extrapolation. 
Table~\ref{tab_ex1_courseofalg} provides details for the run from Table~\ref{tab_ex1_meshind} with $\DOF=97643$ and $\eta_k=\hat\eta_k$.
Table~\ref{tab_ex1_courseofalg} includes 
$\tau^i:=\norm[H^1(\Omega_h)]{(\hat y_{i+1},\beta^{-1}\hat p_{i+1})-(\hat y_i,\beta^{-1}\hat p_i)}$, which appears in the termination criterion of Algorithm~\ref{alg:pfdisc}, and $\tau_u^i:=\norm[L^2(\Omega_h)]{u_h(\hat p_{i+1})-u_h(\hat p_i)}$.

\begin{table}
\caption{Example~1: Number of Newton steps and errors for several meshes;
the first value is for the forcing term $\bar\eta_k$, the second for $\hat\eta_k$ (only shown if different)}
\label{tab_ex1_meshind}
\scalebox{0.95}{
\begin{tabular}{@{}llllllll@{}}
	\toprule
		DOF&$\gamma_{\fin}$&\#it&\#it$_u$&$\calE_j$&$\calE_u$&$\calE_y$&$\calE_p$\\\midrule
		$1588$&$1.2\times 10^{-8}/9.3\times 10^{-9}$&58/72
		&390/428
		&$3.4\times 10^{-2}$&$18.7$&$2.3$&$6.0\times 10^{-2}$\\
		$6251$&$1.7/1.6\times 10^{-10}$&78/91
		&597/608
		&$4.0\times 10^{-3}$&$7.3$&$1.1$&$3.3\times 10^{-2}$\\
		$24443$&$2.1/1.5\times 10^{-11}$&55/64
		&454/491
		&$9.4\times 10^{-4}$&$5.1$&$0.50$&$1.3\times 10^{-2}$\\
		$97643$&$5.4/6.2\times 10^{-11}$&46/48
		&407/380
		&$3.3\times 10^{-4}$&$3.7$&$0.22$&$5.6\times 10^{-3}$\\
		$389027$&$4.6/4.3\times 10^{-10}$&32/34
		&367/358
		&$1.2\times 10^{-4}$&$2.8$&$0.09$&$2.9\times 10^{-3}$\\
	\bottomrule
	\end{tabular}
}
\end{table}	

\begin{table}
	\begin{center}
		\caption{Example~1: Course of Algorithm~\ref{alg:pfdisc}}
		\label{tab_ex1_courseofalg}	
	\scalebox{0.76}{		
		\begin{tabular}{@{}llllllllll@{}}
			\toprule
			$i$&$\gamma_i$&$\sigma_i$&(\#it$^i$,\#it$^i_u$)&$\calE_j^i$&$\calE_u^i$&$\calE_y^i$&$\calE_p^i$&$\tau^i$&$\tau_u^i$\\\midrule
			0/1/2&$1.0/0.45/0.18$&$0.45/0.41/0.37$&$(0,0)$&$575$&$155$&$38.5$&$9.8\times 10^{-3}$&$0$&$0$\\
			3&$6.8\times 10^{-2}$&$0.33$&$(1,1)$&$5.7$&$44$&$1.8$&$1.4$&$1440$&$11.3$\\
			4&$2.2\times 10^{-2}$&$0.30$&$(1,2)$&$2.0$&$38$&$1.0$&$0.48$&$959$&$0.83$\\
			\vdots&\vdots&\vdots&\vdots&\vdots&\vdots&\vdots&\vdots&\vdots&\vdots\\
			12&$8.7\times 10^{-6}$&$0.32$&$(3,18)$&$1.8\times 10^{-3}$&$5.1$&$0.23$&$8.8\times 10^{-3}$&$3.5$&$0.22$\\
			13&$2.8\times 10^{-6}$&$0.31$&$(3,20)$&$8.7\times 10^{-4}$&$4.3$&$0.23$&$6.8\times 10^{-3}$&$2.0$&$0.15$\\
			14&$8.7\times 10^{-7}$&$0.28$&$(3,18)$&$5.1\times 10^{-4}$&$3.9$&$0.22$&$6.0\times 10^{-3}$&$0.84$&$0.073$\\
			15&$2.4\times 10^{-7}$&$0.26$&$(5,20)$&$3.8\times 10^{-4}$&$3.7$&$0.22$&$5.7\times 10^{-3}$&$0.32$&$0.027$\\
			16&$6.4\times 10^{-8}$&$0.25$&$(3,15)$&$3.5\times 10^{-4}$&$3.7$&$0.22$&$5.6\times 10^{-3}$&$0.13$&$8.2\times 10^{-3}$\\	
			17&$1.6\times 10^{-8}$&$0.25$&$(3,16)$&$3.3\times 10^{-4}$&$3.7$&$0.22$&$5.6\times 10^{-3}$&$0.081$&$2.2\times 10^{-3}$\\	
			18&$4.0\times 10^{-9}$&$0.25$&$(3,15)$&$3.3\times 10^{-4}$&$3.7$&$0.22$&$5.6\times 10^{-3}$&$0.060$&$7.0\times 10^{-4}$\\	
			19&$9.9\times 10^{-10}$&$0.25$&$(3,20)$&$3.3\times 10^{-4}$&$3.7$&$0.22$&$5.6\times 10^{-3}$&$0.045$&$3.6\times 10^{-4}$\\
			20&$2.5\times 10^{-10}$&$0.25$&$(3,22)$&$3.3\times 10^{-4}$&$3.7$&$0.22$&$5.6\times 10^{-3}$&$0.028$&$2.2\times 10^{-4}$\\
			21&$6.2\times 10^{-11}$&---&$(3,33)$&$3.3\times 10^{-4}$&$3.7$&$0.22$&$5.6\times 10^{-3}$&$0.017$&$1.4\times 10^{-4}$\\
			\bottomrule
		\end{tabular}
	}
	\end{center}
\end{table}	

Table~\ref{tab_ex1_meshind} indicates convergence of the computed solutions 
$(\hat u_\fin,\hat y_\fin,\hat p_\fin)$ to $(\bar u,\bar y,\bar p)$ 
and of the objective value $j_{\gamma_\fin,\delta_\fin,h}$ to $\bar j$. 
It also suggests that convergence takes place at certain rates with respect to $h$.

Moreover, the total number of Newton steps both for $(y,p)$ and for $u$ stays bounded as DOF increases, which suggests mesh independence. The choice $\eta_k=\bar\eta_k$ frequently yields lower numbers of Newton steps for $(y,p)$ and for $u$, yet the runtime (not depicted) is consistently higher than for $\eta_k=\hat\eta_k$ since more iterations of \gmres are required to compute the step for $(y,p)$. Specifically, using $\hat\eta_k$ saves between 5\% and 36\% of runtime, 
with $36\%$ being the saving on the finest grid. 
(Since the runtime depends on many factors, these numbers are intended as reference points rather than exact values.)
In the vast majority of iterations, step size $1$ is accepted for $(y_k,p_k)$. 
For instance, all of the 52 iterations required for $\DOF=97643$ and $\eta_k=\hat\eta_k$ use full steps; 
for $\DOF=6251$ and $\eta_k=\bar\eta_k$, 86 of the 87 iterations use step size 1.

Table~\ref{tab_ex1_sigma} displays the effect of fixing $(\sigma_i)\equiv \sigma$ in Algorithm~\ref{alg:pfdisc}. The mesh uses $\DOF=24443$ and is the same as in Table~\ref{tab_ex1_meshind}. 

\begin{table}
	\begin{center}
		\caption{Example~1: Results for fixed values $(\sigma_i)\equiv\sigma$; the first value is for the forcing term $\bar\eta_k$, the second for $\hat\eta_k$ (only shown if different)}
		\label{tab_ex1_sigma}
		\scalebox{0.99}{
		\begin{tabular}{@{}llllllll@{}}
			\toprule
			$\sigma$&$\gamma_\fin$&\#it&\#it$_u$&$\calE_j$&$\calE_u$&$\calE_y$&$\calE_p$\\\midrule
			0.2&$6.6\times 10^{-12}$&48/51&505/512&$9.4\times 10^{-4}$&$5.1$&$0.50$&$1.3\times 10^{-2}$\\
			0.3&$3.5\times 10^{-11}$&51/57&408/405&$9.4\times 10^{-4}$&$5.1$&$0.50$&$1.3\times 10^{-2}$\\
			0.5&$2.3\times 10^{-10}$&70/86&454/466&$9.4\times 10^{-4}$&$5.1$&$0.50$&$1.3\times 10^{-2}$\\
			0.7&$5.1\times 10^{-10}$&97/130&522/552&$9.4\times 10^{-4}$&$5.1$&$0.50$&$1.3\times 10^{-2}$\\
			0.9&$1.1/8.0\times 10^{-9}$&276/474&1261/1329&$9.4/9.5\times 10^{-4}$&$5.1$&$0.50$&$1.3\times 10^{-2}$\\
			\bottomrule
		\end{tabular}
	}
	\end{center}
\end{table}	

For $\sigma=0.1$ the iterates failed to converge for both forcing terms once $\gamma_{12}=10^{-12}$ is reached because $u_h(-p_k)$ could not be computed to sufficient accuracy within the 200 iterations that we allow for this process. 
Together with the case $\sigma=0.2$ in Table~\ref{tab_ex1_sigma} this shows that small values of $\sigma_i$ can increase the number of steps required for $u$ and even prevent convergence. We therefore enforce $\sigma_i\geq 0.25$ for all $i$ in all other experiments, although this diminishes the efficacy of Algorithm~\ref{alg:pfdisc} in some cases. 

We now turn to the robustness of Algorithm~\ref{alg:pfdisc}. We emphasize that in our numerical experience the robustness of algorithms for optimal control problems involving the TV seminorm in the objective is a delicate issue.
Table~\ref{tab_ex1_differentbeta} displays the iteration numbers required
by Algorithm~\ref{alg:pfdisc} for different values of $\beta$ on the mesh with $\DOF=24443$ along with the error $\calE_u$ for $\eta_k=\hat\eta_k$ for the two choices $(\gamma_i/\delta_i) \equiv 10^2$ and $(\gamma_i/\delta_i)\equiv 1$.
The omitted values for $\beta=10^{-3}$ and 
$(\gamma_i/\delta_i) \equiv 10^2$ are identical to those from Table~\ref{tab_ex1_meshind} for $\DOF=24443$ and $\eta_k=\hat\eta_k$.
Table~\ref{tab_ex1_differentkappa} provides iteration numbers and errors for various fixed choices of $(\gamma_i/\delta_i)$ on the mesh with $\DOF=24443$ for $\beta=10^{-3}$, $\eta_k=\bar\eta_k$ and $(\sigma_i)\equiv 0.5$. For the ratios $10^{-1}$ and $10^{-2}$ we increased $\kappa$ from $10^{-3}$ to $5\cdot 10^{-3}$ to obtain convergence. Since our goal is to demonstrate robustness, no further changes are made although this would lower the iteration numbers.

\begin{table}
	\begin{center}
		\caption{Example~1: Results for various values of $\beta$; the first line is for the choice $(\gamma_i/\delta_i)\equiv 10^2$, the second for $(\gamma_i/\delta_i)\equiv 1$}
		\label{tab_ex1_differentbeta}
		\begin{tabular}{@{}lllll@{}}
			\toprule
			$\beta$&$10^{-1}$&$10^{-2}$&
			$10^{-4}$&$10^{-5}$
			\\\midrule
			$(\#$it,$\#$it$_u)/\calE_u$
			&$(28,373)/21$
			&$(37,217)/7.2$
			&$(100,914)/4.3$
			&$(153,1230)/4.1$
			\\
			$(\#$it,$\#$it$_u)/\calE_u$
			&$(40,435)/21$
			&$(78,795)/7.3$
			&$(126,996)/4.3$
			&$(137,1147)/4.1$
			\\
			\bottomrule
		\end{tabular}
	\end{center}
\end{table}	

\begin{table}
	\begin{center}
		\caption{Example~1: Iteration numbers and errors for several ratios $\gamma_i/\delta_i$; the computations for $\gamma_i/\delta_i\in \{10^{-1},10^{-2}\}$ use a lower accuracy}    
		\label{tab_ex1_differentkappa}	
		\begin{tabular}{@{}llllllll@{}}
			\toprule
			$\frac{\gamma_i}{\delta_i}$&$\delta_\fin$&\#it&\#it$_u$&$\calE_j$&$\calE_u$&$\calE_y$&$\calE_p$\\\midrule
			$10^{-2}$&$3.0\times 10^{-8}$ 
			&71&288&$9.9\times 10^{-4}$&$5.1$&$0.50$&$1.3\times 10^{-2}$\\
			$10^{-1}$&$3.0\times 10^{-8}$ 
			&58 
			&246 
			&$9.9\times 10^{-4}$ 
			&$5.1$&$0.50$&$1.3\times 10^{-2}$\\
			$1$&$1.8\times 10^{-12}$ 
			&102 
			&469 
			&$9.4\times 10^{-4}$&$5.1$&$0.50$&$1.3\times 10^{-2}$\\
			$10^1$&$2.9\times 10^{-12}$ 
			&80   
			&396   
			&$9.4\times 10^{-4}$&$5.1$&$0.50$&$1.3\times 10^{-2}$\\
			$10^{2}$&$2.3\times 10^{-12}$&70&454&$9.4\times 10^{-4}$&$5.1$&$0.50$&$1.3\times 10^{-2}$\\
			$10^3$&$1.9\times 10^{-12}$ 
			&59   
			&475   
			&$9.4\times 10^{-4}$&$5.1$&$0.50$&$1.3\times 10^{-2}$\\
			\bottomrule
		\end{tabular}
	\end{center}
\end{table}	

Table~\ref{tab_ex1_differentbeta} and \ref{tab_ex1_differentkappa} suggest that Algorithm~\ref{alg:pfdisc} is able to handle a range of parameter values without modification of its internal parameters.

\subsection{Example~2}

From section~\ref{sec_regularity} onward we have used that $\Omega$ is of class $C^{1,1}$. 
To show that Algorithm~\ref{alg:pfdisc} can still solve \eqref{eq:ocintro} if $\Omega$ is only Lipschitz,
we consider an example from \cite[section~4.2]{Clason2011} on the square $\Omega = [-1,1]^2$.
We have $\calA=-\Delta$, $c_0\equiv 0$, $\beta=10^{-4}$ and $y_\Omega = 1_D$, where $1_D\colon\Omega\rightarrow\{0,1\}$ is the characteristic function of the square $D=(-0.5,0.5)^2$. 
We use uniform triangulations and denote by $n+1$ the number of nodes in coordinate direction.
Figure~\ref{fig_ex1} depicts the optimal control $\bar u_h$, optimal state $\bar y_h$ and negative optimal adjoint state $-\bar p_h$, which were computed with $n=1024$. Apparently, $\bar u_h$ is piecewise constant.
\begin{figure}
	\centering
\begin{subfigure}[c]{.495\textwidth}
	\centering
	\includegraphics[width=\linewidth]{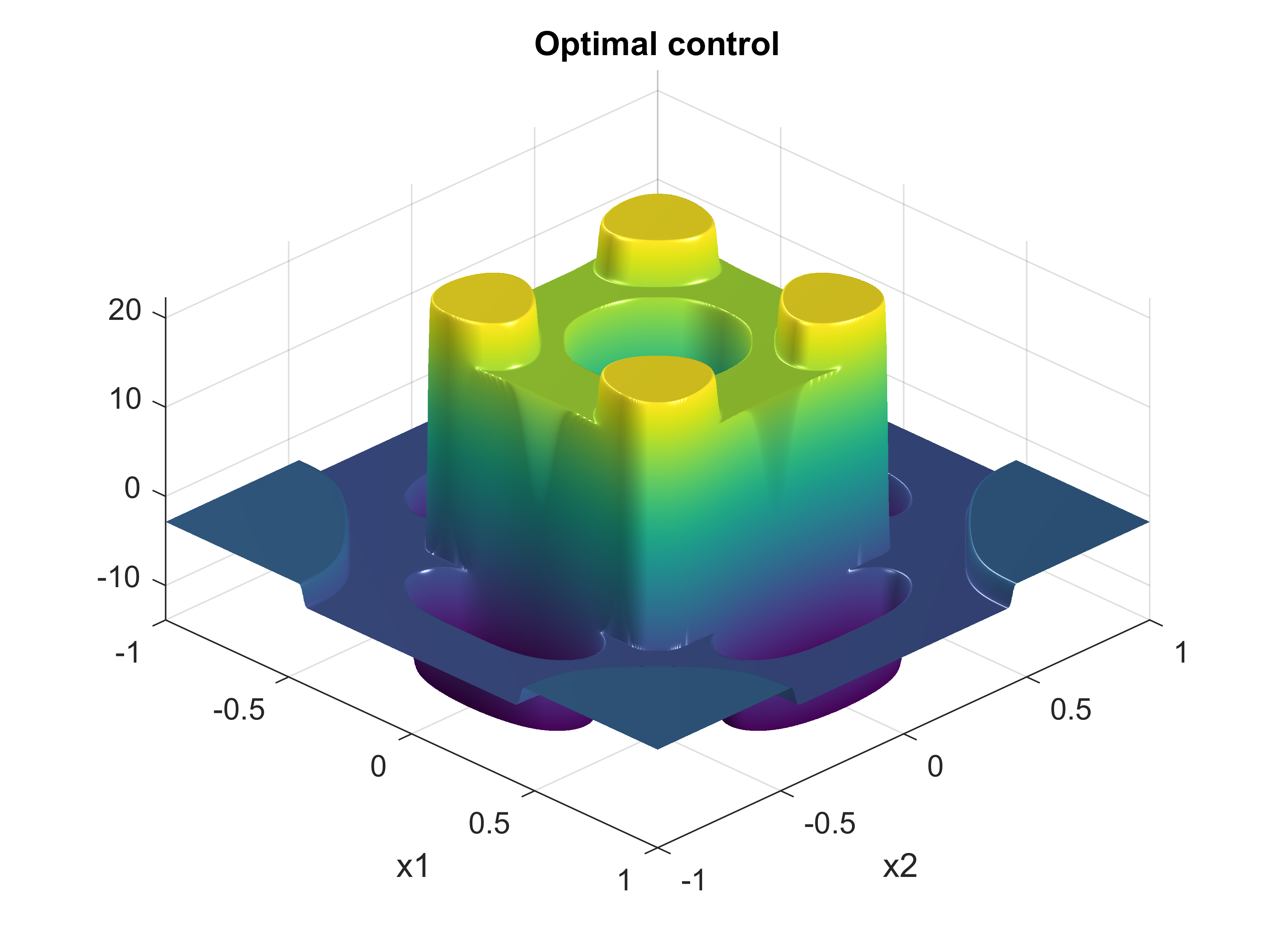}
	\caption{$\bar u_h$}
\end{subfigure}
\hfill
\begin{subfigure}[c]{.495\textwidth}
	\centering
	\includegraphics[width=\linewidth]{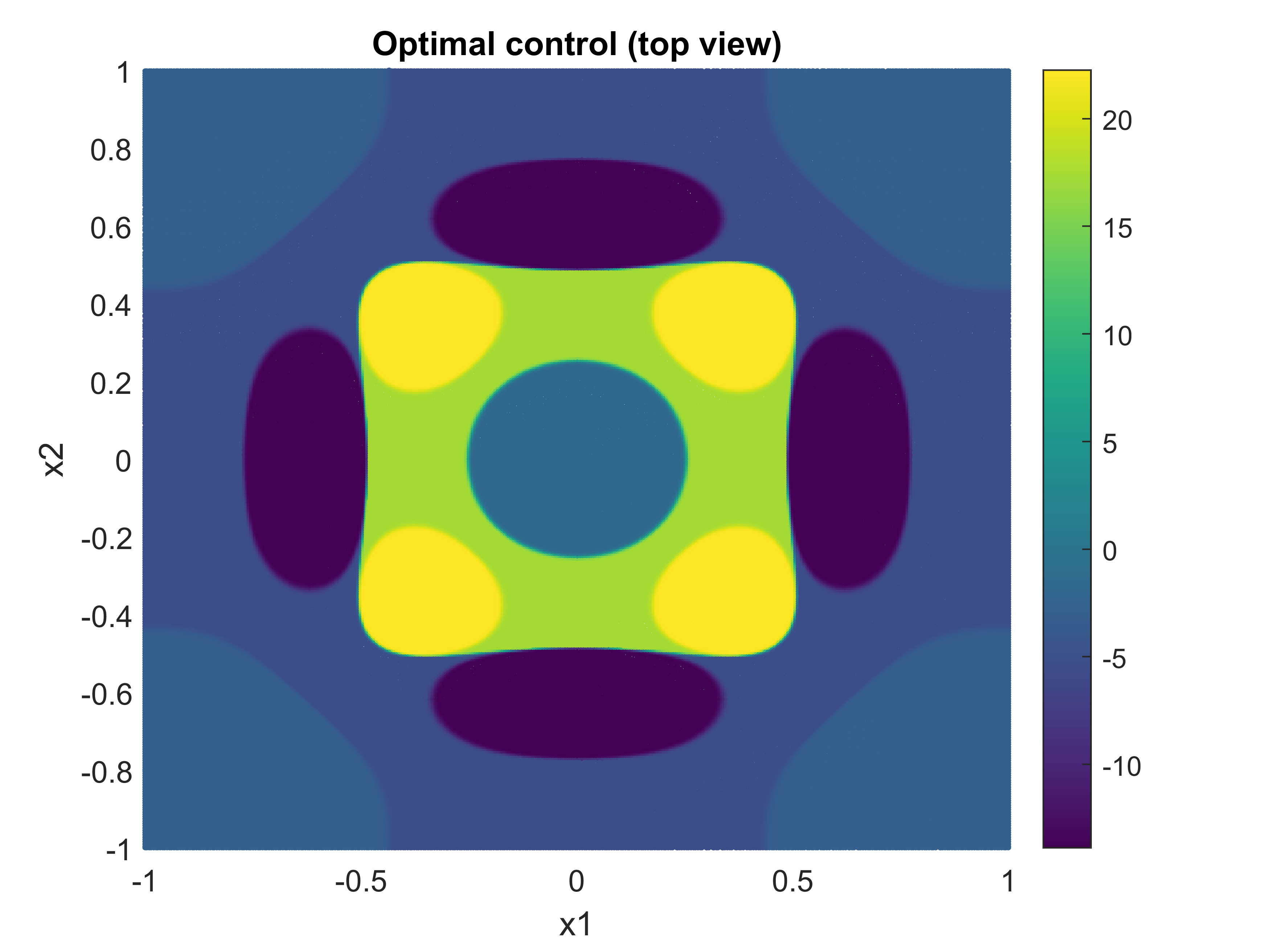}
	\caption{$\bar u_h$ (top view)} 
\end{subfigure}
	
	\bigskip
	
	\centering
	\begin{subfigure}[c]{.495\textwidth}
		\centering
		\includegraphics[width=\linewidth]{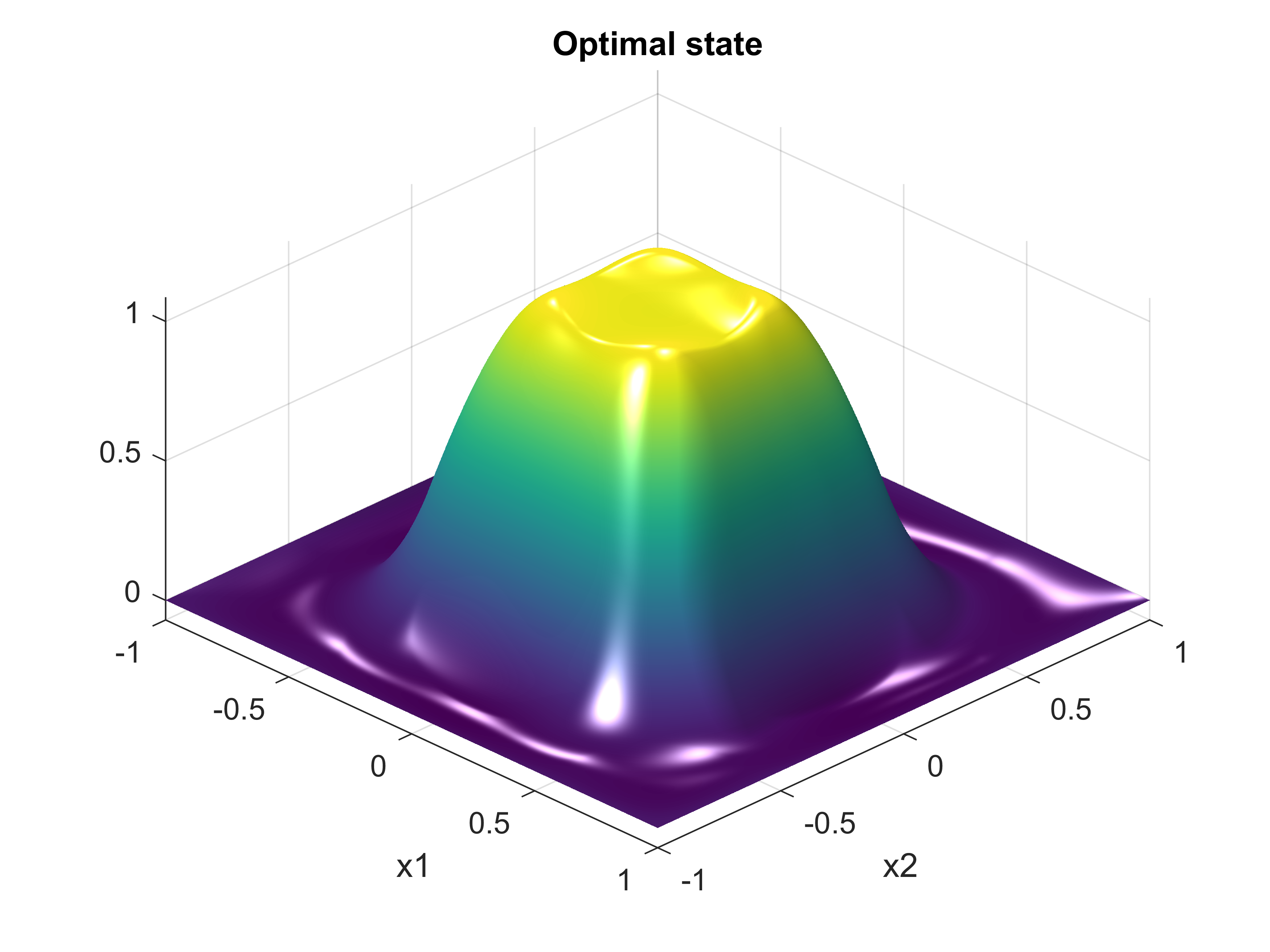}
		\caption{$\bar y_h$}
	\end{subfigure}
	\hfill
	\begin{subfigure}[c]{.495\textwidth}
		\centering
		\includegraphics[width=\linewidth]{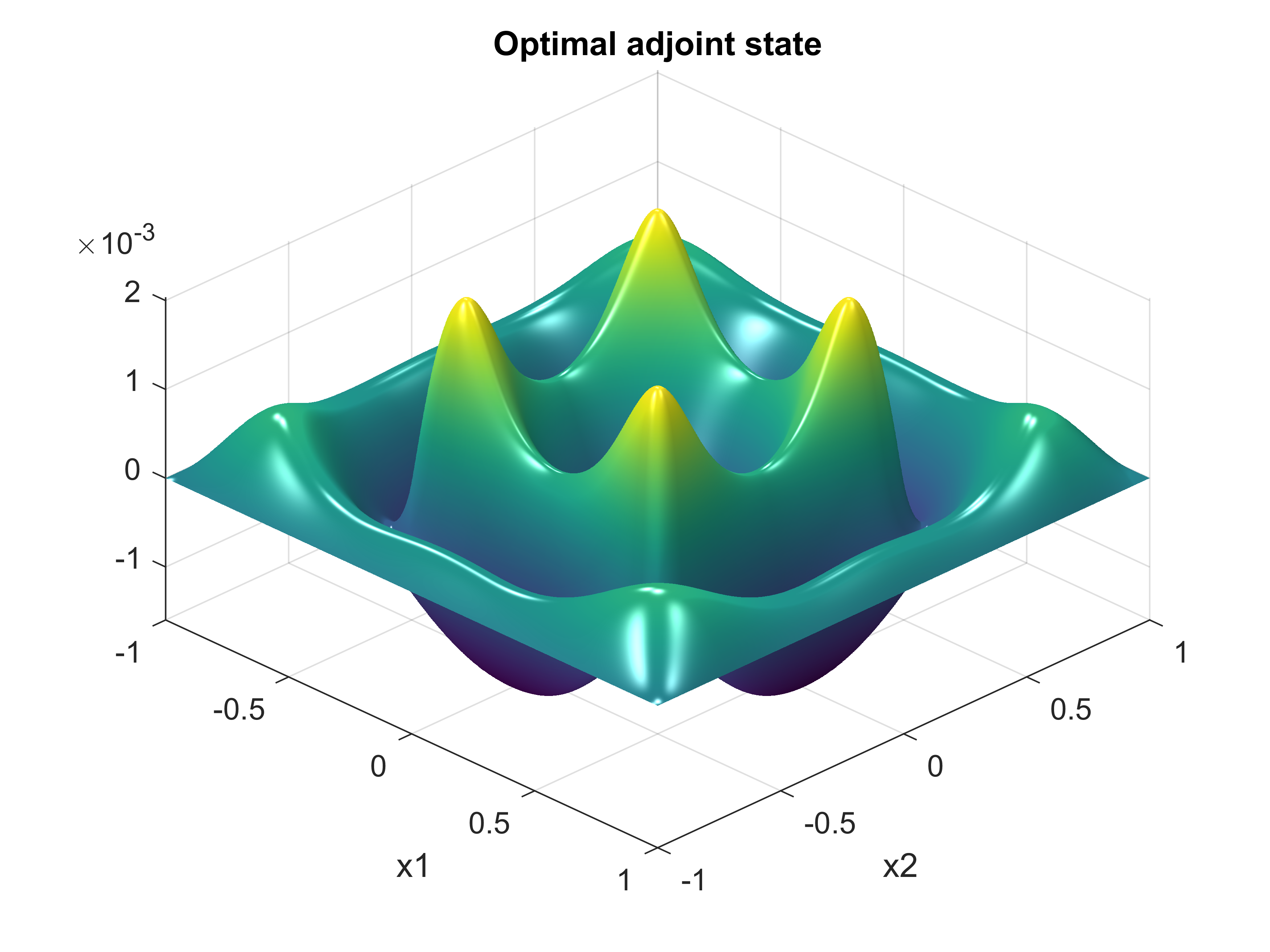}
		\caption{$-\bar p_h$} 
	\end{subfigure}
	\caption{Numerically computed optimal solutions for Example~2}\label{fig_ex2}
\end{figure}

Throughout, we use the fixed ratio $(\gamma_i/\delta_i)\equiv 10^{-2}$ and apply Algorithm~\ref{alg:pfdisc} with $(\gamma_0,\delta_0)=(0.01,1)$ and $(\hat y_0,\hat p_0)=(0,0)$. As in example~1, 
cf. Table~\ref{tab_ex1_differentkappa}, other ratios for $\gamma_i/\delta_i$ can be employed as well.
We only provide results for $\bar\eta_k$ since the forcing term $\hat\eta_k$ does not yield lower runtimes in this example; both forcing terms produce the same errors, though. 
Table~\ref{tab_ex2_meshind} displays iteration numbers and errors for different grids, while Table~\ref{tab_ex2_courseofalg} shows details for $n=256$.

\begin{table}
	\begin{center}
		\caption{Example~2: Number of Newton steps and errors for several meshes}
		\label{tab_ex2_meshind}	
		\begin{tabular}{@{}llllllll@{}}
			\toprule
			$n$&$\gamma_{\fin}$&\#it&\#it$_u$&$\calE_j$&$\calE_u$&$\calE_y$&$\calE_p$\\\midrule
			32&$1.7\times 10^{-11}$&43 
			&321 
			&$1.8\times 10^{-2}$&$8.6$&$0.75$&$9.0\times 10^{-3}$\\
			64&$9.8\times 10^{-12}$ 
			&48
			&551
			&$9.9\times 10^{-3}$&$4.3$&$0.37$&$4.9\times 10^{-3}$\\
			128&$3.2\times 10^{-11}$ 
			&46
			&902
			&$4.9\times 10^{-3}$&$2.3$&$0.19$&$2.4\times 10^{-3}$\\
			256&$3.3\times 10^{-11}$ 
			&50
			&1212 
			&$2.2\times 10^{-3}$&$1.1$&$0.081$&$1.1\times 10^{-3}$\\
			512&$5.6\times 10^{-11}$ 
			&58
			&2868
			&$7.3\times 10^{-4}$
			&$0.42$
			&$0.031$&$4.2\times 10^{-4}$\\
			\bottomrule
		\end{tabular}
	\end{center}
\end{table}	

\begin{table}
	\begin{center}
		\caption{Example~2: Course of Algorithm~\ref{alg:pfdisc}}
		\label{tab_ex2_courseofalg}	
		\scalebox{0.86}{		
			\begin{tabular}{@{}llllllllll@{}}
				\toprule
				$i$&$\gamma_i$&$\sigma_i$&(\#it$^i$,\#it$^i_u$)&$\calE_j^i$&$\calE_u^i$&$\calE_y^i$&$\calE_p^i$&$\tau^i$&$\tau_u^i$\\\midrule
				0--4&$0.01/\ldots$&$0.45/\ldots$&$(0,0)$&$0.42$&$34$&$3.4$&$1.7\times 10^{-2}$&$0$&$0$\\
				5&$6.7\times 10^{-5}$&$0.27$&$(2,7)$&$6.0\times 10^{-2}$&$27$&$1.9$&$4.2\times 10^{-2}$&$498$&$7.3$\\
				6&$1.8\times 10^{-5}$&$0.26$&$(2,37)$&$3.4\times 10^{-2}$&$23$&$1.6$&$2.4\times 10^{-2}$&$212$&$3.4$\\
				\vdots&\vdots&\vdots&\vdots&\vdots&\vdots&\vdots&\vdots&\vdots&\vdots\\
				11&$2.9\times 10^{-7}$&$0.55$&$(2,59)$&$3.1\times 10^{-3}$&$12$&$0.50$&$4.6\times 10^{-3}$&$16.4$&$1.6$\\
				12&$1.6\times 10^{-7}$&$0.55$&$(2,59)$&$1.6\times 10^{-3}$&$9.6$&$0.39$&$3.4\times 10^{-3}$&$12.6$&$1.6$\\
				13&$9.0\times 10^{-8}$&$0.53$&$(2,56)$&$4.4\times 10^{-4}$&$7.4$&$0.29$&$2.5\times 10^{-3}$&$9.5$&$1.4$\\
				14&$4.7\times 10^{-8}$&$0.53$&$(3,70)$&$4.5\times 10^{-4}$&$5.3$&$0.21$&$1.9\times 10^{-3}$&$7.6$&$1.4$\\
				\vdots&\vdots&\vdots&\vdots&\vdots&\vdots&\vdots&\vdots&\vdots&\vdots\\
				21&$9.5\times 10^{-10}$&$0.52$&$(3,64)$&$2.1\times 10^{-3}$&$1.3$&$0.084$&$1.1\times 10^{-3}$&$0.62$&$0.32$\\
				22&$5.0\times 10^{-10}$&$0.47$&$(2,23)$&$2.1\times 10^{-3}$&$1.1$&$0.083$&$1.1\times 10^{-3}$&$0.38$&$0.24$\\
				23&$2.4\times 10^{-10}$&$0.45$&$(2,56)$&$2.1\times 10^{-3}$&$1.1$&$0.081$&$1.1\times 10^{-3}$&$0.28$&$0.27$\\
				24&$1.1\times 10^{-10}$&$0.59$&$(2,80)$&$2.1\times 10^{-3}$&$1.1$&$0.081$&$1.1\times 10^{-3}$&$0.15$&$0.17$\\
				25&$6.2\times 10^{-11}$&$0.53$&$(2,15)$&$2.2\times 10^{-3}$&$1.1$&$0.081$&$1.1\times 10^{-3}$&$0.062$&$0.052$\\
				26&$3.3\times 10^{-11}$&---&$(2,17)$&$2.2\times 10^{-3}$&$1.1$&$0.081$&$1.1\times 10^{-3}$&$0.042$&$0.037$\\
				\bottomrule
			\end{tabular}
		}
	\end{center}
\end{table}

Table~\ref{tab_ex2_meshind} hints at possible mesh independence for $(y,p)$, but indicates that the number of Newton steps for $u$ increases with $n$. The depicted errors are computed by use of a reference solution that is obtained by Algorithm~\ref{alg:pfdisc} with $\eta_k=\bar\eta_k$ on the mesh with $n=1024$.
As in the first example it seems that convergence with respect to $h$ takes place at certain rates. 
The majority of iterations use full Newton steps for $(y,p)$. For instance, all but one of the 50 iterations for $n=256$ use step length one.
We also repeated the experiments from Table~\ref{tab_ex2_meshind} with $y_\Omega$ rotated by $30$, respectively, $45$ degree. 
The omitted results are similar to those in Table~\ref{tab_ex2_meshind}, which further illustrates the robustness of our approach. 

Table~\ref{tab_ex2_nestedgrids} shows the outcome of Algorithm~\ref{alg:pfdisc} if a sequence of nested grids is used, where the grids are refined once $\gamma_i<10^{-4}$, $\gamma_i < 10^{-6}$ and $\gamma_i<10^{-8}$, respectively. In this example, nesting reduces the runtime by about $57\%$ 
while providing the same accuracy as a run for $n=512$, cf. the last line of Table~\ref{tab_ex2_meshind}.

\begin{table}
	\begin{center}
		\caption{Example~2: Results for a sequence of nested grids}
		\label{tab_ex2_nestedgrids}	
		\begin{tabular}{@{}llllllll@{}}
			\toprule
			$n$&$\gamma_\fin$&\#it&\#it$_u$&$\calE_j$&$\calE_u$&$\calE_y$&$\calE_p$\\\midrule
			64&$4.0\times 10^{-5}$&5&26
			&$3.9\times 10^{-2}$&$25$&$1.8$&$3.4\times 10^{-2}$\\
			128&$4.8\times 10^{-7}$&12&260
			&$1.7\times 10^{-3}$&$14$&$0.64$&$6.3\times 10^{-3}$\\
			256&$6.3\times 10^{-9}$&19&481
			&$1.7\times 10^{-3}$&$2.1$&$0.10$&$1.2\times 10^{-3}$\\
			512&$5.6\times 10^{-11}$&20&792
			&$7.3\times 10^{-4}$&$0.42$&$0.031$&$4.2\times 10^{-4}$\\
			\bottomrule
		\end{tabular}
	\end{center}
\end{table}	

Table~\ref{tab_ex2_differentbeta} addresses the robustness of Algorithm~\ref{alg:pfdisc} with respect to $\beta$. The computations are carried out on nested grids and the displayed iteration numbers are those for the finest grid, which has $n=128$. The reference solution is computed for $n=256$. The final grid change happens once $\gamma_i<10^{-8}$. 

\begin{table}
	\begin{center}
		\caption{Example~2: Results for various values of $\beta$.
			A sequence of nested grids is used and the displayed iteration numbers are for the finest grid only}
		\label{tab_ex2_differentbeta}	
		\scalebox{0.8}{
		\begin{tabular}{@{}lllll@{}}
			\toprule
			$\beta$   
			&$10^{-3}$&$10^{-4}$&$10^{-5}$&$5\times 10^{-6}$\\\midrule
			$(\#$it,$\#$it$_u)/\calE_j$
			&$(12,117)/3.3\times 10^{-3}$
			&$(22,355)/2.8\times 10^{-3}$
			&$(70,958)/2.4\times 10^{-3}$ 
			&$(104,1569)/2.2\times 10^{-3}$ 
			\\
			\bottomrule
		\end{tabular}
		}
	\end{center}
\end{table}	

Table~\ref{tab_ex2_differentbeta} indicates that Algorithm~\ref{alg:pfdisc} is robust with respect to $\beta$.
As in example~1 it is possible to achieve lower iteration numbers through manipulation of the algorithmic parameters. For instance, if the final grid change for $\beta=10^{-5}$ happens once $\gamma_i<10^{-9}$ instead of $\gamma_i<10^{-8}$, then only $(41,638)$ iterations are needed on the final grid instead of $(70,958)$. 

\section{Summary}\label{sec_sum}

We have studied an optimal control problem
with controls from BV in which the control costs are given by the TV seminorm, favoring piecewise constant controls. 
By smoothing the TV seminorm and adding the $H^1$ norm we obtained a family of auxiliary problems whose solutions converge to the optimal solution of the original problem in appropriate function spaces. For fixed smoothing and regularization parameter we showed local convergence of an infinite-dimensional inexact Newton method applied to a reformulation of the optimality system that involves the control as an implicit function of the adjoint state.
Based on a convergent Finite Element approximation a practical path-following algorithm was derived, and it was demonstrated that the algorithm is able to robustly compute the optimal solution of the control problem with considerable accuracy. To verify this, a two-dimensional test problem with known solution was constructed. 

\appendix

\section{Differentiability of \texorpdfstring{$\mathbf{\psi_\delta}$}{the smoothed TV seminorm}}\label{sec_differentiabilityofpsidelta}

The following result follows by standard arguments, but we provide its proof for convenience. 

\begin{lemma}\label{thm:psiderivativeNEU}
	Let $\delta>0$, $N\in\N$ and let $\Omega\subset\R^N$ be open. The functional
	\begin{equation*}
	\psi_\delta: H^1(\Omega)\rightarrow\R, \qquad 
	u\mapsto\int_\Omega \sqrt{\delta+\lvert\nabla u\rvert^2}\,\dx 
	\end{equation*}
	is Lipschitz continuously \Frechet differentiable and twice \Gateaux differentiable. Its first derivative at $u$ in direction $v$ and its second derivative at $u$ in directions $v,w$ are given by
	\begin{equation*}
	\psi_\delta^\prime(u)v = \int_\Omega \frac{ \left(\nabla u, \nabla v\right) }{\sqrt{\delta + \rvert\nabla u\lvert^2 } } \dx
	\quad\text{and}\quad
	\psi_\delta^{\prime\prime}(u)[v,w] = \int_\Omega \frac{ \left( \nabla v, \nabla w \right)}{ \sqrt{ \delta + \lvert\nabla u\rvert^2 } } - \frac{ \left(\nabla u, \nabla v \right) \left(\nabla u, \nabla w \right)  }{ \left( \delta + \lvert\nabla u\rvert^2 \right)^{\frac{3}{2}} } \dx.
	\end{equation*}
\end{lemma}

\begin{proof}
	\textbf{First \Gateaux derivative} \\
	Let $u, v\in H^1(\Omega)$. 
	As $s \mapsto \sqrt{\delta + s}$ is Lipschitz on $[0,\infty)$ with constant $\frac{1}{2\sqrt{\delta}}$, we obtain for all $t\in[-1,1]$, $t\neq 0$, 
	\begin{equation}\label{eq_domconvprereq}
		\left\lvert \frac{\sqrt{ \delta + |\nabla u + t\nabla v|^2} -  \sqrt{\delta + \lvert\nabla u\rvert^2 } }{t}\right\rvert
		\leq \frac{\left\lvert \nabla v\right\rvert\cdot\left(2\left\lvert \nabla u\right\rvert + \left\lvert\nabla v\right\rvert\right)}{2\sqrt{\delta}} \qquad\text{ a.e. in }\Omega.
		\end{equation}
	Thus, we can apply the theorem of dominated convergence, which yields 
	\begin{equation*}
		\lim_{t\to 0} \frac{\psi_\delta(u + tv) - \psi_\delta(u)}{t} 
		= \int_\Omega \lim_{t\to 0} \frac{\sqrt{ \delta + |\nabla u + t\nabla v|^2} -  \sqrt{\delta + |\nabla u|^2 } }{t} \dx
		= \int_\Omega \frac{ \left(\nabla u, \nabla v\right) }{\sqrt{\delta + |\nabla u|^2 } } \dx.
		\end{equation*}
	From
	\begin{equation*}
		\left\lvert\int_\Omega \frac{ \left(\nabla u, \nabla v\right) }{\sqrt{\delta + \lvert\nabla u\rvert^2 } } \dx\right\rvert
		\leq \int_\Omega \left\lvert\frac{ \left(\nabla u, \nabla v\right) }{\sqrt{\delta + \lvert\nabla u\rvert^2 } } \right\rvert\dx
		\leq \frac{\lVert \nabla u\rVert_{L^2(\Omega)}\lVert \nabla v\rVert_{L^2(\Omega)}}{\sqrt{\delta}}
		\leq \frac{\lVert u\rVert_{H^1(\Omega)}\lVert v\rVert_{H^1(\Omega)}}{\sqrt{\delta}}
		\end{equation*}
	we see that the functional $v\mapsto\psi_\delta^\prime(u)v$ is linear and continuous.\\
	\textbf{Second \Gateaux derivative}\\
	Let $u, v, w \in H^1(\Omega)$. Since
	$g:\R^N\rightarrow\R$, $g(y):=\frac{(y,z)}{\sqrt{\delta + \lvert y\rvert^2}}$, with $z\in\R^N$ fixed, 
	is Lipschitz continuous on $\R^N$ with constant $\frac{2}{\sqrt{\delta}}\lvert z\rvert$, we obtain for all $t\in\R$, $t\neq 0$, 
	\begin{equation}\label{eq_domconvprereq2}
		\left\lvert\frac{1}{t}\right\rvert\left\lvert\frac{ \left( \nabla u + t \nabla w, \nabla v \right) }{ \sqrt{ \delta + \lvert\nabla u + t\nabla w\rvert^2 } } - \frac{ \left(\nabla u, \nabla v\right) }{\sqrt{\delta + \lvert\nabla u\rvert^2 } } \right\rvert
		\leq \frac{2}{\sqrt{\delta}}\lvert\nabla v\rvert \lvert\nabla w\rvert \qquad\text{ a.e. in }\Omega.
		\end{equation}
	Dominated convergence yields
	\begin{equation*}
		\begin{split}
			\lim_{t\to 0} \frac{\psi_\delta^\prime(u + tw)v - \psi_\delta^\prime(u)v}{t} 
			& = \int_{\Omega}\lim_{t\to 0}\frac{1}{t}\left(\frac{ \left( \nabla u + t \nabla w, \nabla v \right) }{ \sqrt{ \delta + \lvert\nabla u + t\nabla w\rvert^2 } } - \frac{ \left(\nabla u, \nabla v\right) }{\sqrt{\delta + \lvert\nabla u\rvert^2 } } \right)\dx\\
			& = \int_\Omega \frac{ \left( \nabla v, \nabla w \right)}{ \sqrt{ \delta + \lvert\nabla u\rvert^2 } } - \frac{ \left(\nabla u, \nabla v \right) \left(\nabla u, \nabla w \right)  }{ \left( \delta + \lvert\nabla u\rvert^2 \right)^{\frac{3}{2}} } \dx,
			\end{split}
		\end{equation*}
	where we used the directional derivative of $g$ to derive the last equality. 
	From \eqref{eq_domconvprereq2} we deduce the boundedness of the bilinear mapping $(v,w)\mapsto\psi_\delta^{\prime\prime}(u)[v,w]$ by
	\begin{equation}\label{eq_lipcontoffirstderivativeofpsidelta}
		\left\lvert\int_\Omega \frac{ \left( \nabla v, \nabla w \right)}{ \sqrt{ \delta + \lvert\nabla u\rvert^2 } } - \frac{ \left(\nabla u, \nabla v \right) \left(\nabla u, \nabla w \right)  }{ \left( \delta + \lvert\nabla u\rvert^2 \right)^{\frac{3}{2}} } \dx\right\rvert
		\leq \frac{2}{\sqrt{\delta}}\lVert v\rVert_{H^1(\Omega)}\lVert w\rVert_{H^1(\Omega)}.
		\end{equation}
	\textbf{Lipschitz continuous \Frechet differentiability}\\
	Denoting by
	$\norm[{\cal B}]{\cdot}$ the standard norm for bounded bilinear forms on $H^1(\Omega)\times H^1(\Omega)$, we infer from \eqref{eq_lipcontoffirstderivativeofpsidelta} that 
	$\sup_{u\in H^1(\Omega)}\norm[{\cal B}]{\psi_\delta^{\prime\prime}(u)}\leq \frac{2}{\sqrt{\delta}}$.
	This implies that $u\mapsto\psi_\delta^{\prime}(u)$ is Lipschitz with constant
	$\frac{2}{\sqrt{\delta}}$, hence 
	$u\mapsto\psi_\delta(u)$ is \Frechet differentiable. \qed
\end{proof}

\section{H\"older continuity for quasilinear partial differential equations}\label{sec_HoeldercontinuityforquasilinPDEs}

To prove results on the H\"older continuity of solutions to quasilinear elliptic PDEs, we first discuss linear elliptic PDEs.

\begin{theorem}\label{thm:linearregularity}
	Let $\alpha \in (0,1)$ and let $\Omega$ be a bounded $C^{1,\alpha}$ domain. Let $\gamma_0, \mu>0$ be given.  
	Let $A\in C^{0,\alpha}(\Omega,\R^{N\times N})$ be a uniformly elliptic matrix with ellipticity constant $\mu$ and let $\gamma \geq \gamma_0$. Let $a^0>0$ be such that $\gamma, \lVert A \rVert_{C^{0,\alpha}(\Omega)} \leq a^0$. Then there is a constant $C>0$ depending only on $\alpha$, $\Omega$, $N$, $\mu$, $a^0$ and $\gamma_0$ such that for any $p\in L^\infty(\Omega)$ and any $f\in C^{0,\alpha}(\Omega,\R^N)$ the unique weak solution $u$ to
	\begin{equation} \label{eq:linearpde2}
	\left\{
	\begin{aligned}
	-\divg (A\nabla u) +\gamma u  & = p - \divg(f) && \text{ in }\Omega,\\
	\partial_{A_\nu} u & = 0 && \text{ on }\Gamma,
	\end{aligned}
	\right.
	\end{equation}
	satisfies $u\in C^{1,\alpha}(\Omega)$ and 
	\begin{equation*}
	\lVert u \rVert_{C^{1,\alpha}(\Omega)} \leq C \left( \lVert p \rVert_{L^\infty(\Omega)} + \lVert f \rVert_{C^{0,\alpha}(\Omega)} \right).
	\end{equation*}
\end{theorem}

\begin{proof}
	We did not find a proof of Theorem~\ref{thm:linearregularity} in the literature, so we provide it here.\\	
	A standard ellipticity argument delivers unique existence and $\lVert u \rVert_{H^1(\Omega)} \leq C \lVert p \rVert_{L^\infty(\Omega)}$, where $C$ only depends on the claimed quantities. Moreover, by \cite[Theorem 3.16(iii)]{troianello} 
	\begin{equation*}
	\lVert u \rVert_{\mathcal{L}^{2,N+2\alpha}(\Omega)} + \lVert \nabla u \rVert_{\mathcal{L}^{2,N+2\alpha}(\Omega)} \leq C \Bigl( \lVert p \rVert_{\mathcal{L}^{2,(N+2\alpha-2)^+}(\Omega)} + \lVert f \rVert_{\mathcal{L}^{2,N+2\alpha}(\Omega)} + \lVert u \rVert_{H^1(\Omega)} \Bigr).
	\end{equation*}
	Here, $C$ depends on all of the claimed quantities except $\gamma_0$, and $\mathcal{L}^{2,\lambda}(\Omega)$ denotes a Campanato space; for details see \cite[Chapter~1.4]{troianello}. The definition of Campanato spaces implies $\lVert p \rVert_{\mathcal{L}^{2,(N+2\alpha-2)^+}(\Omega)}\leq C \lVert p \rVert_{L^\infty(\Omega)}$.  Using the isomorphism between $\mathcal{L}^{2,N+2\alpha}(\Omega)$ and $C^{0,\alpha}(\Omega)$ from \cite[Theorem 1.17~(ii)]{troianello} we obtain
	\begin{equation*}
	\lVert u \rVert_{C^{1,\alpha}(\Omega)} \leq C \left( \lVert p \rVert_{L^\infty(\Omega)} + \lVert f \rVert_{C^{0,\alpha}(\Omega)} + \lVert u \rVert_{H^1(\Omega)} \right).
	\end{equation*}
	The earlier ellipticity estimate concludes the proof. \qed 
\end{proof}

The next result concerns quasilinear PDEs. It follows directly from \cite[Theorem 2]{Liebermann1988}.
\begin{theorem}\label{thm:lieberalpha}
	Let $\Omega$ be a bounded $C^{1,\alpha^\prime}$ domain for some $\alpha^\prime \in (0,1]$. Let $A: \Omega\times\R\times\R^N\rightarrow \R^{N}$, $B: \Omega\times\R\times\R^N\rightarrow\R$, $M>0$ and $0 < \lambda \leq \Lambda$. Let $\kappa, m \geq 0$ and suppose that 
	\begin{align}
	& \sum_{i,j=1}^N \partial_{\eta_i} A_j(x,u,\eta) \xi_i\xi_j \geq \lambda \bigl(\kappa + |\eta|_2^2 \bigr)^m |\xi|_2^2, &&\text{(ellipticity)} \label{eq:Lk1}\\
	& \sum_{i,j=1}^N |\partial_{\eta_i} A_j(x,u,\eta)| \leq \Lambda \bigl(\kappa + |\eta|_2\bigr)^m, &&\text{(boundedness of $A$)} \label{eq:Lk2}\\
	& |B(x,u,\eta)| \leq \Lambda \bigl(1+|\eta|_2\bigr)^{m+2}, &&\text{(boundedness of $B$)} \label{eq:Lk4}
	\end{align}
	as well as the H\"older continuity property 
	\begin{equation}\label{eq:Lk3}
		|A(x_1,u_1,\eta) - A(x_2,u_2,\eta)| \leq \Lambda\bigl(1+|\eta|_2\bigr)^{m+1} \bigl( |x_1-x_2|^{\alpha^\prime} + |u_1-u_2|^{\alpha^\prime} \bigr)
	\end{equation}
	are satisfied for all $x,x_1,x_2\in\Omega$, $u,u_1,u_2 \in [-M, M]$ and $\eta,\xi\in\R^N$. Then there exist constants $\alpha \in (0,1)$ and $C>0$ such that each solution $u\in H^1(\Omega)$ of 
	\begin{equation*}
	\int_\Omega A(x,u,\nabla u)^T \nabla \varphi \dx = \int_\Omega B(x,u,\nabla u)\varphi\dx \qquad \forall\varphi\in H^1(\Omega)
	\end{equation*}
	satisfies
	\begin{equation*}
	\lVert u \rVert_{C^{1,\alpha}(\Omega)} \leq C.
	\end{equation*}
	Here, $C>0$ only depends on $\alpha^\prime$, $\Omega$, $N$, $\Lambda/\lambda$, $m$, and $M$, while $\alpha \in (0,1)$ only depends on $\alpha^\prime$, $N$, $\Lambda/\lambda$ and $m$.
\end{theorem}

We collect elementary estimates for H\"older continuous functions.

\begin{lemma} \label{lem_hoeldercompositions}
	Let $\Omega\subset\R^N$ be nonempty, let $\alpha>0$, and let $f,g\in C^{0,\alpha}(\Omega)$. Then:
	\begin{itemize}
		\item[$\bullet$] $\lVert fg \rVert_{C^{0,\alpha}(\Omega)} \leq \lVert f \rVert_{C^{0,\alpha}(\Omega)} \lVert g \rVert_{C^{0,\alpha}(\Omega)}$.
		\item[$\bullet$] $\lVert \sqrt{\epsilon+f^2}\rVert_{C^{0,\alpha}(\Omega)} \leq 
		\sqrt{\epsilon}+\lVert f \rVert_{C^{0,\alpha}(\Omega)}$ for all $\epsilon>0$.
		\item[$\bullet$] If $\lvert f\rvert \geq \epsilon>0$ on $\Omega$ for some constant $\epsilon>0$, then there holds
		\begin{equation*}
		\left\lVert 1/f \right\rVert_{C^{0,\alpha}(\Omega)} \leq \epsilon^{-2} \lVert f \rVert_{C^{0,\alpha}(\Omega)} + \epsilon^{-1}.
		\end{equation*}
		\item[$\bullet$] $\lVert \,\lvert h \rvert \, \rVert_{C^{0,\alpha}(\Omega)}
		\leq \lVert h \rVert_{C^{0,\alpha}(\Omega,\R^N)}$
		for all $h\in C^{0,\alpha}(\Omega,\R^N)$.
		\item[$\bullet$] Let $N_i\in\N$ and let $U_i\subset\R^{N_i}$ be nonempty, $1\leq i\leq 4$.
		For $\phi\in C^{0,1}(U_2,U_3)$, $h\in C^{0,\alpha}(U_1,U_2)$ and $H\in C^{0,\alpha}(U_3,U_4)$ there hold
		\begin{equation*}
		\lVert\phi\circ h\rVert_{C^{0,\alpha}(U_1,U_3)}\leq 
		\lvert\phi\rvert_{C^{0,1}(U_2,U_3)} \lvert h\rvert_{C^{0,\alpha}(U_1,U_2)} + 
		\lVert\phi\rVert_{L^\infty(U_2,U_3)}
		\end{equation*}
		and
		\begin{equation*}
		\lVert H\circ\phi\rVert_{C^{0,\alpha}(U_2,U_4)}\leq \lvert\phi\rvert_{C^{0,1}(U_2,U_3)}^\alpha\lvert H\rvert_{C^{0,\alpha}(U_3,U_4)}+
		\lVert H\rVert_{L^\infty(U_3,U_4)}.
		\end{equation*}
	\end{itemize}
\end{lemma}

\begin{proof}
	\textbf{First claim:} Because of 
	$\lvert f(x)g(x) - f(y)g(y)\rvert 
	\leq \lvert f(x)\rvert\,\lvert g(x)-g(y)\rvert + \lvert g(y)\rvert\,\lvert f(x)-f(y)\rvert
	\leq ( \, 
	\lVert f \rVert_{L^\infty(\Omega)} \lvert g \rvert_{C^{0,\alpha}(\Omega)} + \lVert g \rVert_{L^\infty(\Omega)} \lvert f \rvert_{C^{0,\alpha}(\Omega)} \, ) |x-y|^\alpha$ for all $x,y\in\Omega$, we infer
	$\lvert f g\rvert_{C^{0,\alpha}(\Omega)} \leq 
	\lVert f \rVert_{L^\infty(\Omega)} \lvert g \rvert_{C^{0,\alpha}(\Omega)} + \lVert g \rVert_{L^\infty(\Omega)} \lvert f \rvert_{C^{0,\alpha}(\Omega)}$.
	Together with $\lVert fg\rVert_{L^\infty(\Omega)} \leq 
	\lVert f \rVert_{L^\infty(\Omega)} \lVert g \rVert_{L^\infty(\Omega)}$
	this implies the first claim.\\
	\textbf{Second claim:} 
	Since $\phi(t):=\sqrt{\epsilon + t^2}$ is Lipschitz continuous with constant $1$ in $\R$, the assertion follows from the fifth claim
	by use of
	$\lVert\sqrt{\epsilon+f^2}\rVert_{L^\infty(\Omega)}\leq \sqrt{\epsilon}+\lVert f\rVert_{L^\infty(\Omega)}$.\\ 
	\textbf{Third claim:} 
	Since $\phi(t):=|t|^{-1}$ is Lipschitz continuous with constant $\epsilon^{-2}$ in $\R\setminus(-\epsilon,\epsilon)$, the assertion follows from the fifth claim, applied with $U_2:=\{f(x):\,x\in\Omega\}$, by use of
	$\lVert\phi\rVert_{L^\infty(U_2,U_3)}=\lVert \, \lvert f\rvert^{-1}\, \rVert_{L^\infty(\Omega)}\leq\epsilon^{-1}$.\\ 
	\textbf{Fourth claim:} 
	The assertion follows from the fifth claim. 
	\\
	\textbf{Fifth claim:}
	For $x,y\in U_1$ we have
	\begin{equation*}
		\lvert\phi(h(x))-\phi(h(y))\rvert
		\leq \lvert\phi\rvert_{C^{0,1}(U_2,U_3)} \lvert h(x)-h(y)\rvert
		\leq \lvert\phi\rvert_{C^{0,1}(U_2,U_3)} \lvert h\rvert_{C^{0,\alpha}(U_1,U_2)}\lvert x-y\rvert^\alpha.
		\end{equation*}
	Together with $\lvert\phi(h(x))\rvert \leq \sup_{y\in U_2} \lvert\phi(y)\rvert = \lVert\phi\rVert_{L^\infty(U_2,U_3)}$ for all $x\in U_1$ we obtain the assertion for $\phi\circ h$. The assertion for $H\circ\phi$ can be established analogously. 
	\qed 
\end{proof}

We can now establish the desired regularity and continuity result for \eqref{eq:quasilinearpde}.

\begin{theorem} \label{thm:quasilinearls}
	Let $\Omega\subset\R^N$ be a bounded $C^{1,\alpha^\prime}$ domain for some $\alpha^\prime\in (0,1]$. Let $\beta>0$ and $\gamma^0 \geq \gamma \geq\gamma_0 > 0$ and $\delta^0\geq\delta\geq \delta_0>0$.
	By $u(p) \in H^1(\Omega)$ we denote for each $p\in L^\infty(\Omega)$ the unique weak solution of
	\begin{equation} \label{eq:thm:nonlinearpde}
	\left\{
	\begin{aligned}
	-\divg\Biggl( \Biggl[ \gamma +\frac{\beta}{\sqrt{\delta + \lvert\nabla u\rvert^2}} \Biggr] \nabla u \Biggr) + \gamma u & = p &&\text{ in }\Omega,\\
	\Biggl( \left[ \gamma+\frac{\beta}{\sqrt{\delta + \lvert\nabla u\rvert^2}} \right] \nabla u, \nu \Biggr) & = 0 &&\text{ on }\Gamma.
	\end{aligned}
	\right.
	\end{equation}
	Let $b^0>0$ be arbitrary and let $\Ball \subset L^\infty(\Omega)$ denote the open ball with radius $b_0$. There exists $\alpha\in(0,1)$ such that $u:\Ball\rightarrow C^{1,\alpha}(\Omega)$ is well-defined and Lipschitz continuous, i.e. $\lVert u(p_1)-u(p_2)\rVert_{C^{1,\alpha}(\Omega)}\leq L\lVert p_1-p_2\rVert_{L^\infty(\Omega)}$ for all $p_1,p_2\in \Ball\subset L^\infty(\Omega)$ and some $L>0$. The constants $L$ and $\alpha$ are independent of $\gamma$ and $\delta$, but may depend on $\alpha^\prime$, $\Omega$, $N$, $\beta$, $b^0$, $\gamma_0,\gamma^0,\delta_0$ and $\delta^0$.
\end{theorem}

\begin{proof}
	We did not find a proof of Theorem~\ref{thm:quasilinearls} in the literature, so we provide it here. To this end, let $b^0>0$ and let $p_1,p_2\in L^\infty(\Omega)$ with $\lVert p_1 \rVert_{L^\infty(\Omega)}, \lVert p_2 \rVert_{L^\infty(\Omega)} < b^0$.\\	
	\textbf{Part 1: Showing existence of $\mathbf{u_1,u_2\in H^1(\Omega)}$.}\\
	For $i=1,2$ we define
	\begin{equation*}
		F_i \colon \Hone \rightarrow \R, \qquad 
		F_i(v):=\gamma \norm[\Hone]{v}^2 + \beta \int_\Omega \sqrt{ \delta + |\nabla v|^2 } \dx - (p_i,v)_{L^2(\Omega)}.
	\end{equation*}
	As $F_i$ is strongly convex, 
	it has a unique minimizer $u_i \in H^1(\Omega)$. 
	Since $F_i$ is \Frechet differentiable by Lemma~\ref{thm:psiderivativeNEU}, 
	we have $F^\prime(u_i) = 0$ in $H^1(\Omega)^\ast$, 
	which is equivalent to \eqref{eq:thm:nonlinearpde}.\\
	\textbf{Part 2: Showing $\mathbf{u_1,u_2\in L^\infty(\Omega)}$ and an estimate for $\mathbf{\lVert u_1\rVert_{L^\infty(\Omega)}}$ and $\mathbf{\lVert u_2\rVert_{L^\infty(\Omega)}}$}\\
		Fix $M>\gamma_0^{-1} b^0$ and let $u_{i,M} := \min(M, \max(-M, u_i) )$,  $i=1,2$.
		For any $\N \ni q\geq 1$ we have
		\begin{equation*}
			\nabla \bigl( u_{i,M}^{2q-1} \bigr) = (2q-1) \, u_{i}^{2q-2} \, \nabla u_i \cdot 1_{\{ -M < u_i < M \}} \in L^2(\Omega).
			\end{equation*}
		Testing \eqref{eq:thm:nonlinearpde} with $u_{i,M}^{2q-1}$ yields
		\begin{equation*}
			\begin{split}
				\gamma & (u_i, u_{i,M}^{2q-1})_{L^2(\Omega)}\\ & = (p, u_{i,M}^{2q-1})_{L^2(\Omega)} - (2q-1) \int_{\{-M< u_i < M\}} \gamma  u_{i}^{2q-2} |\nabla u_i|^2 + \beta u_{i}^{2q-2} \frac{|\nabla u_i|^2}{\sqrt{\delta+|\nabla u_i|^2}}\dx \\
				& \leq \lVert p \rVert_{L^\infty(\Omega)} \lVert u_{i,M}^{2q-1} \rVert_{L^1(\Omega)} 
				\leq b^0 \lVert 1 \rVert_{L^{2q}(\Omega)} \lVert u_{i,M}^{2q-1} \rVert_{L^{\frac{2q}{2q-1}}(\Omega)} 
				= b^0 |\Omega|^{\frac{1}{2q}} \lVert u_{i,M} \rVert_{L^{2q}(\Omega)}^{2q-1}.
				\end{split}
			\end{equation*}
		In combination with 
		\begin{equation*}
			\begin{split}
				& \gamma (u_i, u_{i,M}^{2q-1})_{L^2(\Omega)}\\ & = \gamma \left( \int_{\{ -M < u_i < M \}} u_{i}^{2q} \dx + \int_{\{ u_i \leq -M \}} u_i (-M)^{2q-1} \dx + \int_{\{ M \leq u_i \}} u_i M^{2q-1} \dx \right) \\
				& \geq \gamma \left( \lVert u_i \rVert_{L^{2q}(\{ -M < u_i < M \})}^{2q} + \int_{ \{ u_i \leq -M \} } (-M)^{2q} \dx + \int_{\{ M \leq u_i \}} M^{2q} \dx \right) \geq \gamma_0 \lVert u_{i,M} \rVert_{L^{2q}(\Omega)}^{2q}
				\end{split}
			\end{equation*}
		this yields 
		$\gamma_0 \lVert u_{i,M} \rVert_{L^{2q}(\Omega)} \leq b^0 |\Omega|^{\frac{1}{2q}}$. Sending $q\to\infty$ gives $\lVert u_{i,M} \rVert_{L^\infty(\Omega)} \leq \gamma_0^{-1} b^0$. As $M>\gamma_0^{-1} b^0$ by assumption we conclude that
	\begin{equation} \label{eq:proof:uiLinfty}
		\lVert u_i \rVert_{L^\infty(\Omega)} \leq \gamma_0^{-1} b^0 \quad\text{ for }i=1,2.
	\end{equation}
	\textbf{Part 3: Obtaining $\mathbf{C^{1,\alpha}}$ regularity of $\mathbf{u_1,u_2}$}\\
	We apply Theorem~\ref{thm:lieberalpha} with $m = 0$, $A(x,u,\eta) = \gamma\eta + {\beta\eta}/{\sqrt{\delta+\lvert\eta\rvert^2}}$, $B(x,u,\eta)=p_i(x)$ for $i=1,2$, $\kappa=0$, identical values for $\alpha^\prime$, $\lambda =\gamma_0$, $\Lambda =\max\{b^0,\gamma^0 N + \delta_0^{-1/2}\beta (N+N^2)\}$ and $M = \gamma_0^{-1} b^0$, cf. \eqref{eq:proof:uiLinfty}. Since $A$ is independent of $(x,u)$ and continuously differentiable, it is easy to see that the requirements of Theorem~\ref{thm:lieberalpha} are met.
	This shows $u_1,u_2\in C^{1,\alpha}({\Omega})$ for some $\alpha>0$ and yields 
	\begin{align} \label{eq:proof:uiC1alpha}
		\lVert u_i \rVert_{C^{1,\alpha}(\Omega)} \leq C,
	\end{align}
	where $C>0$ and $\alpha\in (0,1)$ depend only on $\alpha^\prime$, $\Omega$, $N$, $\Lambda / \lambda = \gamma_0^{-1} \Lambda$ and $M = \gamma_0^{-1} b^0$.\\	
	\textbf{Part 4: Lipschitz continuity of $\mathbf{p\mapsto u(p)}$}\\
	Taking the difference of the weak formulations supplies 
	\begin{equation} \label{eq:proof:localLSofT0}
		\int_\Omega \nabla \varphi^T \Biggl( \gamma \nabla \tilde u + \beta \frac{\nabla u_1}{\sqrt{\delta + \lvert\nabla u_1\rvert^2}} - \beta \frac{\nabla u_2}{\sqrt{\delta + \lvert\nabla u_2\rvert^2}} \Biggr) + \gamma \varphi \tilde u \dx = \int_{\Omega}\varphi\tilde p\dx\;\quad\forall\varphi\in H^1(\Omega),
	\end{equation}
	where we abbreviated $\tilde u:= u_1-u_2$ and $\tilde p:= p_1-p_2$.
	The function $H: \R^N \rightarrow \R$ given by $H(v) := \sqrt{\delta + \lvert v\rvert^2}$ is convex. 
	Let $t\in [0,1]$ and denote by $u^\tau:\Omega\rightarrow\R$ the $C^{1,\alpha}(\Omega)$ function $u^\tau(x) := u_2(x) + \tau \tilde u(x)$. 
	For every $x\in\Omega$ it holds that
	\begin{equation*}
		\begin{split}
			\frac{\nabla u_1(x)}{\sqrt{\delta + \lvert\nabla u_1(x)\rvert^2}} - \frac{\nabla u_2(x)}{\sqrt{\delta + \lvert\nabla u_2(x)\rvert^2}} 
			& = \nabla H\bigl(\nabla u_1(x)\bigr) - \nabla H\bigl(\nabla u_2(x)\bigr)\\
			& = \int_0^1 \nabla^2 H(\nabla u^\tau(x)) \dtau\;\nabla \tilde u(x),
		\end{split}
	\end{equation*}
	where the integral is understood componentwise.
	Together with \eqref{eq:proof:localLSofT0} we infer that $\tilde u$ satisfies
	\begin{equation*}
		\left\{
		\begin{aligned}
			-\divg\Bigl(\tilde A \nabla \tilde u \Bigr) +\gamma \tilde u & = \tilde p &&\text{ in }\Omega,\\
			\partial_{\nu_{\tilde A}} \tilde u & = 0 &&\text{ on }\Gamma,
		\end{aligned}
		\right.
	\end{equation*}
	where $\tilde A:\Omega\rightarrow\R^{N\times N}$ is given by
	\begin{equation*}
		\tilde A(x):=\gamma I+ \beta \int_0^1 \nabla^2 H(\nabla u^\tau(x)) \dtau. 
	\end{equation*}
		In order to apply Theorem~\ref{thm:linearregularity} to this PDE, we show $\tilde A \in C^{0,\alpha}(\Omega,\R^{N\times N})$.
		The convexity of $H$ implies that $\nabla^2 H$ is positive semi-definite. 
		Thus we find for any $v\in \R^N$ and any $x\in\Omega$
		\begin{equation*}
			v^T \tilde A(x) v \geq \gamma |v|^2 \geq \gamma_0 |v|^2.
			\end{equation*}
		For $x\in\Omega$ and $1\leq i,j\leq N$ it holds that
		\begin{equation*}
			\lvert \tilde A_{ij}(x)\rvert\leq \gamma + \beta \int_0^1 \left\lvert \bigl[\nabla^2 H(\nabla u^\tau(x))\bigr]_{ij}\right\rvert\dtau
			\leq \gamma^0 + \beta \sup_{\tau\in[0,1]} \left\lVert\bigl[\nabla^2 H(\nabla u^\tau)\bigr]_{ij}\right\rVert_{L^\infty(\Omega)}.
			\end{equation*}
		We also have for all $x,y\in\Omega$
		\begin{equation*}
			\begin{split}
				\left\lvert \tilde A_{ij}(x)-\tilde A_{ij}(y)\right\rvert
				& \leq \beta \int_0^1 \left\lvert \Bigl[\nabla^2 H(\nabla u^\tau(x))-\nabla^2 H(\nabla u^\tau(y))\Bigr]_{ij}\right\rvert\dtau\\
				& \leq \beta\sup_{\tau\in[0,1]} \left\lvert \Bigl[\nabla^2 H(\nabla u^\tau(x))-\nabla^2 H(\nabla u^\tau(y))\Bigr]_{ij}\right\rvert\\
				& \leq \beta \sup_{\tau\in[0,1]} \left\lvert \Bigl[\nabla^2 H(\nabla u^\tau)\Bigr]_{ij} \right\rvert_{C^{0,\alpha}(\Omega)} \left\lvert x-y\right\rvert^\alpha,
				\end{split}
			\end{equation*}
		which shows $\lvert \tilde A_{ij}\rvert_{C^{0,\alpha}(\Omega)}\leq 
		\beta \sup_{\tau\in[0,1]} \lvert [\nabla^2 H(u^\tau)]_{ij} \rvert_{C^{0,\alpha}(\Omega)}$. Together, we infer that
		\begin{equation} \label{eq:proof:LScontinuityC1alpha}
			\bigl\lVert \tilde A_{ij}\bigr\rVert_{C^{0,\alpha}(\Omega)}
			\leq 	
			\gamma^0 + 2\beta\sup_{\tau\in[0,1]}\left\lVert \bigl[\nabla^2 H(\nabla u^\tau)\bigr]_{ij} \right\rVert_{C^{0,\alpha}(\Omega)} 
			\end{equation}
		for all $1\leq i,j\leq N$. From Lemma~\ref{lem_hoeldercompositions} we obtain for every fixed $1\leq i,j\leq N$ 
		\begin{equation*} 
			\begin{split}
				& \left\lVert \bigl[\nabla^2 H(\nabla u^\tau)\bigr]_{ij} \right\rVert_{C^{0,\alpha}(\Omega)} 
				\leq \left\lVert \frac{1}{\sqrt{\delta + \lvert \nabla u^\tau\rvert^2}} \right\rVert_{C^{0,\alpha}(\Omega)} + \enspace \left\lVert \frac{\partial_{x_i} u^\tau \partial_{x_j} u^\tau}{\sqrt{ \delta + \lvert \nabla u^\tau\rvert^2 }^3} \right\rVert_{C^{0,\alpha}(\Omega)}\\
				& \hspace{1cm} \leq C \left( 1 + \Bigl\lVert \sqrt{ \delta + \lvert \nabla u^\tau\rvert^2 } \,\Bigr\rVert_{C^{0,\alpha}(\Omega)} + \bigl\lVert \nabla u^\tau \bigr\rVert_{C^{0,\alpha}(\Omega)}^2 \Bigl\lVert \left(\delta + \lvert \nabla u^\tau\rvert^2\right)^{-\frac32} \Bigr\rVert_{C^{0,\alpha}(\Omega)} \right),
				\end{split}
			\end{equation*}
		where $C$ only depends on $\delta_0$.
		Since $\lVert \nabla u_1 \rVert_{C^{0,\alpha}(\Omega)}, \lVert \nabla u_2 \rVert_{C^{0,\alpha}(\Omega)} \leq C$ by \eqref{eq:proof:uiC1alpha}, there holds $\lVert \nabla u^\tau \rVert_{C^{0,\alpha}(\Omega)}\leq C$ with the same $C>0$. This $C$ only depends on $\alpha^\prime$, $\Omega$, $N$, $\beta$, $b^0$, $\gamma_0$, $\gamma^0$ and $\delta_0$.
		This and Lemma~\ref{lem_hoeldercompositions} show	
		\begin{equation*}
			\begin{split}
				\Bigl\lVert \bigl[\nabla^2 H(\nabla u^\tau)\bigr]_{ij} \Bigr\rVert_{C^{0,\alpha}(\Omega)} 
				& \leq C \Bigl( 1 + \Bigl\lVert \sqrt{ \delta + \lvert \nabla u^\tau\rvert^2 } \,\Bigr\rVert_{C^{0,\alpha}(\Omega)} + \Bigl\lVert \sqrt{ \delta + \lvert \nabla u^\tau\rvert^2 } \,\Bigr\rVert_{C^{0,\alpha}(\Omega)}^3 \Bigr)\\
				& \leq C \Bigl( 1 + \Bigl(\sqrt{\delta}+\left\lVert \nabla u^\tau \right\rVert_{C^{0,\alpha}(\Omega)} \Bigr)^3\Bigr)
				\leq C,
				\end{split}
			\end{equation*}
		where $C>0$ is independent of $\tau$ and only depends on the quantities stated in the theorem. Hence, with the same $C$ there holds 
		\begin{equation*}
			\sup_{\tau\in[0,1]}
			\left\lVert\bigl[\nabla^2 H(\nabla u^\tau)\bigr]_{ij} \right\rVert_{C^{0,\alpha}(\Omega)} \leq C
			\qquad\forall \, 1\leq i,j\leq N.
			\end{equation*}
		Inserting this into
		\eqref{eq:proof:LScontinuityC1alpha} yields $\tilde A \in C^{0,\alpha}(\Omega,\R^{N\times N})$ with $\lVert \tilde A \rVert_{C^{0,\alpha}(\Omega)} \leq \gamma^0+2\beta C$,
	This implies that Theorem~\ref{thm:linearregularity} is applicable,
	which yields $\lVert \tilde u \rVert_{C^{1,\alpha}(\Omega)} \leq C \lVert \tilde p \rVert_{L^\infty(\Omega)}$
	with a constant $C$ that only depends on the claimed quantities. This proves the asserted Lipschitz continuity. \qed 
\end{proof}

\section{The original problem: Optimality conditions}\label{sec_optcondorigprob}

The first order optimality conditions of \eqref{eq:redProblem} can be obtained by use of \cite{holler}.
The space $W_0^q(\divg;\Omega)$, $q\in[1,\infty)$, that appears in the following is defined in \cite[Definition~10]{holler}.
\begin{theorem}\label{thm:firstorderoptcondoriginalproblem}
	Let $\Omega\subset\R^N$, $N\in\{1,2,3\}$, be a bounded Lipschitz domain and let $r_N=\frac{N}{N-1}$ if $N>1$, respectively, $r_N\in[1,\infty)$ if $N=1$. Then we have: 
	The function $\bar u\in \BV$ is the solution of \eqref{eq:redProblem} iff there is
	\begin{equation*}
	\bar h \in L^\infty(\Omega,\R^N)\cap W_0^{r_N}(\divg;\Omega)
	\end{equation*}
	that satisfies $\lVert\lvert\bar h\rvert\rVert_{L^\infty(\Omega)} \leq \beta$ and $\operatorname{div} \bar h= \bar p$, where $\bar p$ is defined as in section~\ref{sec:reducedproblem}, as well as
	\begin{equation*}
	\begin{aligned}
	\bar h & = \beta \frac{\nabla \bar u_a}{|\nabla \bar u_a|} && ~~~~\mathcal{L}^N \text{-a.e. in } \Omega \setminus \left\{ x: \nabla \bar u_a(x) = 0 \right\}, \\
	T\bar h & = \beta \frac{\bar u^+(x) - \bar u^-(x)}{|\bar u^+(x) - \bar u^-(x)|} \nu_{\bar u} && ~~~~ \mathcal{H}^1 \text{-a.e. in } J_{\bar u},\\
	T\bar h &  = \beta \sigma_{C_{\bar u}} &&  ~~~~ |\nabla \bar u_c| \text{-a.e.}
	\end{aligned}
	\end{equation*}
	Here, the first, second and third equation correspond to the absolutely continuous part, the jump part, respectively, the 
	Cantor part of the vector measure $\nabla \bar u$. 
	Also, $\sigma_{C_{\bar u}}$ is the Radon-Nikodym density of $\nabla \bar u_c$ with respect to $|\nabla \bar u_c|$, 
	cf. e.g. \cite[Theorem 9.1]{brokatemassengl}. Moreover, $\nu_{\bar u}$ is the jump direction of $\bar u$ and $J_{\bar u}$ denotes the discontinuity set of $\bar u$ in the sense of \cite[Definition 3.63]{ambrosio}. Further, $\mathcal{H}^1$ is the Hausdorff measure of $J_{\bar u}$, cf. e.g. \cite[Chapter~4]{Attouch}.
	The operator $T \colon \operatorname{dom}(T) \subset W^{\operatorname{div}, q}(\Omega) \cap L^\infty(\Omega,\R^N) \rightarrow L^1(\Omega,\R^N, |\nabla u|)$ is called the \emph{full trace operator} and is introduced in \cite[Definition 12]{holler}. We emphasize that $\bar h\in \operatorname{dom}(T)$.
\end{theorem}

\begin{proof}
	The well-known optimality condition $0\in\partial j(\bar u)$ from convex analysis can be expressed as
	$-\frac{\bar p}{\beta} \in \partial\lvert \bar u\rvert_{\BV}$, so the claim follows from \cite[Proposition~8]{holler}. \qed 
\end{proof}	

\begin{remark} \label{rem_sparsityg}
	Theorem~\ref{thm:firstorderoptcondoriginalproblem} 
	implies the sparsity relation 
	$\{x: \nabla \bar u_a(x)\neq 0\}\subset \{x: |\bar h(x)|=\beta\}$.
	Since $\{x: |\bar h(x)|=\beta\}$ typically has small Lebesgue measure (often: measure $0$), $\bar u$ is usually constant a.e. in large parts (often: all) of $\Omega$; cf. also the example in section~\ref{sec_examplewithexplicitsolution}.
\end{remark}

\section{An example with explicit solution}\label{sec_examplewithexplicitsolution}

Using rotational symmetry we construct an example for \eqref{eq:ocintro} for $N=2$ with an explicit solution.
We let $\calA=-\Delta$ and $c_0\equiv 0$ in the governing PDE. 
We define
$\hat h: [0,\infty) \rightarrow \R$, $\hat h(r):=\frac{\beta}{2} (\cos(\frac{2\pi}{R}r)-1)$ and $\Omega := B_{2R}(0) \setminus \overline{B_R(0)}$, where the parameters $R>0$ and $\beta>0$ are arbitrary.
We introduce the functions
\begin{equation*}
r(x,y):=\sqrt{x^2+y^2},\qquad
\bar h(x,y):=\hat h(r(x,y))\nabla r(x,y)
\quad\text{ and }\quad
\bar u(x,y):= 1_{(R,\frac{3R}{2})}(r(x,y)),
\end{equation*}
all of which are defined on $\Omega$. The problem data is given by 
\begin{equation*}
\bar p := \operatorname{div} \bar h,
\qquad
\bar y := S\bar u
\qquad\text{ and }\qquad
y_\Omega := \Delta\bar p + \bar y.
\end{equation*}
We now show that these quantities satisfy the properties of Theorem~\ref{thm:firstorderoptcondoriginalproblem}. 
By construction $\bar y$ and $\bar p$ are the state and adjoint state associated to $\bar u$ and we have $\bar p = \operatorname{div} \bar h$. We check the properties of $\bar h$. Since $\lvert\nabla r\rvert = 1$ for $(x,y)\in\Omega$,
we obtain $|\bar h(x,y)| = |\hat h(r(x,y))| \leq \frac{\beta}{2} 2 = \beta$. We also see that $\bar h$ is $C^1$ in $\bar\Omega$ and satisfies $\bar h = 0$ on $\partial\Omega$ so that $\bar h \in L^\infty(\Omega,\R^N)\cap W_0^q(\divg;\Omega)$ for any $q\in[1,\infty)$. 
As $\nabla \bar u(x,y) = -\nabla r(x,y) \mathcal{H}^1_{\partial B_{\frac{3R}{2}}(0)}(x,y)$, we find that $\nabla \bar u$ has no Cantor part and no parts that are absolutely continuous with respect to the Lebesgue measure. Thus, the first and third condition on $\bar h$ in Theorem~\ref{thm:firstorderoptcondoriginalproblem} are trivially satisfied. For $(x,y) \in \partial B_{\frac{3R}{2}}(0)=J_{\bar u}$ we have $\bar h(x,y) =-\beta\nabla r(x,y)=-\beta \nu_{\bar u}$ and $\bar u^+(x)=0$, $\bar u^-(x)=1$ for $x\in J_{\bar u}$, hence the second condition on $\bar h$ in Theorem~\ref{thm:firstorderoptcondoriginalproblem} holds. 
Let us confirm that $\bar p$ satisfies the homogeneous Dirichlet boundary conditions. From $\Delta r = r^{-1}$ and $|\nabla r|^2 = 1$ we obtain
\begin{equation*}
\bar p = \operatorname{div} \bar h = \nabla \hat h(r)^T\nabla r + \hat h(r) \Delta r = \hat h^\prime(r) \lvert\nabla r\rvert^2 + r^{-1} \hat h(r) = \hat h^\prime(r) + r^{-1} \hat h(r).
\end{equation*}
Thus, $\bar p$ satisfies the boundary conditions.
Let us confirm that $\bar y$ satisfies the boundary conditions.
The ansatz $\bar y(x,y) = \hat y(r(x,y))$,
with $\hat y:\Omega\rightarrow\R$ to be determined, 
yields
\begin{equation*}
- 1_{(R, 3R/2)}(r) = - \bar u(x,y) = \Delta \bar y(x,y) = \operatorname{div} (\hat y^\prime(r) \nabla r) = \hat y^{\prime\prime}(r) + r^{-1} \hat y^\prime(r).
\end{equation*}
This leads to
\begin{equation*}
\hat y(r) = \begin{cases}
- \frac{r^2}{4} + A\ln(r/(2R)) + B & \text{ if } r\in (R,3R/2),\\
C \ln (r/(2R)) & \text{ if } r\in (3R/2,2R),
\end{cases}
\end{equation*}
and it is straightforward to check that 
$\bar y$ satisfies the boundary conditions and is continuously differentiable for the parameters
\begin{equation*}
A 
= \frac{R^2}{8} \cdot \frac{18\ln(3/4)-5}{\ln(1/4)},
\qquad
B 
= \frac{9R^2}{8} \left( \frac12 - \ln(3/4) \right)
\qquad\text{and}\qquad
C = \frac{R^2}{8} \cdot\frac{18\ln(3/2) - 5}{\ln(1/4)}.
\end{equation*}
All in all, the optimality conditions of Theorem~\ref{thm:firstorderoptcondoriginalproblem} are satisfied.
Moreover, the optimal value in this example is given by 
\begin{equation*}
j(\bar u)=\frac12\lVert \bar y - y_\Omega\rVert_{\LLL}^2 + \beta\lvert\bar u\rvert_{\BV}
= \frac12\lVert \Delta\bar p \rVert_{\LLL}^2 + \beta\lvert\bar u\rvert_{\BV},
\end{equation*}
which for $R=2\pi$ results in 
\begin{equation*}
j(\bar u) = \frac{\beta^2 \pi}{4} \left( 3 \pi^2 + \ln(8) + \frac{15}{4} \Ci(2\pi) - \frac{27}{4} \Ci(4\pi) + 3 \Ci(8\pi) \right) + 6 \pi^2 \beta
\approx 24.85 \beta^2 + 59.22 \beta
\end{equation*}
with $\Ci(t):=-\int_t^\infty \frac{\cos\tau}{\tau}\dtau$.

\begin{acknowledgements}
Dominik Hafemeyer acknowledges support from the graduate program TopMath of the Elite Network of Bavaria and the TopMath Graduate Center of TUM Graduate School at Technische Universit\"at München. He received a scholarship from the Studienstiftung des deutschen Volkes. He receives support from the IGDK Munich-Graz. 
Funded by the Deutsche Forschungsgemeinschaft, grant no 188264188/GRK1754.
\end{acknowledgements}

%
%

\bibliographystyle{spmpsci}      
\bibliography{references}        

\begin{thebibliography}{10}
\providecommand{\url}[1]{{#1}}
\providecommand{\urlprefix}{URL }
\expandafter\ifx\csname urlstyle\endcsname\relax
  \providecommand{\doi}[1]{DOI~\discretionary{}{}{}#1}\else
  \providecommand{\doi}{DOI~\discretionary{}{}{}\begingroup
  \urlstyle{rm}\Url}\fi

\bibitem{acar}
{Acar}, R., {Vogel}, C.R.: {Analysis of bounded variation penalty methods for
  ill-posed problems.}
\newblock {Inverse Probl.} \textbf{10}(6), 1217--1229 (1994).
\newblock \doi{10.1088/0266-5611/10/6/003}

\bibitem{AllendesFuicaOtarola}
Allendes, A., Fuica, F., Otárola, E.: {Adaptive finite element methods for
  sparse PDE-constrained optimization}.
\newblock IMA Journal of Numerical Analysis \textbf{40}(3), 2106--2142 (2019).
\newblock \doi{10.1093/imanum/drz025}

\bibitem{AlnaesBlechta2015a}
Aln{\ae}s, M.S., Blechta, J., Hake, J., Johansson, A., Kehlet, B., Logg, A.,
  Richardson, C., Ring, J., Rognes, M.E., Wells, G.N.: The {FEniCS} project
  version 1.5.
\newblock Archive of Numerical Software \textbf{3}(100), 9--23 (2015).
\newblock \doi{10.11588/ans.2015.100.20553}

\bibitem{ambrosio}
Ambrosio, L., Fusco, N., Pallara, D.: Functions of bounded variation and free
  discontinuity problems.
\newblock Oxford Mathematical Monographs. The Clarendon Press, Oxford
  University Press (2000)

\bibitem{Attouch}
Attouch, H., Buttazzo, G., Michaille, G.: {Variational analysis in Sobolev and
  BV spaces. Applications to PDEs and optimization. 2nd revised ed.},
  \emph{MPS/SIAM Series on Optimization}, vol.~6.
\newblock SIAM (2014).
\newblock \doi{10.1137/1.9781611973488}

\bibitem{Bartels2012}
Bartels, S.: Total variation minimization with finite elements: convergence and
  iterative solution.
\newblock SIAM J. Numer. Anal. \textbf{50}(3), 1162--1180 (2012).
\newblock \doi{10.1137/11083277X}

\bibitem{BergouniouxBonnefondHaberkornPrivat}
{Bergounioux}, M., {Bonnefond}, X., {Haberkorn}, T., {Privat}, Y.: {An optimal
  control problem in photoacoustic tomography.}
\newblock {Math. Models Methods Appl. Sci.} \textbf{24}(12), 2525--2548 (2014).
\newblock \doi{10.1142/S0218202514500286}

\bibitem{BoydVandenberghe}
{Boyd}, S., {Vandenberghe}, L.: {Convex optimization.}
\newblock Cambridge University Press (2004).
\newblock \urlprefix\url{http://www.stanford.edu/~boyd/cvxbook/bv_cvxbook.pdf}

\bibitem{holler}
Bredies, K., Holler, M.: A pointwise characterization of the subdifferential of
  the total variation functional  (2012).
\newblock Preprint, IGDK1754

\bibitem{brokatemassengl}
{Brokate}, M., {Kersting}, G.: {Measure and integral}.
\newblock Cham: Birkh\"auser/Springer (2015).
\newblock \doi{10.1007/978-3-319-15365-0}

\bibitem{CCK2012}
{Casas}, E., {Clason}, C., {Kunisch}, K.: {Approximation of elliptic control
  problems in measure spaces with sparse solutions.}
\newblock {SIAM J. Control Optim.} \textbf{50}(4), 1735--1752 (2012).
\newblock \doi{10.1137/110843216}

\bibitem{CCK2013}
Casas, E., Clason, C., Kunisch, K.: Parabolic control problems in measure
  spaces with sparse solutions.
\newblock SIAM J. Control Optim. \textbf{51}(1), 28--63 (2013)

\bibitem{CasasKogutLeugering}
{Casas}, E., {Kogut}, P.I., {Leugering}, G.: {Approximation of optimal control
  problems in the coefficient for the \(p\)-Laplace equation. I: Convergence
  result.}
\newblock {SIAM J. Control Optim.} \textbf{54}(3), 1406--1422 (2016).
\newblock \doi{10.1137/15M1028108}

\bibitem{kruse}
{Casas}, E., {Kruse}, F., {Kunisch}, K.: {Optimal control of semilinear
  parabolic equations by BV-functions.}
\newblock {SIAM J. Control Optim.} \textbf{55}(3), 1752--1788 (2017).
\newblock \doi{10.1137/16M1056511}

\bibitem{CasasKunisch2014}
{Casas}, E., {Kunisch}, K.: {Optimal control of semilinear elliptic equations
  in measure spaces.}
\newblock {SIAM J. Control Optim.} \textbf{52}(1), 339--364 (2014).
\newblock \doi{10.1137/13092188X}

\bibitem{Casas2019}
Casas, E., Kunisch, K.: Analysis of optimal control problems of semilinear
  elliptic equations by {BV}-functions.
\newblock Set-Valued Var. Anal. \textbf{27}(2), 355--379 (2019).
\newblock \doi{10.1007/s11228-018-0482-7}

\bibitem{CasasKunischPola}
{Casas}, E., {Kunisch}, K., {Pola}, C.: {Some applications of BV functions in
  optimal control and calculus of variations.}
\newblock {ESAIM, Proc.} \textbf{4}, 83--96 (1998).
\newblock \doi{10.1051/proc:1998022}

\bibitem{CasasKunischPola2}
Casas, E., Kunisch, K., Pola, C.: Regularization by functions of bounded
  variation and applications to image enhancement.
\newblock Appl. Math. Optim. \textbf{40}(2), 229--257 (1999).
\newblock \doi{10.1007/s002459900124}

\bibitem{CasasRyllTroeltzsch}
{Casas}, E., {Ryll}, C., {Tr\"oltzsch}, F.: {Sparse optimal control of the
  Schl\"ogl and Fitzhugh-Nagumo systems.}
\newblock {Comput. Methods Appl. Math.} \textbf{13}(4), 415--442 (2013).
\newblock \doi{10.1515/cmam-2013-0016}

\bibitem{CasasVexlerZuazua}
{Casas}, E., {Vexler}, B., {Zuazua}, E.: {Sparse initial data identification
  for parabolic PDE and its finite element approximations.}
\newblock {Math. Control Relat. Fields} \textbf{5}(3), 377--399 (2015).
\newblock \doi{10.3934/mcrf.2015.5.377}

\bibitem{chan}
Chan, T.F., Zhou, H.M., Chan, R.H.: {Continuation method for total variation
  denoising problems}.
\newblock In: F.T. Luk (ed.) Advanced Signal Processing Algorithms, vol. 2563,
  pp. 314--325. International Society for Optics and Photonics, SPIE (1995).
\newblock \doi{10.1117/12.211408}

\bibitem{Clason2018}
Clason, C., Kruse, F., Kunisch, K.: Total variation regularization of
  multi-material topology optimization.
\newblock ESAIM Math. Model. Numer. Anal. \textbf{52}(1), 275--303 (2018).
\newblock \doi{10.1051/m2an/2017061}

\bibitem{Clason2011}
Clason, C., Kunisch, K.: A duality-based approach to elliptic control problems
  in non-reflexive {B}anach spaces.
\newblock ESAIM Control Optim. Calc. Var. \textbf{17}(1), 243--266 (2011).
\newblock \doi{10.1051/cocv/2010003}

\bibitem{ElvetunNielsen}
{Elvetun}, O.L., {Nielsen}, B.F.: {The split Bregman algorithm applied to
  PDE-constrained optimization problems with total variation regularization}.
\newblock {Comput. Optim. Appl.} \textbf{64}(3), 699--724 (2016).
\newblock \doi{10.1007/s10589-016-9823-3}

\bibitem{EngelKunischBVWaveSemismooth}
{Engel}, S., {Kunisch}, K.: {Optimal control of the linear wave equation by
  time-depending BV-controls: A semi-smooth Newton approach}.
\newblock {Math. Control Relat. Fields} \textbf{10}(3), 591--622 (2020).
\newblock \doi{10.3934/mcrf.2020012}

\bibitem{EngelVexlerTrautmann}
Engel, S., Vexler, B., Trautmann, P.: {Optimal finite element error estimates
  for an optimal control problem governed by the wave equation with controls of
  bounded variation}.
\newblock IMA Journal of Numerical Analysis  (2020).
\newblock \doi{10.1093/imanum/draa032}

\bibitem{Ern2004}
Ern, A., Guermond, J.L.: Theory and practice of finite elements, \emph{Applied
  Mathematical Sciences}, vol. 159.
\newblock Springer (2004).
\newblock \doi{10.1007/978-1-4757-4355-5}

\bibitem{grisvard}
{Grisvard}, P.: {Elliptic problems in nonsmooth domains}, vol.~69, reprint of
  the 1985 hardback edn.
\newblock SIAM (2011).
\newblock \doi{10.1137/1.9781611972030}

\bibitem{Hafemeyer15}
{Hafemeyer}, D.: {Optimale Steuerung von Differentialgleichungen mit
  BV-Funktionen}.
\newblock Bachelor's thesis, Technische Universität München, Munich (2016)

\bibitem{HafemeyerMaster}
{Hafemeyer}, D.: Regularization and discretization of a {BV}-controlled
  elliptic problem: A completely adaptive approach.
\newblock Master's thesis, Technische Universität München, Munich (2017)

\bibitem{Hafemeyer2020}
{Hafemeyer}, D.: Optimal control of parabolic obstacle problems -- optimality
  conditions and numerical analysis.
\newblock Phd thesis, Technische Universität München, Munich (2020).
\newblock
  \urlprefix\url{http://nbn-resolving.de/urn/resolver.pl?urn:nbn:de:bvb:91-diss-20200508-1524287-1-1}

\bibitem{HafeMan_reg}
Hafemeyer, D., Mannel, F.: {A path-following inexact Newton method for optimal
  control in BV}.
\newblock {Comput. Optim. Appl., accepted for publication}  (2022)

\bibitem{Hafemeyer2019}
Hafemeyer, D., Mannel, F., Neitzel, I., Vexler, B.: {Finite element error
  estimates for one-dimensional elliptic optimal control by BV functions}.
\newblock {Math. Control Relat. Fields} \textbf{10}(2), 333--363 (2020).
\newblock \doi{10.3934/mcrf.2019041}

\bibitem{HSW}
Herzog, R., Stadler, G., Wachsmuth, G.: Directional sparsity in optimal control
  of partial differential equations.
\newblock SIAM J. Control Optim. \textbf{50}(2), 943--963 (2012).
\newblock \doi{10.1137/100815037}

\bibitem{HintermKunisch}
{Hinterm\"uller}, M., {Kunisch}, K.: {Total bounded variation regularization as
  a bilaterally constrained optimization problem}.
\newblock {SIAM J. Appl. Math.} \textbf{64}(4), 1311--1333 (2004).
\newblock \doi{10.1137/S0036139903422784}

\bibitem{HintermStadler}
{Hinterm\"uller}, M., {Stadler}, G.: {An infeasible primal-dual algorithm for
  total bounded variation--based Inf-convolution-type image restoration}.
\newblock {SIAM J. Sci. Comput.} \textbf{28}(1), 1--23 (2006).
\newblock \doi{10.1137/040613263}

\bibitem{Hinze05}
{Hinze}, M.: A variational discretization concept in control constrained
  optimization: The linear-quadratic case.
\newblock Comput. Optim. Appl. \textbf{30}(1), 45--61 (2005).
\newblock \doi{10.1007/s10589-005-4559-5}

\bibitem{HinzeTroe}
{Hinze}, M., {Tr\"oltzsch}, F.: {Discrete concepts versus error analysis in
  PDE-constrained optimization}.
\newblock {GAMM-Mitt.} \textbf{33}(2), 148--162 (2010).
\newblock \doi{10.1002/gamm.201010012}

\bibitem{Kato1966}
Kato, T.: Perturbation theory for linear operators.
\newblock Die Grundlehren der mathematischen Wissenschaften, Band 132. Springer
  (1966).
\newblock \doi{10.1007/978-3-662-12678-3}

\bibitem{Kelley_Itmethodseq}
{Kelley}, C.T.: {Iterative methods for linear and nonlinear equations},
  vol.~16.
\newblock SIAM (1995).
\newblock \doi{10.1137/1.9781611970944}

\bibitem{LiStadler}
{Li}, C., {Stadler}, G.: {Sparse solutions in optimal control of PDEs with
  uncertain parameters: the linear case.}
\newblock {SIAM J. Control Optim.} \textbf{57}(1), 633--658 (2019).
\newblock \doi{10.1137/18M1181419}

\bibitem{LiFukushima2000}
{Li}, D., {Fukushima}, M.: {A derivative-free line search and global
  convergence of Broyden-like method for nonlinear equations}.
\newblock {Optim. Methods Softw.} \textbf{13}(3), 181--201 (2000).
\newblock \doi{10.1080/10556780008805782}

\bibitem{PreconditioningTVRegularization}
{Li}, H., {Wang}, C., {Zhao}, D.: Preconditioning for pde-constrained
  optimization with total variation regularization.
\newblock Appl. Math. Comput. \textbf{386}, 125470 (2020).
\newblock \doi{10.1016/j.amc.2020.125470}

\bibitem{Liebermann1988}
Lieberman, G.M.: Boundary regularity for solutions of degenerate elliptic
  equations.
\newblock Nonlinear Anal. \textbf{12}(11), 1203--1219 (1988).
\newblock \doi{10.1016/0362-546X(88)90053-3}

\bibitem{LoggMardalEtAl2012a}
Logg, A., Mardal, K.A., Wells, G.N., et~al.: Automated Solution of Differential
  Equations by the Finite Element Method.
\newblock Springer (2012).
\newblock \doi{10.1007/978-3-642-23099-8}

\bibitem{LoggWells2010a}
Logg, A., Wells, G.N.: Dolfin: Automated finite element computing.
\newblock ACM Transactions on Mathematical Software \textbf{37}(2) (2010).
\newblock \doi{10.1145/1731022.1731030}

\bibitem{LoggWellsEtAl2012a}
Logg, A., Wells, G.N., Hake, J.: DOLFIN: a C++/Python Finite Element Library,
  chap.~10.
\newblock Springer (2012).
\newblock \doi{10.1007/978-3-642-23099-8_10}

\bibitem{NeitzelPruefertSlawig}
{Neitzel}, I., {Pr\"ufert}, U., {Slawig}, T.: {Strategies for time-dependent
  PDE control with inequality constraints using an integrated modeling and
  simulation environment}.
\newblock {Numer. Algorithms} \textbf{50}(3), 241--269 (2009).
\newblock \doi{10.1007/s11075-008-9225-4}

\bibitem{PieperDiss}
{Pieper}, K.: Finite element discretization and efficient numerical solution of
  elliptic and parabolic sparse control problems.
\newblock Phd thesis, Technische Universität München, Munich (2015).
\newblock
  \urlprefix\url{https://nbn-resolving.de/urn/resolver.pl?nbn:de:bvb:91-diss-20150420-1241413-1-4}

\bibitem{Rudin1992}
{Rudin}, L.I., {Osher}, S., {Fatemi}, E.: Nonlinear total variation based noise
  removal algorithms.
\newblock Physica D \textbf{60}, 259--268 (1992).
\newblock \doi{10.1016/0167-2789(92)90242-F}

\bibitem{Rudin1987}
Rudin, W.: Real and complex analysis, third edn.
\newblock McGraw-Hill Book Co. (1987)

\bibitem{Savare98}
{Savar\'e}, G.: {Regularity results for elliptic equations in Lipschitz
  domains.}
\newblock {J. Funct. Anal.} \textbf{152}(1), 176--201 (1998).
\newblock \doi{10.1006/jfan.1997.3158}

\bibitem{Schiela_IPMefficient}
{Schiela}, A.: {An interior point method in function space for the efficient
  solution of state constrained optimal control problems}.
\newblock {Math. Program.} \textbf{138}(1-2 (A)), 83--114 (2013).
\newblock \doi{10.1007/s10107-012-0595-y}

\bibitem{Stadler}
{Stadler}, G.: {Elliptic optimal control problems with \(L^1\)-control cost and
  applications for the placement of control devices.}
\newblock {Comput. Optim. Appl.} \textbf{44}(2), 159--181 (2009).
\newblock \doi{10.1007/s10589-007-9150-9}

\bibitem{troianello}
Troianiello, G.M.: Elliptic differential equations and obstacle problems.
\newblock The University Series in Mathematics. Plenum Press, New York (1987).
\newblock \doi{10.1007/978-1-4899-3614-1}

\bibitem{WeiserGaenzlerSchiela}
{Weiser}, M., {G\"anzler}, T., {Schiela}, A.: {A control reduced primal
  interior point method for a class of control constrained optimal control
  problems}.
\newblock {Comput. Optim. Appl.} \textbf{41}(1), 127--145 (2008).
\newblock \doi{10.1007/s10589-007-9088-y}

\end{thebibliography}

\end{document}